\documentclass[11pt, reqno]{amsart}

\usepackage{amsfonts}
\usepackage{amssymb, amsmath, amsthm, enumitem}
\usepackage[colorlinks=true, linkcolor = blue, citecolor= Green]{hyperref}
\usepackage[alphabetic,lite]{amsrefs}
\usepackage{amscd}   % for commutative diagrams
\usepackage{fullpage}
\usepackage[all]{xy} % for complicated commutative diagrams
\usepackage{faktor}
\usepackage[utf8x]{inputenc}
\usepackage{verbatim}
\usepackage{mathrsfs}
\usepackage{stmaryrd}
\usepackage[new]{old-arrows}
\DeclareMathAlphabet{\mathpzc}{OT1}{pzc}{m}{it}
\DeclareFontEncoding{OT2}{}{} % to enable usage of cyrillic fonts
\newcommand{\textcyr}[1]{%
 {\fontencoding{OT2}\fontfamily{wncyr}\fontseries{m}\fontshape{n}\selectfont #1}}
\newcommand{\Sha}{{\mbox{\textcyr{Sh}}}}
\usepackage[usenames,dvipsnames]{color}
\newtheorem{lemma}{Lemma}[section]
\newtheorem{theorem}[lemma]{Theorem}
\newtheorem{proposition}[lemma]{Proposition}

\newtheorem{corollary}[lemma]{Corollary}
\newtheorem{cor}[lemma]{Corollary}

\newtheorem{claim*}{Claim}

\newtheorem{definition}[lemma]{Definition}

\theoremstyle{definition}
\newtheorem{remark}[lemma]{Remark}

\newcommand{\C}{{\mathbb C}}
\newcommand{\F}{{\mathbb F}}
\newcommand{\Q}{{\mathbb Q}}
\newcommand{\R}{{\mathbb R}}
\newcommand{\Z}{{\mathbb Z}}
\newcommand{\NN}{{\mathbb N}}

\newcommand{\calC}{{\mathcal C}}
\newcommand{\calD}{{\mathcal D}}

\newcommand{\calF}{{\mathcal F}}
\newcommand{\calG}{{\mathcal G}}

\newcommand{\calO}{{\mathcal O}}

\newcommand{\calS}{{\mathcal S}}
\newcommand{\calT}{{\mathcal T}}

\newcommand{\frakh}{{\mathfrak h}}

\newcommand{\frakp}{{\mathfrak p}}
\newcommand{\frakq}{{\mathfrak q}}

\newcommand{\frakP}{{\mathfrak P}}

\usepackage[mathscr]{euscript}

\DeclareMathOperator{\Frob}{Frob}
\DeclareMathOperator{\coker}{coker}

\DeclareMathOperator{\Char}{char}

\DeclareMathOperator{\im}{im}

\DeclareMathOperator{\Hom}{Hom}

\DeclareMathOperator{\Aut}{Aut}
\DeclareMathOperator{\Gal}{Gal}
\DeclareMathOperator{\Ind}{Ind}

\DeclareMathOperator{\Res}{Res}

\DeclareMathOperator{\Cl}{Cl}

\DeclareMathOperator{\ord}{ord}

\DeclareMathOperator{\Proj}{Proj}

\DeclareMathOperator{\tor}{tor}

\DeclareMathOperator{\N}{N}

\DeclareMathOperator{\id}{id}

\DeclareMathOperator{\B}{{\mbox{\textcyr{B}}}}
\DeclareMathOperator{\Mod}{Mod}

\DeclareMathOperator{\ad}{ad}

\DeclareMathOperator{\Prob}{Prob}
\newcommand{\hH}{\widehat{H}}

\numberwithin{equation}{section}
\numberwithin{table}{section}

\usepackage{tikz-cd}

\title{Presentations of Galois groups of maximal extensions with restricted ramification}
\author{Yuan Liu}
\address{Department of Mathematics\\
University of Illinois Urbana-Champaign \\ 1409 W Green St \\ Urbana, IL 61801
 USA}  
\email{yyyliu@illinois.edu}

\begin{document}

	%%%%%%%%%%%%%%%%%%%%%%%%%%%%%%%%%%%%%%%%%%%%%%%%%%%%%%%%%%%%%%%%%%%%%%%%%%%%
	\begin{abstract}
		Motivated by the work of Lubotzky, we use Galois cohomology to study the difference between the number of generators and the minimal number of relations in a presentation of $G_S(k)$, the Galois group of the maximal extension of a global field $k$ that is unramified outside a finite set $S$ of places, as $k$ varies among a certain family of extensions of a fixed global field $Q$. 
 		We prove a generalized version of the global Euler--Poincar\'{e} Characteristic
		and define a group $\B_S(k,A)$, for each finite simple $G_S(k)$-module $A$, to generalize the work of Koch and Shafarevich on the pro-$\ell$ completion of $G_S(k)$.
		We prove that $G_S(k)$ always admit a balanced presentation.
		In the setting of the non-abelian Cohen--Lenstra heuristics, we prove that the unramified Galois groups studied by the Liu--Wood--Zureick-Brown conjecture always admit a balanced presentation in the form of the random group in the conjecture. 
	\end{abstract}
	%%%%%%%%%%%%%%%%%%%%%%%%%%%%%%%%%%%%%%%%%%%%%%%%%%%%%%%%%%%%%%%%%%%%%%%%%%%%

	\maketitle
%\yuan{To-do \& to-check:
%\begin{enumerate}
%	\item Do I actually need $A$ to be a finite module? It seems like Tate duality theorem requires the finiteness of $A$.
%	\item notation: when we say ``generate'', we mean ``topologically generate''.
%\end{enumerate}
%}

\section{Introduction}

	For a global field $k$ and a set $S$ of primes of $k$, we denote by $G_S(k)$ the Galois group of the maximal extension of $k$ that is unramified outside $S$. 
	Determining whether $G_{\O}(k)$ is finitely generated and finitely presented is a long-existing open question. It is well known by class field theory that the abelianization of $G_{\O}(k)$ is finitely presented and, in particular, is finite when $k$ is a number field. Golod and Shafarevich \cite{Golod-Shafarevich} constructed the first infinite $\ell$-class tower group of a number field, where the $\ell$-class tower group of $k$ is the pro-$\ell$ completion of $G_{\O}(k)$ for a prime number $\ell$. 
	The minimal numbers of generators and relations, which are called the generator rank and relator rank, in presentations of a pro-$\ell$ group is determined by its group cohomology with coefficient $\F_{\ell}$. Using this idea, Koch and Shafarevich \cite{Koch} employed Galois cohomology to give estimations of the generator rank and relator rank of the pro-$\ell$ completion of $G_S(k)$ when $S$ is finite and $\ell \neq \Char(k)$; and in particular, in such cases, the pro-$\ell$ completion of $G_S(k)$ is always finitely presented.
	
	Recently the development on the non-abelian Cohen--Lenstra program pushes us to study canonical quotients of $G_{\O}(k)$ beyond the pro-$\ell$ completion. 
	Let $\Gamma$ be a finite group, $Q$ the global field $\Q$ or $\F_q(t)$, and $\mu(Q)$ the group of roots of unity of $Q$. For a Galois extension $k/Q$ with $\Gal(k/Q) \simeq \Gamma$, define $k^{\#}$ to be the maximal unramified extension of $k$, that is split completely at places of $k$ over $\infty$ and of order relatively prime to $|\mu(Q)||\Gamma|$ and $\Char(Q)$ (if non-zero).  
	In work \cite{LWZB} of Wood, Zureick-Brown and the author, we constructed random group models to make conjectures on the distributions for some families of canonical quotients $\Gal(k^{\#}/k)$ of $G_{\O}(k)$ as $k$ varies among all $\Gamma$-extensions of $Q$ split completely at $\infty$. Because $\Gal(k^{\#}/k)$ has (supernatural) order prime to $|\Gamma|$, a homomorphic split of $\Gal(k^{\#}/Q) \twoheadrightarrow \Gal(k/Q)$ defines by conjugation a continuous $\Gamma$ action on $\Gal(k^{\#}/k)$; and this action satisfies some property which we called \emph{admissible} (see Definition~\ref{def:admissible}). The set of all isomorphism classes of all admissible profinite $\Gamma$-groups are closed under taking $\Gamma$-equivariant quotients, and we can construct \emph{the free admissible profinite $\Gamma$-group $\calF_n(\Gamma)$ on $n$ generators} (see Section~\ref{sect:pres-ad} for its definition).
	For a profinite $\Gamma$-group $G$ and a finite set of $\calC$ of isomorphism classes of finite $\Gamma$-groups, let $G^{\calC}$ denote the \emph{pro-$\calC$ completion} of $G$ with respect to the variety of groups generated by $\calC$ (see Section~\ref{sect:pro-C} for definition).
	The work \cite{LWZB} uses quotients of $\calF_n(\Gamma)$ as $n \to \infty$ to construct a random group model; this model together with the conjectures implies a surprising phenomenon of the structure of $\Gal(k^\#/k)$ that was not known before: for any finite set $\calC$ of finite $\Gamma$-groups, the followings happen to $k$ with probability 1
	\begin{enumerate}
		\item\label{item:phen-1} the pro-$\calC$ completion $\Gal(k^\#/k)^{\calC}$ is a finite group, and moreover,
		\item\label{item:phen-2} there exists some finite integer $n$ such that $\Gal(k^\#/k)^{\calC}$ can be presented as the quotient of $\calF_n(\Gamma)^{\calC}$ by $[r^{-1} \gamma(r)]_{r \in X, \, \gamma \in \Gamma}$ for some subset $X$ of $\calF_n(\Gamma)^{\calC}$ of cardinality $n+1$, where the symbol $[r^{-1} \gamma(r)]_{r \in X, \, \gamma \in \Gamma}$ denotes the $\Gamma$-closed normal subgroup of $\calF_n(\Gamma)^{\calC}$ generated by $r^{-1} \gamma(r)$ for all $r \in X$ and $\gamma \in \Gamma$.
	\end{enumerate}
	The statement in \eqref{item:phen-2} implies that the deficiency (i.e., the difference between the minimal number of generators and the minimal number of relations) of $\Gal(k^{\#}/k)$ has an upper bound depending only on the order of $\Gamma$, if $\Gal(k^{\#}/k)$ is finitely generated. In this paper, we prove that both of \eqref{item:phen-1} and \eqref{item:phen-2} hold for all $\Gamma$-extensions $k/Q$ split completely above $\infty$, which strongly supports that the random group model in \cite{LWZB} is the right object to study.
	
	\begin{theorem}\label{thm:main}
	Let $\Gamma$ be a nontrivial finite group and $Q$ be either $\Q$ or $\F_q(t)$ with $q$ relatively prime to $|\Gamma|$. Let $\calC$ be a finite set of isomorphism classes of finite $\Gamma$-groups all of whose orders are prime to $|\mu(Q)||\Gamma|$ and $\Char(Q)$ (if non-zero). Then for a Galois extension $k/Q$ with Galois group $\Gamma$ that is split completely over $\infty$, we have the following isomorphism of $\Gamma$-groups ($\Gamma$ acts on the left-hand side via $\Gamma\simeq \Gal(k/Q)$)
	\begin{equation}\label{eq:main}
		\Gal(k^\#/k)^{\calC} \simeq \faktor{\calF_n(\Gamma)^{\calC}}{[r^{-1}\gamma(r)]_{r\in X , \gamma \in \Gamma}} 
	\end{equation}
	for some positive integer $n$ and some set $X$ consisting of $n+1$ elements of $\calF_n(\Gamma)^{\calC}$.
\end{theorem}

	Let $G_{\O, \infty}(k)$ denote the Galois group of the maximal unramified extension of $k$ that is completely split at every place above $\infty$, and note that with the assumptions in Theorem~\ref{thm:main} one have $\Gal(k^{\#}/k)^{\calC}=G_{\O, \infty}(k)^{\calC}$. The method we develop in this paper in fact works for $G_S(k)^{\calC}$ for any finite set $S$ of primes of $k$ and any global base field $Q$, so it can be used to study the presentation of Galois groups with restricted ramification. 
	In the case that $k$ is a function field and $\Gamma=1$ (so $k=Q$), building on the theorem of Lubotzky \cite{Lubotzky}, Shusterman \cite{Shusterman} showed that $G_{\O}(k)$ admits a finite presentation in which the number of relations is exactly the same as the number of generators (such a presentation is called a balanced presentation). Note that, in \cite{Shusterman}, the fact that $G_{\O}(k)$ is finitely generated follows by the Grothendieck's result of the geometric fundamental group of a smooth projective curve defined over a finite field, but when $k$ is a number field, whether $G_{\O}(k)$ is finitely generated or not is unknown. We prove an analogous result regarding the number field case.

	\begin{theorem}\label{thm:fin-pres}
		Let $k$ be a number field and $S$ a finite set of places of $k$. If $G_S(k)$ is topologically generated by $n$ elements, then it admits a finite presentation on $n$ generators and $[k:\Q]+n$ relations.  
	\end{theorem}

	We also apply our methods to the situations that are not considered in Theorem~\ref{thm:main}. We study the presentation of the pro-$\ell$ completion of $G_{\O, \infty}(k)$ for a Galois $\Gamma$-extension $k/Q$ in two exceptional cases: 
	\begin{enumerate}[label=(\roman*)]
		\item\label{item:exp-1} $Q$ is a number field not containing the $\ell$-th roots of unity and we do not make any assumptions on the ramification of $\infty$ in $k$ (Section~\ref{ss:other-sgn}); and 
		\item\label{item:exp-2} $Q$ is a global field containing the $\ell$-th roots of unity (Section~\ref{ss:rootsof1}). 
	\end{enumerate}
	When considering the $\ell$-parts of class groups, it is known for a long time that the Cohen--Lenstra heuristics need to be corrected in these two cases (see \cite{Cohen-Martinet, Malle2010}). 
	In each of these two cases, we use our method to compute an upper bound for the deficiency of $G_{\O, \infty}(k)$ at the pro-$\ell$ level, and then show why the conjecture of Liu–Wood–Zureick-Brown doesn’t work in these two exceptional cases. This computation of deficiencies also provides insights of how the random group model should be modified in these two cases.

\subsection{Method of the proof}

	The bulk of this paper is devoted to establishing the techniques for proving Theorem~\ref{thm:main}. Motivated by the work of Lubotzky \cite{Lubotzky}, we first translate the question to understanding the Galois cohomology groups. In Section~\ref{sect:pre-Gamma}, we construct the \emph{free profinite $\Gamma$-group $F_n(\Gamma)$ on $n$ generators}, and, for a finitely generated profinite $\Gamma$-group $G$, we study the minimal number of relations of a presentation defined by a $\Gamma$-equivariant surjection $\pi: F_n(\Gamma) \twoheadrightarrow G$. The minimal number of relations is closely related to the multiplicities of the finite irreducible $G\rtimes \Gamma$-modules appearing as quotients of $\ker(\pi)$ (Definition~\ref{def:mult-general}). In Lemma~\ref{lem:d-cohom}, we show that for such a module $A$ with $\gcd(|A|, |\Gamma|)=1$, the multiplicity of $A$ can be computed by a formula involving $\dim H^2(G \rtimes \Gamma, A)-\dim H^1(G\rtimes \Gamma, A)$. So when restricted to the category of profinite $\Gamma$-groups whose order is prime to $|\Gamma|$, by using these multiplicites, we obtain formulas for the minimal number of relations of the presentation $F_n'(\Gamma)\twoheadrightarrow G'$, where $F_n'(\Gamma)$ and $G'$ are the maximal pro-prime-to-$|\Gamma|$ quotients of $F_n(\Gamma)$ and $G$ respectively (Propositions~\ref{prop:d'-cohom} and \ref{prop:min-gen}). In particular, the formulas provide an upper bound for the minimal number of relations of this presentation using $\dim H^2(G, A)^{\Gamma}-\dim H^1(G, A)^{\Gamma}$, where $\Gamma$ acts on the cohomolgy groups by conjugation. These upper bound formulas set up the strategy of the proof of Theorem~\ref{thm:main}. Building upon it, we explore the multiplicities of \emph{admissible presentations} $\calF_n(\Gamma) \twoheadrightarrow G$ in Section~\ref{sect:pres-ad} and the multiplicities of \emph{pro-$\calC$ presentations} in Section~\ref{sect:pro-C}, where we obtain formulas that will be directly applied to the proof of Theorem~\ref{thm:main}. Then in Section~\ref{sect:height}, we define \emph{the height of a group} and show in Proposition~\ref{prop:C-finite-h} that there is an upper bound for the heights of pro-$\calC$ groups (not necessarily finitely generated) when $\calC$ is a finite set. Then Theorem~\ref{thm:finite-GC} proves the finiteness of $G_S(k)^{\calC}$ when $S$ is a finite set of primes of $k$ and $\calC$ is a finite set of finite groups, which confirms the phenomenon \eqref{item:phen-1}.

	Therefore, in order to prove Theorem~\ref{thm:main}, we need to deal with the Galois cohomology groups. In a more general setting, assuming that $Q$ is an arbitrary global field, that $k/Q$ is a Galois extension with $\Gal(k/Q)\simeq \Gamma$, and that $S$ is a finite set of primes of $k$, we want to understand 
	\begin{equation}\label{eq:deficient-cohom}
		\delta_{k/Q,S}(A):=\dim_{\F_\ell} H^2(G_S(k), A)^{\Gamma} - \dim_{\F_\ell} H^1(G_S(k), A)^{\Gamma},
	\end{equation}
	for all prime integers $\ell$ relatively prime to $|\Gamma|$ and $\Char(Q)$, and for all finite simple $\F_{\ell}[\Gal(k_S/Q)]$-modules $A$.
	In \eqref{eq:deficient-cohom}, the set $S$ needs to contain enough primes to ensure that $k_S/Q$ is Galois (see the definition of the \emph{$k/Q$-closed} sets in Section~\ref{sect:notation}), and the $\Gamma$ action on the cohomology groups is defined via the conjugation by $\Gal(k/Q)$. In Section~\ref{sect:Euler-Poincare}, we prove a generalized version of the Global Euler-Poincar\'{e} Characteristic formula (Theorem~\ref{thm:EPChar}), from which we can compute $\delta_{k/Q,S}$ when $S$ is nonempty and contains the primes above $\infty$ and $\ell$ if $Q$ is a number field. The proof basically follows the original proof of the Global Euler-Poincar\'{e} Characteristic formula, but taking the $\Gamma$ actions into account creates many technical difficulties. 
	%To deal with the case when $S$ does not satisfies the assumptions in Theorem~\ref{thm:EPChar}, we first define a group $\B_S(k, A)$, with which we can follow the Koch's argument for the $p$-class tower groups. 

	In the work of Koch, when dealing with the case that $A=\F_\ell$ and $S$ does not satisfy the assumptions in Theorem~\ref{thm:EPChar}, the abelian group $\B_S(k)$ plays an important role in the computation of $\dim_{\F_\ell} H^i(G_S(k), \F_{\ell})$ for $i=1,2$, and is defined to be the Pontryagin dual of the Kummer group
	\[
		V_S(k) = \ker \left( k^{\times}/k^{\times \ell} \longrightarrow \prod_{\frakp \in S} k_{\frakp}^{\times} /k_{\frakp}^{\times \ell} \times  \prod_{\frakp \not \in S} k_{\frakp}^{\times} / U_{\frakp} k_{\frakp}^{\times \ell} \right),
	\]
	where $k_{\frakp}$ is the completion of $k$ at $\frakp$ and $U_{\frakp}$ is the group of units of $k_{\frakp}$. In Definition~\ref{def:RB}, we define a group $\B_S(k, A)$ in a cohomological way as 
	\[
		\coker\left( \prod_{\frakp \in S} H^1(k_{\frakp}, A) \times \prod_{\frakp \not \in S} H_{nr}^1 (k_{\frakp}, A) \to H^1(k, A')^{\vee}\right),
	\]
	in order to generalize Koch's work to compute $\delta_{k/Q, S}(A)$ by replacing the trivial module $\F_\ell$ with an arbitrary finite simple module $A$. The definition of $\B_S(k,A)$ agrees with the one of $\B_S(k)$ when $A=\F_\ell$ (Proposition~\ref{prop:RB-trivial-mod}). However, Koch's argument does not directly apply to $\B_S(k,A)$, because the Hasse principle does not always hold for an arbitrary Galois module $A$ (that is, the Shafarevich group $\Sha^1(k,A)$ might be nontrivial). 
	In Section~\ref{sect:RussianB}, we modify Koch's work to overcome this obstacle, and show that most properties of $\B_S(k)$ also hold for $\B_S(k,A)$. In particular, one example, clearly showing that the failure of the Hasse principle makes a difference, is that there is a natural embedding $\Sha^2_S(k, A) \hookrightarrow \B_S(k,A)$ for $A=\F_\ell$ but not for arbitrary $A$ (Proposition~\ref{prop:ShainB} and Remark~\ref{rmk:ShainB}). In Section~\ref{sect:compute-delta}, we explicitly compute $\delta_{k/Q, S}(A)$ for all $S$ by applying the results from Sections~\ref{sect:Euler-Poincare} and \ref{sect:RussianB}, and then we prove Theorem~\ref{thm:fin-pres}. In Section~\ref{sect:proof-main}, we give the proof of Theorem~\ref{thm:main}.
	Finally, in Section~\ref{sect:exceptional}, we apply our methods to the exceptional cases \ref{item:exp-1} and \ref{item:exp-2} of Theorem~\ref{thm:main}.
	
\subsection{Previous works}
	
		For an odd prime $\ell$, the Cohen--Lenstra heuristics \cite{Cohen-Lenstra} give predictions of the distribution of $\ell$-primary parts of the class groups $\Cl(k)$ as $k$ varies over quadratic number fields. Later, Friedman and Washington \cite{Friedman-Washington} formulated an analogous conjecture for global function fields. The probability measure used for the conjectural distributions in the Cohen--Lenstra heuristics matches the one defined by the random abelian group
	\begin{equation}\label{eq:abelian-random}
		\lim_{n\to \infty} \Z_{\ell}^{\oplus n} / n+u \text{ random relations}
	\end{equation}
	where the random relations are taken with respect to the Haar measure, and $u$ is chosen to be $0$ and $1$ respectively when $k$ varies among imaginary quadratic fields and real quadratic fields. Ellenberg and Venkatesh \cite{Ellenberg-Venkatesh-ICM} theoretically explained the random group model \eqref{eq:abelian-random} and the value of $u$, by viewing $\Cl(k)$ as the cokernel of the map sending the $S$-units of $k$ to the group of fractional ideals of $k$ generated by $S$ with $S$ running along an ascending sequence of finite sets of primes of $k$. Boston, Bush and Hajir \cite{BBH-imaginary, BBH-real} extended the Cohen--Lenstra heuristics to a non-abelian setting considering the distribution of $\ell$-class tower groups (for odd $\ell$). In their work, the probability measure in the heuristics is defined by a random pro-$\ell$ group generalizing \eqref{eq:abelian-random}, and the value of $u$ (which is the deficiency in this setting) is obtained by applying Koch's argument. Notably, the moment versions of the function field analogs of the Cohen--Lenstra heuristics and the Boston--Bush--Hajir heuristics are both proven, see \cite{EVW, Boston-Wood}. In the work \cite{LWZB} of Wood, Zureick-Brown and the author, we constructed the random $\Gamma$-group
	\begin{equation}\label{eq:LWZB-rgm}
		\lim_{n\to \infty} \faktor{\calF_n(\Gamma)}{[r^{-1} \gamma(r)]_{r\in X, \gamma\in \Gamma}},
	\end{equation}
	where $X$ is a set of $n+u$ random elements of $\calF_n(\Gamma)$. 
	We showed that the moment proven in the function field case matches the moment of the probability measure defined by \eqref{eq:LWZB-rgm} exactly when $u=1$. With this evidence, we conjectured that the random group \eqref{eq:LWZB-rgm} gives the distribution of $\Gal(k^{\#}/k)$ in both the function field case and the number field case. Theorem~\ref{thm:main} explains the theoretical reason behind $u=1$ in the Liu--Wood--Zureick-Brown conjecture. 
	
		Regarding the exceptional case \ref{item:exp-1}, Cohen and Martinet \cite{Cohen-Martinet} provided a modification for the case that $Q=\Q$ and $k/Q$ varies among $\Gamma$-extensions with a fixed signature. Wang and Wood \cite{Wang-Wood19} proved some results about the probability measures described in the Cohen--Martinet heuristics. From these works, one can see that the decomposition subgroup $\Gamma_\infty$ at $\infty$ of $k/\Q$ crucially affects the probability measures. In Lemma~\ref{lem:other-sgn-mult}, we explicitly compute the upper bounds of multiplicities in a pro-$\ell$ admissible $\Gamma$-presentation of $G_{\O}(k)(\ell)$, which shows how the multiplicities are determined by $\Gamma_\infty$. Then in Corollary~\ref{cor:imag-quad} and Remark~\ref{rmk:imag-quad}, we prove that, when $k/\Q$ is an imaginary quadratic field, $G_{\O}(k)(\ell)$ can be achieved by a random group model which defines a probability measure agreeing with the Boston--Bush--Hajir heuristics. 
		
	For the exceptional case \ref{item:exp-2}, when the base field $Q$ contains the $\ell$-th roots of unity, we give upper bounds for multiplicities in Lemma~\ref{lem:rootsof1} and Corollary~\ref{cor:mult-roots}, which suggests that the distributions of $G_{\O, \infty}(k)(\ell)$ should be different between the function field case and the number field case (Remark~\ref{rmk:roots}\eqref{item:roots-2}). This difference is not surprising, as Malle observed in \cite{Malle2010} that his conjecture regarding the class groups of number fields does not easily match the result for function fields. So the upper bounds obtained in Corollary~\ref{cor:mult-roots} supports Malle's observation. The phenomenon related to the presence of the roots of unity has been numerically computed in \cite{Malle2008, Malle2010}, and the random matrices in this setting and their relation with function field counting has been studied in \cite{Katz-Sarnak1999, Achter2006, Achter2008, Garton, Adam-Malle}. A correction for roots of unity, provided with empirical evidence, is presented in \cite{Wood-nonab}.

\subsection{Other applications and further questions}

	We expect that the techniques established in this paper will have many interesting and important applications. For example, the author applies the results in this paper to the following work. In \cite{Liu-rou}, the author studies the exceptional case \ref{item:exp-2}, where the moment conjecture in the number field case is inspired by the computation of $\delta_{k/Q, \O}(A)$ similar to Section~\ref{ss:rootsof1}.
	In \cite{Liu-pgp}, the author uses the abelian group $\B_S(k,A)$ to study the embedding problems with restricted ramification, which will be crucial for the forthcoming work on the generalized Cohen--Lenstra--Martinet--Gerth conjectures.

	There are many further questions we would like to understand. First, the techniques in this paper work for any finite set $S$ of primes. So we would like to ask whether the random group models (in the abelian, pro-$\ell$ and pro-$\calC$ versions) can also be applied to predict the distributions of $G_S(k)$ as $k/Q$ varies among certain families of $\Gamma$-extensions.
	Secondly, the group $\B_S(k,A)$, which is the generalization of $\B_S(k)$ that we construct in Section~\ref{sect:RussianB}, has its own interest, because it could be applied to extend our knowledge of $G_S(k)$ from the pro-$\ell$ completion to the whole group, and moreover, it bounds the Shafarevich group via (see Proposition~\ref{prop:ShainB})
	\begin{equation}\label{eq:ShaB}
		\#\Sha^2_S(k,A) \leq \# \B_S(k,A).
	\end{equation}
	We emphasis here that understanding when $\#\Sha_{\O}^2(k,A)=\#\B_{\O}(k)$ holds can help us determine whether our upper bound of multiplicities is sharp or not (see how the inequality \eqref{eq:ShaB} is used in the proof of Proposition~\ref{prop:nf-delta}). %When $A=\F_{\ell}$, Hajir, Maire and Ramakrishna \cite{Hajir-Maire-Ram} studied the embedding $\Sha_S^2(k,\F_\ell) \hookrightarrow \B_S(k)$ when $S_{\ell}(k) \not\subset S$.
	Last but not least, the techniques established in Sections~\ref{sect:pre-Gamma}, \ref{sect:pres-ad} and \ref{sect:pro-C}, which use group cohomology to understand the presentation of a $\Gamma$-group, are purely group theoretical and independent of the number theory background, so we hope that they could have other interesting applications.
	
	In this paper, we only study the maximal prime-to-$|\Gamma|$ quotient of $G_{\O,\infty}(k)$ for a Galois $\Gamma$-extension $k/Q$, and one can see that this ``prime-to-$|\Gamma|$'' requirement is necessary in almost every crucial step. We would like to ask if the ideas of this paper can be generalized to the $|\Gamma|$-part of $G_{\O,\infty}(k)$ too.

\subsection*{Acknowledgements} 
I would like to thank Nigel Boston and Melanie Matchett Wood for helpful conversations and encouragement which inspire the author to work on the questions studied in this paper. I also thank Nigel Boston, Yufan Luo, Mark Shusterman, Preston Wake and Melanie Matchett Wood for comments on and corrections to an early draft of the paper. I am partially supported by NSF Grant DMS-2200541.

%%%%%%%%%%%%%%%%%%%%%%%%%%%%%%%%%%%%%%%%%%%%%%%%%%%%%%%%%%%%%%%%%%%%%%%%%%%%%%%%%%%%

\section{Notation and Preliminary}\label{sect:notation}

\subsection{Profinite groups and modules}

	In this paper, groups are always profinite groups and subgroups are always closed subgroups. For a group $G$, a \emph{$G$-group} is a group with a continuous $G$ action.
	If $x_1, \cdots$ are elements of a $G$-group $H$, we write $[x_1, \cdots]$ for the closed normal $G$-subgroup of $H$ topologically generated by $x_1, \cdots$. If $H$ is a $G$-group, then we write $H\rtimes G$ for the semidirect product induced by the $G$ action on $H$, and its multiplication rule is given by $(h_1, g_1)(h_2,g_2)=(h_1 g_1(h_2),g_1g_2)$ for $h_1, h_2 \in H$ and $g_1, g_2 \in G$.
	\emph{Morphisms of $G$-groups} are $G$-equivariant group homomorphisms. 
	We write $\simeq_G$ to represent isomorphism of $G$-groups, write $\Hom_G$ to represent the set of $G$-equivariant homomophisms, and define $G$-subgroup and $G$-quotient accordingly. 
	We say that a set of elements of a $G$-group $H$ \emph{$G$-generates} $H$ if $H$ is the smallest closed $G$-subgroup containing this set.
	We say that $H$ is an \emph{irreducible} $G$-group if it is a nontrivial $G$-group and has no proper, nontrivial $G$-subgroups. For a positive integer $n$, a \emph{pro-$n'$ group} is a group such that every finite quotient has order relatively prime to $n$. The \emph{pro-$n'$ completion of $G$} is the inverse limit of all pro-$n'$ quotients of $G$. 
	For a prime $\ell$, we denote the pro-$\ell$ and the pro-$\ell'$ completions of $G$ by $G(\ell)$ and $G(\ell')$ respectively.

	For a group $G$ and a commutative ring $R$, we denote by $R[G]$ the completed $R$-group ring of $G$.
	We use the following notation of $G$-modules
	\begin{eqnarray*}
		\Mod(G)&=&\text{ the category of isomorphism classes of finite $G$-modules,}\\
		\Mod(R[G])&=&\text{ the category of isomorphism classes of finite $R[G]$-modules, and}\\
		\Mod_{n} (G) &=&\text{ the category of isomorphism classes of finite $\Z/n\Z[G]$-modules.}
	\end{eqnarray*}
	For a prime integer $\ell$ and a finite $\F_{\ell}[G]$-module $A$, we define $h_G(A)$ to be the $\F_{\ell}$-dimension of $\Hom_G(A,A)$.
	We consider the \emph{Grothendieck group $K'_0(R[G])$}, which is the abelian group generated by the set $\{[A]\mid A\in \Mod(R[G])\}$ and the relations
	\[
		[A]-[B]+[C]=0
	\]
	arising from each exact sequence $0\to A\to B\to C \to 0$ of modules in $\Mod(R[G])$. For $A, B\in \Mod(R[G])$, the tensor product $A\otimes_{R} B$ endowed with the diagonal action of $G$ is an element of $\Mod(R[G])$. Then $K'_0(R[G])$ becomes a ring by linear extensions of the product $[A][B]=[A\otimes_{R} B]$. If $H$ is a subgroup of $G$, then the action of taking induced modules $\Ind_G^H$ defines a map from $K'_0(R[H])$ to $K'_0(R[G])$, which we will also denote by $\Ind_G^H$.
	
	Let $\ell$ denote a prime integer. If $H$ is a pro-$\ell'$ subgroup of $G$, then it follows by the Schur--Zassenhaus theorem that $H^1(H, A)=0$ for any $A\in \Mod_{\ell}(G)$, and hence taking the $H$-invariants is an exact functor on $\Mod_{\ell}(G)$. Moreover, when $G$ is a pro-$\ell'$ group, $\Mod_\ell(G)$ is the free abelian group generated by the isomorphism classes of finite simple $\F_{\ell}[G]$-modules, and elements $[A]$ and $[B]$ of $K_0'(\F_{\ell}[G])$ are equal if and only if $A$ and $B$ are isomorphic as $\F_\ell[G]$-modules.
	For an abelian group $A$, we let $A^{\vee}$ denote the Pontryagin dual of $A$.

\subsection{Galois groups and Galois cohomology}

For a field $k$, we write $\overline{k}$ for a fixed choice of separable closure of $k$, and write $G_k$ for the absolute Galois group $\Gal(\overline{k}/k)$. For a finite $G_k$-module $A$, we denote $A'=\Hom(A, \overline{k}^{\times})$. Let $k/Q$ be a finite Galois extension of global fields. When $v$ is a prime of the field $Q$, we define $S_{v}(k)$ to be the set of all primes of $k$ lying above $v$. Note that the function field $\F_q(t)$ has an infinite place defined by the valuation $|\cdot|_{\infty}:=q^{\deg(\cdot)}$, but this $\infty$ place is nonarchemedean. 
We define $S_{\infty}(k)$ to be the set of all archemedean places of $k$, so it is the empty set if $k$ is a function field. For a number field $k$, we let $S_{\R}(k)$ and $S_{\C}(k)$ denote the set of all real archimedean places and the set of all imaginary archimedean places of $k$ respectively.
We let $G_{\O, \infty}(k)$ denote the Galois group of the maximal unramfied extension of $k$ that is split completely at every prime above $\infty$.
So if $k$ is a number field, then $G_{\O, \infty}(k)$ is $G_{\O}(k)$. If $k$ is a funtion field, then $G_{\O, \infty}(k)$ is the quotient of $G_{\O}(k)$ by the decomposition subgroups of $k$ at primes above $\infty$.

Let $S$ be a set of places of $k$. We let $k_S$ denote the maximal extension of $k$ that is unramified outside $S$, and denote $\Gal(k_S/k)$ by $G_S(k)$ or just $G_S$ when the choice of $k$ is clear.
The set $S$ is called \emph{$k/Q$-closed} if $S_{v}(k)$ either is contained in $S$ or intersects emptily with $S$ for any prime $v$ of $Q$. When $S$ is $k/Q$-closed, it is not hard to check by Galois theory that $k_S$ is Galois over $Q$, and hence each element of $\Gal(k/Q)$ defines an outer automorphism of $G_S(k)$.
We denote
\[
	\NN(S)=\{ n \in \NN \mid n \in \calO_{k, S}^{\times}\},
\]
where $\calO_{k,S}^{\times}$ is the ring of $S$-integers of $k$. Explicitly, if $k$ is a number field, then $\NN(S)$ consists of the natural numbers such that $\ord_\frakp(n)=0$ for all $\frakp \not\in S$; and if $k$ is a function field, then $\NN(S)$ is the set of all natural numbers prime to $\Char(k)$.
For a group $G$, we define
	\[
		\Mod_S(G)=\text{the category of finite $G$-modules whose order is in $\NN(S)$}.
	\]
	In particular, if $Q$ is a function field, then $\Mod_S(G)$ consists of modules of order prime to $\Char(Q)$.
	
	Let $k$ be a global field, and $\frakp$ a prime of $k$. The completion of $k$ at $\frakp$ is denoted by $k_{\frakp}$, and the absolute Galois group and its inertia subgroup of $k_{\frakp}$ are denoted by $\calG_{\frakp}(k)$ and $\calT_{\frakp}(k)$ respectively. When the choice of $k$ is clear, we denote $\calG_{\frakp}(k)$ and $\calT_{\frakp}(k)$ by $\calG_{\frakp}$ and $\calT_{\frakp}$. Let $k/Q$ be a Galois extension of global fields. For a prime $v$ of $Q$ and a prime $\frakp \in S_v(k)$, the Galois group of $k_{\frakp}/Q_v$, denoted by $\Gal_\frakp(k/Q)$, is the decomposition subgroup of $\Gal(k/Q)$ at $\frakp$. The subgroups $\Gal_{\frakp}(k/Q)$ are conjugate to each other in $\Gal(k/Q)$ for all $\frakp \in S_v(k)$, so we write $\Gal_v(k/Q)$ for a chosen representative of this conjugacy class. 
	%Sometimes, we also write $k_v$ to represent $k_{\frakp}$ when the choice of $\frakp$ is not important.
	For a group $G$ and an $A\in\Mod(G)$, we write $H^i(G,A)$ and $\widehat{H}^i(G, A)$ for the group cohomology and the Tate cohomology respectively. For a field $k$, we denote $H^i(k, A):=H^i(G_k, A)$ and $\widehat{H}^i(k,A):=\widehat{H}^i(G_k, A)$. Let $A$ be a module in $\Mod(G_Q)$. The Galois group $\Gal(k/Q)$ acts on $H^i(k, A)$ by conjugation.
	The conjugation map commutes with inflations, restrictions, cup products and connecting homomorphisms in a long exact sequence, and hence it is naturally compatible with spectral sequences and duality theorems used in the paper. For a prime $v$ of $Q$, we consider the $\Gal(k/Q)$ action on $\oplus_{\frakp \in S_v(k)} H^i(k_{\frakp}, A)$ defined by the action on  $\oplus_{\frakp \in S_v(k)} H^i(k_{\frakp}, \Res^{G_Q}_{\calG_v(Q)}A)$. In other words, $\Gal(k/Q)$ acts on $\oplus_{\frakp \in S_v(k)} H^i(k_{\frakp}, A)$ by the permutation action on $S_v(k)$ and by the $\Gal_{\frakp}(k/Q)$-conjugation on each summand. We similarly define the $\Gal(k/Q)$ action on the product when each of the local summand is $H^i(\calT_{\frakp}, A)$ or the unramified cohomology group $H^i_{nr}(k_{\frakp}, A):=\ker \left( H^i(k_{\frakp}, A) \overset{\text{res}}{\to} H^i(\calT_{\frakp}, A)\right)$.
	In particular, the product of restriction maps for $v$
	\[
		H^i(k, A) \to \bigoplus_{\frakp \in S_v(k)} H^i(k_{\frakp}, A)
	\]
	respects the $\Gal(k/Q)$ actions. Moreover, one can check that 
	\[
		\bigoplus_{\frakp \in S_v(k)}H^i(k_{\frakp}, A) \cong \Ind_{\Gal(k/Q)}^{\Gal_{\frakq}(k/Q)} H^i(k_{\frakq}, A)
	\]
	as $\Gal(k/Q)$-modules for any $\frakq\in S_v(k)$. The same statement holds for the Tate cohomology groups. For a set $S$ of places of $k$, we use the following notation for Shafarevich groups
	\begin{eqnarray*}
		\Sha^i(k, A) &=& \ker \left(H^i(k,A) \to \prod_{\frakp \text{ all places}} H^i(k_\frakp, A)\right) \text{  and}\\
		\Sha^i_S(k,A) &=& \ker \left(H^i(G_S(k), A) \to \prod_{\frakp \in S} H^i(k_{\frakp}, A)\right),
	\end{eqnarray*}
	and we denote
	\[
		\prod'_{\frakp \in S} H^1(k_{\frakp}, A):= \left\{ (f_{\frakp})_{\frakp \in S} \in \prod_{\frakp \in S} H^1(k_{\frakp}, A)\, \bigg{|}\, f_{\frakp} \text{ is unramified for all but finitely many primes in }S\right\}.
	\]

\section{Presentations of finitely generated profinite $\Gamma$-groups}\label{sect:pre-Gamma}

	A \emph{free profinite $\Gamma$-group on $n$ generators}, denoted by $F_n(\Gamma)$, is defined to be the free profinite group on $\{x_{i,\gamma} \mid i=1, \cdots, n \text{ and } \gamma\in \Gamma\}$, where $\sigma\in \Gamma$ acts on $F_n(\Gamma)$ by
	\[
		\sigma(x_{i,\gamma})=x_{i, \sigma \gamma}.
	\]
	In other words, $F_n(\Gamma)$ is the largest $\Gamma$-group that can be topologically generated by $n$ elements under the $\Gamma$ action. When the choice of $\Gamma$ is clear, we will denote $F_n(\Gamma)$ simply by $F_n$.  
	Let $G$ be a finitely generated $\Gamma$-group. Then when $n$ is sufficiently large, there exists a short exact sequence 
	\begin{equation}\label{eq:G-pre}
		1 \longrightarrow N \longrightarrow F_n\rtimes \Gamma \overset{\pi}{\longrightarrow} G \rtimes \Gamma \longrightarrow 1
	\end{equation}
	where $\pi$ is defined by mapping $\Gamma$ identically to $\Gamma$, and $\{x_{i,1_{\Gamma}}\}_{i=1}^n$ to a set of $n$ elements of $G$ that generates $G$ under the $\Gamma$ action.
	Note that \eqref{eq:G-pre} can be viewed as a presentation of the group $G$ that is compatible with $\Gamma$ actions, and we will call it a \emph{$\Gamma$-presentation of $G$}.
	The minimal number of relations in the presentation \eqref{eq:G-pre}, which is one of the main objects studied in this paper, is related to the multiplicities of the irreducible $F_n\rtimes \Gamma$-quotients of $N$. We define the multiplicity as follows, and one can find that this quantity is similarly defined in \cite{Lubotzky, LW2017, LWZB}.
	
	\begin{definition}\label{def:mult-general}
	     Given a short exact sequence $1 \to N \to F \overset{\omega}{\to} H \to 1$ of $\Gamma$-groups, we let $M$ be the intersection of all maximal proper $F\rtimes \Gamma$-normal subgroups of $N$, and denote $\overline{N}=N/M$ and $\overline{F}=F/M$. Then $\overline{N}$ is a direct sum of finite irreducible $\overline{F}\rtimes \Gamma$-groups. For any finite irreducible $\overline{F}\rtimes \Gamma$-group $A$, we define $m(\omega, \Gamma, H, A)$ to be the multiplicity of $A$ appearing in $\overline{N}$. When $\omega$ refers to the surjection $F\rtimes \Gamma \to H\rtimes \Gamma$ induced by the $\Gamma$-equivariant quotient $F\to H$, we use the notation $m(\omega, \Gamma, H, A)$ instead of $m(\omega|_F, \Gamma, H, A)$ for convenience sake.
	\end{definition}
		
	Consider the short exact sequence \eqref{eq:G-pre}. As in Definition~\ref{def:mult-general}, we let $M$ be the intersection of all maximal proper $F_n\rtimes \Gamma$-normal subgroups of $N$, and define $R=N/M$ and $F=F_n/M$. 
	Then we obtain a short exact sequence
	\[
	    1 \longrightarrow R \longrightarrow F\rtimes \Gamma \longrightarrow G \rtimes \Gamma \longrightarrow 1.
	\]
    Note that $F\rtimes \Gamma$ acts on $R$ by conjugation, and maps the factor $A^{m(\pi, \Gamma, G, A)}$ of $R$ to itself. When $A$ is abelian, then the conjugation action on $A$ by elements in $R$ is trivial, so the $F \rtimes \Gamma$ action on $A$ factors through $G\rtimes \Gamma$, and hence $A$ is a finite simple $G\rtimes\Gamma$-module.

	\begin{lemma}\label{lem:d-cohom}
		Using the notation above, if $A$ is a finite simple $G\rtimes \Gamma$-module such that $\gcd(|\Gamma|, |A|)=1$, then we have
		\[
			m(\pi, \Gamma, G, A)= \frac{n\dim_{\F_\ell} A -\xi(A) +\dim_{\F_{\ell}} H^2(G\rtimes \Gamma, A) - \dim_{\F_{\ell}} H^1(G\rtimes \Gamma, A)}{h_{G\rtimes \Gamma}(A)},
		\]
		where $\ell$ is the exponent of $A$ and $\xi(A):=\dim_{\F_\ell} A^{\Gamma} / A^{G\rtimes \Gamma}$.
	\end{lemma}
	
	\begin{remark}\label{rmk:Lub}
		When $\Gamma$ is the trivial group, the lemma is the result in \cite{Lubotzky}
	\end{remark}
	
	\begin{proof}
		 Applying the inflation-restriction exact sequence to \eqref{eq:G-pre}, we obtain
		\begin{eqnarray}
			0&\longrightarrow& H^1(G\rtimes \Gamma, A^{N}) \longrightarrow H^1(F_n\rtimes \Gamma, A) \longrightarrow H^1(N, A)^{G\rtimes \Gamma} \nonumber \\
			&\longrightarrow& H^2(G\rtimes \Gamma, A^{N}) \longrightarrow H^2 (F_n\rtimes \Gamma, A). \label{eq:inf-res}
		\end{eqnarray}
		Also by $\gcd(|A|,|\Gamma|)=1$, the Hochschild-Serre spectural sequence $E^{ij}= H^i(\Gamma, H^j(F_n,A)) \Longrightarrow H^{i+j}(F_n\rtimes \Gamma, A)$ degenerates, so we have that
		\[
			H^2(F_n\rtimes \Gamma, A)\cong H^2(F_n, A)^{\Gamma}, 
		\]
		which is trivial because $F_n$ is a free profinite group.
		Note that $N$ acts trivially on $A$, so
		\[
			H^1(N, A)^{G\rtimes \Gamma}  = \Hom_{F_n\rtimes \Gamma}( N, A)=\Hom_{G\rtimes \Gamma}(A^{m(\pi,\Gamma,G,A)}, A)
		\]
		 because $A$ is a simple $\F_{\ell}[ G\rtimes \Gamma]$-module and $m(\pi,\Gamma,G,A)$ is the maximal integer such that $A^{m(\pi,\Gamma,G,A)}$ is an $F_n\rtimes \Gamma$-equivariant quotient of $N$. Then it follows that 
		\[
			\dim_{\F_{\ell}} H^1(N, A)^{G\rtimes \Gamma}= m(\pi, \Gamma,G,A) \dim_{\F_{\ell}} \Hom_{G\rtimes \Gamma} (A, A).
		\]
		Thus, by \eqref{eq:inf-res} it suffices to show that $\dim_{\F_{\ell}} H^1(F_n \rtimes \Gamma, A) = n \dim_{\F_{\ell}} A -\xi(A)$.
	
		Elements of $H^1(F_n\rtimes \Gamma, A)$ correspond to the $A$-conjugacy classes of homomorphic sections of $A\rtimes (F_n\rtimes \Gamma) \overset{\rho}{\to}  F_n\rtimes \Gamma$. We write every element of $F_n\rtimes \Gamma$ in the form of $(x, \gamma)$ for $x\in F_n$ and $\gamma\in \Gamma$, and similarly, write elements of $A \rtimes (F_n\rtimes \Gamma)$ as $(a; x, \gamma)$ for $a\in A$, $x\in F_n$ and $\gamma \in \Gamma$. Then $\rho$ maps $(a; x, \gamma)$ to $(x, \gamma)$ for any $a$, $x$ and $\gamma$. Note that a section of $\rho$ is completely determined by the images of $(x_{i, 1_{\Gamma}}, 1)$ and $(1, \gamma)$ for $i=1, \cdots, n$ and $\gamma\in\Gamma$, where $x_{i, 1_{\Gamma}}$'s are the $\Gamma$-generators of $F_n$ defined at the beginning of this section. Since $\gcd(|A|, |\Gamma|)=1$, we have $H^1(\Gamma, A)=0$ by the Schur-Zassenhaus theorem, which implies that the restrictions of all the sections of $\rho$ to the subgroup $\Gamma$ are conjugate to each other by $A$. So we only need to study the $A$-conjugacy classes of sections of $\rho$ which map $(1, \gamma)$ to $(1; 1, \gamma)$ for any $\gamma\in \Gamma$, and such sections are totally determined by the images of $(x_{i, 1_{\Gamma}}, 1)$ for $i=1,\cdots, n$. Let $s_1$ and $s_2$ be two distinct sections of this type.
		Under the multiplication rule of semidirect product, the conjugation of $(a; x, \gamma)$ by an element $\alpha \in A$ is 
		\begin{eqnarray*}
			(\alpha^{-1};1,1) (a; x, \gamma) (\alpha; 1,1) &=& (\alpha^{-1} \cdot a \cdot (x, \gamma)(\alpha); x, \gamma) \\
			&=& (\alpha^{-1} \cdot (x, \gamma)(\alpha); 1,1)(a; x, \gamma),
		\end{eqnarray*}
		where the last equality is because $A$ is abelian. Therefore, because of the assumption that $s_1(1, \gamma)=s_2(1, \gamma)=(1; 1 , \gamma)$ for any $\gamma\in \Gamma$, we see that $s_1$ and $s_2$ are $A$-conjugate if and only if there exists $\alpha\in A^{\Gamma} /A^{G\rtimes \Gamma}$ such that $s_2(x, \gamma)=(\alpha^{-1}\cdot (x,r)(\alpha); 1,1 )s_1(x, \gamma)$ for any $ x$, $\gamma$. So
		\begin{eqnarray*}
			\#\{\text{$A$-conjugacy classes of sections of $\rho$}\} &=& |A^{\Gamma}/ A^{G\rtimes\Gamma}|^{-1}\prod_{i=1}^n \# \rho^{-1}(x_{i, 1_{\Gamma}}, 1)\\
		&=& |A^{\Gamma} / A^{G\rtimes\Gamma}|^{-1} |A|^n,
		\end{eqnarray*}
		which proves that $\dim_{\F_{\ell}} H^1(F_n\rtimes \Gamma, A) = n\dim_{\F_{\ell}}A - \dim_{\F_{\ell}} (A^{\Gamma}/ A^{G\rtimes\Gamma})$.
	\end{proof}

	In this paper, instead of the $\Gamma$-presentations in the form of \eqref{eq:G-pre}, we want to study the presentations of pro-$|\Gamma|'$ completions of $\Gamma$-groups. We denote the pro-$|\Gamma|'$ completions of $F_n(\Gamma)$ and $G$ by $F'_n(\Gamma)$ and $G'$ respectively, and write $F'_n$ for $F'_n(\Gamma)$ when the choice of $\Gamma$ is clear. Then $F_n'$ and $G'$ naturally obtain $\Gamma$ actions from $F_n$ and $G$, and we have a short exact sequence
	\begin{equation}\label{eq:G-pre'}
	    1 \longrightarrow N' \longrightarrow F'_n \rtimes \Gamma \overset{\pi'}{\longrightarrow}  G'\rtimes \Gamma \longrightarrow 1,
	\end{equation}
	induced by \eqref{eq:G-pre}, which will be called \emph{a $|\Gamma|'$-$\Gamma$-presentation of $G'$}.

	\begin{proposition}\label{prop:d'-cohom}
		Use the notation above. Let $A$ be a finite simple $G' \rtimes \Gamma$-module, and denote the exponent of $A$ by $\ell$. If $\ell$ divides $|\Gamma|$, then $m(\pi', \Gamma, G', A)=0$. Otherwise, we have 
		\begin{eqnarray}
			 && m(\pi', \Gamma, G', A) \nonumber\\
%			 &=& \frac{n\dim_{\F_{\ell}} A  -\xi(A)+\dim_{\F_{\ell}} H^2(G'\rtimes \Gamma, A) - \dim_{\F_{\ell}} H^1(G'\rtimes \Gamma, A)}{ h_{G'\rtimes \Gamma} (A)} \label{eq:multi} \\
			 &=& \frac{n \dim_{\F_{\ell}} A -\xi(A)  +\dim_{\F_{\ell}} H^2(G, A)^{\Gamma} - \dim_{\F_{\ell}} H^1(G, A)^{\Gamma}}{h_{G\rtimes \Gamma} (A)}. \label{eq:multi'}
		\end{eqnarray}
%		where in \eqref{eq:multi'} $A$ is viewed as a $G\rtimes \Gamma$-module via the surjection $G\rtimes \Gamma \to G'\rtimes \Gamma$. Moreover, the equality in \eqref{eq:multi'} holds if $H^2(\ker(G \to G'), \F_\ell)=0$.	
	\end{proposition}
	
	\begin{remark}\label{rem:multi-ind}
		We see from \eqref{eq:multi'} that the multiplicity $m(\pi', \Gamma, G', A)$ depends on $n, \Gamma, G$ and $A$, but not on the choice of the quotient map $\pi'$. 
	\end{remark}
	
	\begin{proof}
		It is clear that if $\ell$ divides $|\Gamma|$, then $m(\pi', \Gamma, G', A)=0$. For the rest of the proof, assume $\ell \nmid |\Gamma|$.
		We consider the following commutative diagram
		\[\begin{tikzcd}
			F_n\rtimes \Gamma \arrow["\rho",two heads]{d}  \arrow["\varpi", two heads]{rd} \arrow["\pi", two heads]{r} & G\rtimes \Gamma \arrow["\rho_G", two heads]{d} \\
			F_n' \rtimes \Gamma \arrow["\pi'", two heads]{r} & G'\rtimes \Gamma,
		\end{tikzcd}\]
		where each of the vertical maps is taking the $|\Gamma|'$-completion of the first component in semidirect product. If $U$ is a maximal proper $F'_n \rtimes \Gamma$-normal subgroup of $\ker \pi'$ such that $\ker\pi'/U \simeq_{G'\rtimes\Gamma} A$, then its full preimage $\rho^{-1}(U)$ in $F_n \rtimes \Gamma$ is a maximal proper $F_n \rtimes \Gamma$-normal subgroup of $\ker \varpi$ with $\ker \varpi/\rho^{-1}(U) \simeq_{G'\rtimes \Gamma} A$. So by definition of multiplicities, we have that $m(\pi', \Gamma, G, A) \leq m(\varpi, \Gamma, G, A)$. On the other hand, because of $\gcd(|A|, |\Gamma|)=1$, if $V$ is a maximal proper $F_n \rtimes \Gamma$-normal subgroup of $\ker \varpi$ with $\ker\varpi/V \simeq_{G'\rtimes \Gamma} A$, then $F_n \rtimes \Gamma \twoheadrightarrow (F_n/V)\rtimes \Gamma$ factors through $\rho$, and hence we showed that $m(\pi', \Gamma, G, A)=m(\varpi, \Gamma, G, A)$. Because $\varpi$ defines a $\Gamma$-presentation of $G'$, by Lemma~\ref{lem:d-cohom} we obtain the equality 
		\[
		m(\pi', \Gamma, G', A) 			 = \frac{n\dim_{\F_{\ell}} A  -\xi(A)+\dim_{\F_{\ell}} H^2(G'\rtimes \Gamma, A) - \dim_{\F_{\ell}} H^1(G'\rtimes \Gamma, A)}{ h_{G'\rtimes \Gamma} (A)}.
		\]

		Let $W$ denote $\ker \rho_G=\ker(G \to G')$.  
		Because $G'$ is the pro-$|\Gamma|'$ completion of $G$ and $\ell \nmid |\Gamma|$, the pro-$\ell$ completion of $W$ is trivial. So as $W$ acts trivially on $A$, and by \cite[Prop.~(1.6.2)]{NSW} we have
		$H^i(W, A)=0$ for $i\geq 1$.
		Then by considering the Hochschild-Serre spectral sequence associated to
		\[
			1 \longrightarrow W \longrightarrow G\rtimes \Gamma \longrightarrow G'\rtimes \Gamma \longrightarrow 1,
		\]
		we see that 
		\[
		    H^1(G'\rtimes \Gamma, A) \cong H^1(G\rtimes \Gamma, A) \quad \text{and} \quad H^2(G'\rtimes \Gamma, A) \cong H^2(G\rtimes \Gamma, A).
		\]
%		where the latter embedding is an isomorphism if $H^2(W, A)=0$. Note that $H^2(W, A)= H^2(W, \F_{\ell})^{\oplus \dim_{\F_\ell} A}$ because $W$ acts trivially on $A$.
		
		Finally, since $\gcd(|A|, |\Gamma|)=1$, we have that $H^i(\Gamma, A)=0$ for any $i\geq 1$, and hence by the Hochschild-Serre spectral sequence of 
		\[
			1 \longrightarrow G \longrightarrow G\rtimes \Gamma \longrightarrow \Gamma \longrightarrow 1
		\]
		we have that $H^i(G\rtimes \Gamma, A)\cong H^i(G, A)^{\Gamma}$ for any $i$. Therefore, we have 
		\[
		    \dim_{\F_\ell} H^1(G'\rtimes \Gamma, A)=\dim_{\F_{\ell}} H^1(G, A)^{\Gamma}\quad \text{and} \quad \dim_{\F_{\ell}} H^2(G'\rtimes \Gamma) = \dim_{\F_{\ell}} H^2(G, A)^{\Gamma}.
		\]
%		where the equality holds if $H^2(W, \F_{\ell})=0$, and hence we finish the proof.
	\end{proof}
	
	By Remark~\ref{rem:multi-ind}, we can define the multiplicities as follows.

	\begin{definition}\label{def:multi}
		Let $\Gamma$ be a finite group, $G'$ a finitely generated pro-$|\Gamma|'$ $\Gamma$-group. Let $A$ be a finite simple $G'\rtimes \Gamma$-module and $n$ a sufficiently large integer such that there exists a $\Gamma$-equivariant surjection $\pi': F'_n\to G'$, we define $m(n, \Gamma, G', A)$ to be $m(\pi',\Gamma, G', A)$. In partucular, $m(n, \Gamma, G', A)$ can be computed by formula \eqref{eq:multi'}.
	\end{definition}
	
	\begin{proposition}\label{prop:min-gen}
		Use the notation in Proposition~\ref{prop:d'-cohom}. Then the minimal number of generators of $\ker\pi'$ as a closed normal $\Gamma$-subgroup of $F'_n$ is 
%		\begin{equation}\label{eq:relator-rank}
%			\sup_{\ell\, \nmid\,  |\Gamma|} \sup_{\substack{A: \text{ finite simple} \\ \F_{\ell}[ G'\rtimes \Gamma]\text{-modules}}} \left\lceil \frac{\dim_{\F_{\ell}} H^2(G'\rtimes \Gamma, A)-\dim_{\F_\ell} H^1(G'\rtimes \Gamma, A) -\xi(A)}{\dim_{\F_\ell} A}\right\rceil+n.
%		\end{equation}
%		Moreover, this minimal number is 
		\begin{equation}\label{eq:relator-rank'}
		    \sup_{\ell\, \nmid\,  |\Gamma|} \sup_{\substack{A: \text{ finite simple} \\ \F_{\ell}[ G'\rtimes \Gamma]\text{-modules}}} \left\lceil \frac{\dim_{\F_{\ell}} H^2(G, A)^{\Gamma}-\dim_{\F_\ell} H^1(G, A)^{\Gamma} -\xi(A)}{\dim_{\F_\ell} A}\right\rceil+n
		\end{equation}
%		and the equality holds if $H^2(\ker(G\to G'), \F_{\ell})=0$.
	\end{proposition}
	
	\begin{proof}
		We let $M$ be the intersection of all maximal proper $F'_n\rtimes \Gamma$-normal subgroups of $\ker\pi'$, and denote $R=\ker \pi'/M$ and $F=F'_n/M$. Then $R$ is isomorphic to a direct product of finite irreducible $F\rtimes \Gamma$-groups whose orders are coprime to $|\Gamma|$. A set of elements of $\ker \pi'$ generates $\ker \pi'$ as a normal subgroup of $F'_n\rtimes \Gamma$ if and only if their images generate $R$ as a normal subgroup of $F\rtimes \Gamma$.
		
		By \cite[Cors~5.9, 5.10]{LW2017}, if $m$ is a positive integer and $A$ is a finite irreducible $F\rtimes \Gamma$-group, then the minimal number of elements of $A^m$ that can generate $A^m$ as an $F\rtimes \Gamma$-group is 
		\[
			\text{is}\quad \begin{cases}
				1 &\text{if $A$ is non-abelian}\\
				\left\lceil \frac{m h_{F\rtimes \Gamma}(A)}{\dim_{\F_\ell} A} \right\rceil & \text{if $A$ is abelian, where $\ell$ is the exponent of $A$}.
			\end{cases}
		\]
		Recall that if $A$ is an abelian simple factor appearing in $R$, then the $F\rtimes \Gamma$ action on $A$ factors through $G'\rtimes \Gamma$, since the conjugation action of $R$ on $A$ is trivial. Therefore, by the argument above and \cite[Cor.~5.7]{LW2017}, the minimal number of generators of $R$ as an $F\rtimes \Gamma$-group is 
		\[
			\sup_{\ell\, \nmid\,|\Gamma|} \sup_{\substack{A: \text{ finite simple} \\ \F_{\ell}[G'\rtimes \Gamma]\text{-modules}}} \left\lceil \frac{m(n,\Gamma,G',A)  h_{G'\rtimes \Gamma}(A)}{\dim_{\F_\ell} A}\right\rceil.
		\]
		Then the proposition follows by Proposition~\ref{prop:d'-cohom} and the fact that $h_{G \rtimes \Gamma}(A)=h_{G'\rtimes \Gamma}(A)$.
	\end{proof}
	
	We give a lemma at the end of this section that will be used later.
	
	\begin{lemma}\label{lem:m-sum}
		Let $E, F$ and $G$ be $\Gamma$-groups such that there exist $\Gamma$-equivariant surjections $\alpha:E \to F$, $\beta: F\to G$ and a $\Gamma$-equivariant section $s: F \to E$ of $\alpha$. Denote $\pi=\beta \circ \alpha$. Let $A$ be a finite simple $G\rtimes \Gamma$-module.
		\[\begin{tikzcd}
			E \arrow[two heads, swap, "\alpha"]{r} \arrow[two heads, swap, "\pi", bend right=30]{rr} & F \arrow[two heads, "\beta"]{r} \arrow[hook, "s", swap, bend right]{l} & G.
		\end{tikzcd}\]
	\begin{enumerate}
		\item \label{item:m-sum-1} We have $m(\pi, \Gamma, G, A) \leq m(\alpha, \Gamma, F, A) + m(\beta, \Gamma, G, A).$
		\item \label{item:m-sum-2} Moreover, if every $\Gamma$-group extension of $F$ by $A$ splits, then $m(\pi, \Gamma, G, A) = m(\alpha, \Gamma, F, A) + m(\beta, \Gamma, G, A)$.
	\end{enumerate}
	\end{lemma}
	
	\begin{proof}
		Let $\calS$ be the set of all maximal proper $E\rtimes \Gamma$-normal subgroups $U$ of $\ker \pi$ with $\ker \pi/U \simeq_{G\rtimes \Gamma} A$. So by definition we have $\ker\pi/\left( \cap_{U\in \calS} U \right) \simeq A^{m(\pi, \Gamma, G, A)}$ as $G\rtimes \Gamma$-modules. Define
		\[
			\calS_1=\{ U \in \calS \mid \ker \alpha \subset U \} \quad \text{and} \quad 
			\calS_2=\{ U \in \calS \mid \ker \alpha \not\subset U  \}.
		\]
		One can easily check that there is a natural bijection 
		\begin{eqnarray*}
			\calS_1 & \longleftrightarrow & \left\{ V \,\Bigg|\,  \begin{aligned} \text{$V$ is a maximal proper $F\rtimes \Gamma$-normal subgroup  } \\ \text{of $\ker \beta$ such that $\ker \beta/V\simeq A$ as $G\rtimes \Gamma$-modules}  \end{aligned}\right\}  \\
			U &\longmapsto & \alpha(U),
		\end{eqnarray*}
		and it follows that $\ker\pi/\left( \cap_{U\in \calS_1} U \right) \simeq A^{m(\beta, \Gamma, G, A)}$. Similarly for the set $\calS_2$, there is a bijection
		\begin{eqnarray}
			\calS_2 & \longleftrightarrow & \left\{ V \,\Bigg|\,  \begin{aligned} \text{$V$ is a maximal proper $E\rtimes \Gamma$-normal subgroup  } \\ \text{of $\ker \alpha$ such that $\ker \alpha/V\simeq A$ as $F\rtimes \Gamma$-modules}  \end{aligned}\right\}  \label{eq:bij}\\
			U &\longmapsto & U \cap \ker \alpha \nonumber \\
			s(\ker \beta) V &\longmapsfrom&  V,\nonumber
		\end{eqnarray}
		Let's justify that \eqref{eq:bij} is a bijection. If $U \in \calS_2$, then $U\ker \alpha$ is an $E\rtimes \Gamma$-normal subgroup of $\ker \pi$ that properly contains $U$, and therefore we have $U\ker \alpha = \ker \pi$. So $U \cap \ker \alpha$ satisfies $\ker \alpha /( U \cap \ker \alpha)=(U\ker \alpha)/U = \ker \pi/U \simeq A$ as $F \rtimes \Gamma$-modules, and hence belongs to the the right-hand set in \eqref{eq:bij}. On the other hand, if $V$ is an element in the right-hand set of \eqref{eq:bij}, then we have $s(\ker \beta) V \in \calS_2$ because 
		\begin{eqnarray*}
			\ker \pi /(s(\ker \beta) V) &=& (s(\ker \beta) \ker \alpha) / (s(\ker \beta) V) \\
			&=& \ker \alpha / (s(\ker \beta) V \cap \ker \alpha) \\
			&=& \ker \alpha / V  \simeq_{F\rtimes \Gamma} A,
		\end{eqnarray*}
		where the equalities above use $s(\ker \beta) \ker \alpha = \ker \pi$ and $s(\ker \beta) \cap \ker \alpha =1$. It's not hard to check that the composition of the maps in two directions is the identity map. So we see that $\ker \pi/(\cap_{U\in \calS_2} U) \simeq A^{m(\alpha, \Gamma, F, A)}$. Then $\ker \pi/(\cap_{U\in \calS} U) \simeq A^{m(\pi, \Gamma, G, A)}$ is a submodule of $\ker \pi/(\cap_{U\in \calS_1} U) \times \ker \pi/(\cap_{U\in \calS_2} U)$, which implies \eqref{item:m-sum-1}.
	
	If any $\Gamma$-group extension of $F$ by $A$ splits, then $A^{m(\alpha, \Gamma, F, A) }\rtimes F$ is a $\Gamma$-quotient of $E$. Because the $\Gamma$-equivariant surjection $\beta$ factors through an extensions of $G$ by $A^{m(\beta, \Gamma, G, A)}$, we see that $m(\pi, \Gamma, G, A) \geq m(\alpha, \Gamma, F, A) + m(\beta, \Gamma, G, A)$. So we proved \eqref{item:m-sum-2}.
	\end{proof}

\section{Presentations of finitely generated profinite admissible $\Gamma$-groups}\label{sect:pres-ad}
	
	We first recall the definition of the admissible $\Gamma$-groups and the free admissible $\Gamma$-groups in \cite{LWZB}.
	
	\begin{definition}\label{def:admissible}
		A profinite $\Gamma$-group $G$ is called \emph{admissible} if it is $\Gamma$-generated by  elements $\{g^{-1}\gamma(g) \mid g\in G, \gamma\in \Gamma\}$ and is of order prime to $|\Gamma|$.
	\end{definition}
	
	Recall that for each positive integer $n$, we defined $F'_n$ to be the pro-$|\Gamma|'$ completion of $F_n$. We let $y_{i, \gamma}$ to be the image in $F'_n$ of the generators $x_{i,\gamma}$ of $F_n$, and therefore $F'_n$ is the free pro-$|\Gamma|'$ group on $\{y_{i, \gamma} \mid i=1, \cdots, n \text{ and } \gamma\in \Gamma\}$ where $\sigma\in \Gamma$ acts on $F'_n$ by $\sigma(y_{i,\gamma})=y_{i, \sigma \gamma}$. We fix a generating set $\{\gamma_1, \cdots, \gamma_d\}$ of the finite group $\Gamma$ throughout the paper. We denote $y_i:= y_{i, \id_{\Gamma}}$ and define $\calF_n(\Gamma)$ to be the closed $\Gamma$-subgroup of $F'_n$ that is generated as a $\Gamma$-subgroup by the elements 
	\[
		\{ y_i^{-1} \gamma_j(y_i) \mid i=1, \cdots, n \text{ and } j=1, \cdots, d\}.
	\]
	We will denote $\calF_n(\Gamma)$ by $\calF_n$ when the choice of $\Gamma$ is clear. The following is a list of properties of $\calF_n(\Gamma)$ proven in \cite[Lem.~3.1, Cor. 3.8 and Lem.~3.9]{LWZB}:
	\begin{enumerate}
		\item \label{item:prop-calF-1}$\calF_n$ is an admissible $\Gamma$-group and it does not depend on the choice of the generating set $\{\gamma_1, \cdots, \gamma_d\}$.
		\item \label{item:prop-calF-2} There is a $\Gamma$-equivariant quotient map $\rho_n: F'_n \to \calF_n$ such that the composition of the inclusion $\calF_n \subset F'_n$ with $\rho_n$ is the identity map on $\calF_n$. 
		\item \label{item:prop-calF-3} Define a map of sets for any $\Gamma$-group $G$
		\begin{eqnarray*}
		    Y: G &\longrightarrow& G^{d} \\
		    g &\longmapsto& (g^{-1}\gamma_1(g), g^{-1} \gamma_2(g), \cdots, g^{-1}\gamma_d(g)). 
		\end{eqnarray*}
		Then the map 
		\[
		    Y(G)^n \longrightarrow \Hom_{\Gamma}(\calF_n, G) 
		\]
		taking $(Y(g_1), \cdots, Y(g_n))$ to the restriction of the map $F'_n \to G$ with $y_i \mapsto g_i$ is a bijection.
	\end{enumerate}

	Let $G$ be an admissible $\Gamma$-group with a $\Gamma$-presentation defined by $F_n \rtimes \Gamma \overset{\pi}{\twoheadrightarrow} G \rtimes \Gamma$ such that the reduced map $F'_n \rtimes \Gamma \overset{\pi'}{\twoheadrightarrow} G \rtimes \Gamma$ satisfies that 
	\begin{equation}\label{eq:presentation-type}
		\text{$G$ is $\Gamma$-generated by coordinates of $Y(y_i)$, $i=1, \cdots, n$}.
	\end{equation}
%	where $y_i$ are the $\Gamma$-generators of $F'_n$ as defined above.
	In this section, we are interested in the $\Gamma$-presentations of this type, and under the condition \eqref{eq:presentation-type}, the restriction of $\pi'$ to the admissible subgroup $\calF_n$ of $F'_n$ is surjective. In other words, following by the property \eqref{item:prop-calF-2} of $\calF_n$ listed above, we want to study the $|\Gamma|'$-$\Gamma$-presentation $\pi'$ that factors through the quotient map $\rho_n: F'_n \to \calF_n$. We denote $\pi_{\ad}=\pi'|_{\calF_n\rtimes \Gamma}$ and obtain a short exact sequence
	\begin{equation}\label{eq:admiss-presentation}
		1 \longrightarrow N \longrightarrow \calF_n \rtimes \Gamma \overset{\pi_{\ad}}{\longrightarrow} G \rtimes \Gamma \longrightarrow 1,
	\end{equation}
	and we call it an \emph{admissible $\Gamma$-presentation of $G$}.
	
	Similarly to the previous section, we are interested in the multiplicities of each simple factors appearing as the quotients of $N$.
	
	\begin{lemma}\label{lem:ad-mult-indep}
		Let $G$ be an admissible $\Gamma$-group with an admissible $\Gamma$-presentation \eqref{eq:admiss-presentation} and $A$ a finite simple $G\rtimes \Gamma$-module with $\gcd(|A|,|\Gamma|)=1$. Then we have 
		\[
			m(\pi_{\ad}, \Gamma, G, A)=m(n, \Gamma, G, A)-m(n,\Gamma, \calF_n, A).
		\]
	\end{lemma}
	
	\begin{proof}
		We let $\rho_n: F'_n \to \calF_n$ be the quotient map described in the property \eqref{item:prop-calF-2}. 
		Let $\varpi$ be the composition of the following $\Gamma$-equivariant surjections
		and then $\varpi$ defines a $|\Gamma|'$-$\Gamma$-presentation of $G$. Let $\iota: \calF_n \to F'_n$ be the natural embedding. Then we have the following diagram
		\[\begin{tikzcd}
			F'_n \arrow[two heads, swap, "\rho_n"]{r} \arrow[two heads, swap, bend right, "\varpi"]{rr} & \calF_n \arrow[two heads, "\pi_{\ad}|_{\calF_n}"]{r} \arrow[hook, swap, "\iota", bend right]{l} & G,
		\end{tikzcd}\]
		Also note that $\calF_n$ is free as a $|\Gamma|'$-group, so any group extension of $\calF_n$ by $A$ splits. Therefore the lemma follows immediately from Lemma~\ref{lem:m-sum}.
	\end{proof}

	\begin{definition}
		Let $G$ be a $\Gamma$-group with an admissible $\Gamma$-presentation \eqref{eq:admiss-presentation}. For a finite simple $G\rtimes \Gamma$-module $A$ with $\gcd(|A|,|\Gamma|)=1$, we define $m_{\ad}(n, \Gamma, G, A)$ to be $m(\pi_{\ad}, \Gamma, G, A)$. By Lemma~\ref{lem:ad-mult-indep}, $m_{\ad}(n,\Gamma, G,A)=m(n,\Gamma, G,A)-m(n,\Gamma, \calF_n, A)$ does not depend on the choice of $\pi_{\ad}$.
	\end{definition}

	\begin{lemma}\label{lem:H1-calF}
		Let $A$ be a finite simple $\calF_n\rtimes \Gamma$-module such that $\gcd(|A|, |\Gamma|)=1$. Then we have
		\[
			\dim_{\F_{\ell}} H^1(\calF_n \rtimes\Gamma, A) = n \dim_{\F_{\ell}}(A/A^{\Gamma})-\xi(A).
		\]
	\end{lemma}
	
	\begin{proof}
		We use the idea in the proof of Lemma~\ref{lem:d-cohom}. Elements of $H^1(\calF_n\rtimes \Gamma, A)$ correspond to the $A$-conjugacy classes of homomorphic sections of $A\rtimes (\calF_n\rtimes \Gamma) \overset{\rho}{\to} \calF_n\rtimes \Gamma$. We use $(g,\gamma)$ to represent elements of $\calF_n\rtimes \Gamma$, and $(a; g, \gamma)$ to represent elements of $A\rtimes(\calF_n \rtimes \Gamma)$. Again, by the Schur-Zassenhaus theorem, we only need to count the $A$-conjugacy classes of sections of $\rho$ that maps $(1;1,\gamma)$ to $(1,\gamma)$. In other words, we only need to study the $A$-conjugacy classes of $\Gamma$-equivariant sections of $A\rtimes \calF_n \to \calF_n$. 
		
		By the property \eqref{item:prop-calF-3} of $\calF_n$, there is a bijection $Y(A\rtimes \calF_n)^n \to \Hom_{\Gamma}(\calF_n,  A\rtimes \calF_n)$ taking $(Y(g_1), \cdots, Y(g_n))$ to the restriction of the map $F'_n \to A\rtimes \calF_n$ with $y_i \mapsto g_i$. 
		For a $\Gamma$-equivariant section $s$ of $A\rtimes \calF_n \to \calF_n$, the elements $s(y_i^{-1}\gamma_j(y_i))$ in $A\rtimes \calF_n$ must map to $y_i^{-1}\gamma_j(y_i) \in \calF_n$ for each $i=1, \cdots, n$ and $j=1, \cdots, d$.
		Therefore, the $\Gamma$-equivariant sections of $A\rtimes \calF_n \to \calF_n$ are in one-to-one correspondence with elements in $Y(A\rtimes \calF_n)^n$ which map to $(Y(y_1), \cdots, Y(y_n))\in Y(\calF_n)^n$ under the natural quotient map $A\rtimes \calF_n \to \calF_n$ on each component.
		
		Let's consider $Y(y_i)$ and its preimages in $Y(A\rtimes \calF_n)$. Note that there is also a natural embedding $Y(\calF_n) \hookrightarrow Y(A\rtimes \calF_n)$ defined by the obvious section of split extension $A\rtimes \calF_n \twoheadrightarrow \calF_n$. So we can fix a $g\in A\rtimes \calF_n$ such that $Y(g)$ is the image of $Y(y_{i})$ under this embedding, and then $Y(g)$ is a preimage of $Y(y_i)$ under $\varphi$, where $\varphi$ is the quotient map $(A\rtimes \calF_n)^d \to \calF_n^d$. The self-bijection
		\begin{eqnarray*}
			(A\rtimes \calF_n)^d &\to& (A\rtimes \calF_n)^d \\
			(a_1, \cdots, a_d) &\mapsto& (ga_1 \gamma_1(g)^{-1}, \cdots, ga_d \gamma_d(g)^{-1} )
		\end{eqnarray*}
		maps $Y(A\rtimes \calF_n)$ to itself and $\varphi^{-1}(Y(y_i))$ to $A^d$. Thus,
		\[
			\#Y(A\rtimes \calF_n) \cap \varphi^{-1}(Y(y_i))=\#Y(A\rtimes \calF_n) \cap A^d = \# Y(A) = | A/A^\Gamma|,
		\]
		where the second equality above uses \cite[Lem.~3.4]{LWZB} and the last one uses \cite[Lem.~3.5]{LWZB}. So we've shown that there are $|A/A^{\Gamma}|$ elements in $Y(A\rtimes \calF_n)$ mapping to $Y(y_i)$, and it follows that the number of $\Gamma$-equivariant sections of $A\rtimes \calF_n \to \calF_n$ is $|A/A^{\Gamma}|^n$.
		
		Finally, recall that two sections $s_1, s_2$ of $A\rtimes (\calF_n \rtimes \Gamma) \to \calF_n \rtimes \Gamma$ are $A$-conjugate if and only if $s_1(g, \gamma) = (\alpha^{-1} \cdot (g, \gamma)(\alpha); 1,1) s_1(g, \gamma)$ for some $\alpha \in A^{\Gamma}/ A^{\calF_n \rtimes \Gamma}$ ,by the computation in the proof of Lemma~\ref{lem:d-cohom}. Therefore, we have $\# H^1(\calF_n \rtimes \Gamma, A) = \frac{|A/A^{\Gamma}|^n}{|A^{\Gamma} / A^{\calF_n \rtimes \Gamma}|}$, and hence we proved the lemma.
	\end{proof}
	
	\begin{corollary}\label{cor:ad-mult-compute}
		Under the assumptions in Lemma~\ref{lem:ad-mult-indep}, we have
		 \[
		 	m_{\ad}(n, \Gamma, G, A) = m(n, \Gamma, G, A) - \dfrac{n \dim_{\F_\ell}A^{\Gamma}}{ h_{G\rtimes \Gamma} (A)}.
		\]
	\end{corollary}
	
	\begin{proof}
		By Proposition~\ref{prop:d'-cohom} and Lemma~\ref{lem:H1-calF}, we have
		\[
			m(n, \Gamma, \calF_n , A)= \frac{n\dim_{\F_\ell} A^{\Gamma}  + \dim_{\F_\ell} H^2(\calF_n, A)^\Gamma }{ h_{G\rtimes \Gamma}(A)}.
		\]
		Because $\calF_n$ is a free $|\Gamma|'$-group, we see that $H^2(\calF_n, A)=0$, and then the corollary follows immediately by Lemma~\ref{lem:ad-mult-indep}.
	\end{proof}

	We point out in the next lemma that $A^{\Gamma}$ is strictly smaller than $A$ when $A$ is a nontrivial module.
	
		\begin{lemma}\label{lem:adm-Gamma-fix}
		If $G$ is an admissible $\Gamma$-group and $A$ is a $G\rtimes \Gamma$-group such that $\Gamma$ acts trivially on $A$, then $G\rtimes\Gamma$ acts trivially on $A$. 
	\end{lemma}
	
	\begin{proof}
		The $G\rtimes\Gamma$ action on $A$ induces a group homomorphism $G\rtimes \Gamma \to \Aut(A)$. So it suffices to show that $\Gamma$ is not contained in any proper normal subgroup of $G\rtimes \Gamma$.
		Suppose $M$ is a proper normal subgroup containing $\Gamma$. Then $B:=(G\rtimes\Gamma)/M$ is a $\Gamma$-quotient of $G$ and $\Gamma$ acts trivially on $B$. However, $G$ is admissible, so is generated by elements $g^{-1} \gamma(g)$ for $g\in G$ and $\gamma \in \Gamma$. Then the images of all $g^{-1}\gamma(g)$ in the $\Gamma$-quotient $B$ generate $B$ but each of these images is $1$, and hence we obtain the contradiction.
	\end{proof}

\section{Presentations of finitely generated profinite $\Gamma$-groups of level $\calC$}\label{sect:pro-C}

	Let $\calC$ be a set of isomorphism classes of finite $\Gamma$-groups. The \emph{variety of $\Gamma$-groups generated by $\calC$} is defined to be the smallest set $\overline{\calC}$ of isomorphism classes of $\Gamma$-groups containing $\calC$ that is closed under taking finite direct products, $\Gamma$-quotients and $\Gamma$-subgroups. For a given $\Gamma$-group $G$, we define the pro-$\calC$ completion of $G$ to be
	\[
		G^{\calC} = \varprojlim_{M} G/M,
	\]
	where the inverse limit runs over all closed normal $\Gamma$-subgroups $M$ of $G$ such that the $\Gamma$-group $G/M$ is contained in $\overline{\calC}$. We call a $\Gamma$-group $G$ \emph{level $\calC$} if $G^{\calC}=G$.
	
	We want to emphasis that we do not require $\overline{\calC}$ to be closed under taking group extensions, and it is different to many works in the literature about completions of groups. For example, if we let $\calC$ to be the set containing only the group $\Z/\ell\Z$ with the trivial $\Gamma$ action, then $G^{\calC}$ is the maximal quotient of $G$ that is isomorphic to a direct product of $\Z/\ell\Z$ on which $\Gamma$ acts trivially. If we want $G^{\calC}$ to give us the pro-$\ell$ completion of $G$, then we need to let $\calC$ contain all the finite $\Gamma$-groups of order a power of $\ell$.

	\begin{lemma}\label{lem:pro-C-completion}
		Let $F, G$ be $\Gamma$-groups and $\omega: F \to G$ a $\Gamma$-equivariant surjection. Let $\calC$ be a set of isomorphism classes of finite $\Gamma$-groups, and $\varphi$ the pro-$\calC$ completion map $F\to F^{\calC}$. Then we have the following commutative diagram of $\Gamma$-equivariant surjections
		\[\begin{tikzcd}
			F \arrow["\omega", two heads]{r}\arrow["\varphi", two heads]{d} & G \arrow[two heads]{d} \\
			F^{\calC} \arrow["\omega^\calC", two heads]{r} & G^{\calC},
		\end{tikzcd}\]
		where $\omega^{\calC}$ is the map taking quotient of $F^\calC$ by $\varphi(\ker \omega)$. 
	\end{lemma}
	
	\begin{proof}
		By the set-up, $\im \omega^{\calC}$ naturally fits into the right-lower position of this diagram, so it's enough to show that $\im \omega^{\calC} \simeq G^{\calC}$. First, $\im \omega^{\calC}$ is a quotient of $G$ and a quotient of $\calF^{\calC}$, so it is of level $\calC$ and hence is a quotient of $G^{\calC}$. On the other hand, we consider the natural pro-$\calC$ completion map $\alpha: G\to G^{\calC}$, and the composition $\alpha \circ \omega: F \to G^{\calC}$. Because $G^{\calC}$ is of level $\calC$, it follows that $\ker (\alpha \circ \omega) \supseteq \ker \varphi$. Also, because $\ker \omega \subseteq \ker (\alpha\circ \omega)$, we have that $\im(\alpha \circ \omega)=G^{\calC}$ is a quotient of $F/(\ker\omega \ker \varphi)= (F/\ker\varphi) / (\ker \omega / \ker \omega \cap \ker \varphi)= F^{\calC}/ \ker \omega^{\calC}=\im \omega^{\calC}$. So we proved that $\im \omega^{\calC} \simeq G^{\calC}$.
	\end{proof}
	
	\begin{definition}\label{def:pro-C-map}
		For any $\Gamma$-equivariant surjection $\omega: F \to G$, we define the pro-$\calC$ completion of $\omega$ to be $\omega^{\calC}: F^{\calC} \to G^{\calC}$ in Lemma~\ref{lem:pro-C-completion}.
	\end{definition}

	\begin{corollary}\label{cor:mult-ineq}
		Under the assumptions in Lemma~\ref{lem:pro-C-completion}, for any finite simple $G^{\calC}\rtimes \Gamma$-module $A$, we have $m(\omega^{\calC}, \Gamma, G^{\calC}, A) \leq m(\omega, \Gamma, G, A)$. 
	\end{corollary}
	
	\begin{proof}
		By definition of $\omega^{\calC}$,  $\ker \omega^{\calC}$ is the quotient of $\ker \varphi$ by $\ker \omega \cap \ker \varphi$, and we denote this quotient map by $\phi: \ker \omega \to \ker \omega^{\calC}$.
		If $N$ is a maximal proper $F^{\calC} \rtimes \Gamma$-normal subgroup of $\ker \omega^{\calC}$ such that $\ker \omega^{\calC}/N \simeq A$ as $G^{\calC}\rtimes \Gamma$-modules, then its preimage $\phi^{-1}(N)$ in $F$ is a maximal proper $F \rtimes \Gamma$-normal subgroup of $\ker \omega$ with $\ker \omega/\phi^{-1}(N) \simeq A$. Then the corollary follows by the definition of the multiplicity.
	\end{proof}

	\begin{proposition}\label{prop:C-mult-compute}
		Let $G$ be an admissible $\Gamma$-group, $\calC$ a set of isomorphism classes of finite $\Gamma$-groups and $A$ a finite simple $G^{\calC} \rtimes \Gamma$-module with $\gcd(|A|, |\Gamma|)=1$. Then, for a fixed positive integer $n$ such that there exists an admissible $\Gamma$-presentation of $G$ as \eqref{eq:admiss-presentation}, the multiplicity $m(\pi_{\ad}^{\calC}, \Gamma, G^{\calC}, A)$ does not depend on the choice of $\pi_{\ad}$, and so we denote $m(\pi_{\ad}^{\calC}, \Gamma, G^{\calC}, A)$ by $m_{\ad}^{\calC}(n, \Gamma, G, A)$. Then, we have 
		\[
			m_{\ad}^{\calC}(n, \Gamma, G, A) \leq m_{\ad}(n, \Gamma, G, A).
		\]
		Moreover, if $m_{\ad}(n, \Gamma, G, A)$ is finite, then the equality holds for sufficiently large $\calC$.
	\end{proposition}
	
	\begin{proof}		
	    Since $A$ is finite, we can find a finite set $\calC_1 \subset \calC$ of isomorphism classes of finite $\Gamma$-groups such that the map $G^{\calC}\rtimes \Gamma\to \Aut(A)$ induced by the $G^{\calC}\rtimes \Gamma$ action on $A$ factors through $G^{\calC_1} \rtimes \Gamma$, and hence $A$ is a simple $G^{\calC_1}\rtimes \Gamma$-module. Let $\calC_1 \subset \calC_2 \subset \cdots$ be an ascending sequence of finite sets of isomorphism classes of finite $\Gamma$-groups with $\cup \calC_i =\calC$. For each $i\leq j$, we have that $m(\pi_{\ad}^{\calC_i}, \Gamma, G^{\calC_i}, A) \leq m(\pi_{\ad}^{\calC_j}, \Gamma, G^{\calC_j}, A) \leq m(\pi_{\ad}^{\calC}, \Gamma, G^{\calC}, A)$ by Corollary~\ref{cor:mult-ineq}, and hence 
	    \[
	        m(\pi_{\ad}^{\calC}, \Gamma, G^{\calC}, A) = \lim_{i\to \infty} m(\pi_{\ad}^{\calC_i}, \Gamma, G^{\calC_i}, A).
	   \]
	   Since $\calC_i$ is a finite set of $\Gamma$-groups, \cite[Rmk.~4.9]{LWZB} shows that the multiplicity $m(\pi_{\ad}^{\calC_i}, \Gamma, G^{\calC_i}, A)$ does not depend on the choice of $\pi_{\ad}$. So we obtained that $m(\pi_{\ad}^{\calC}, \Gamma, G^{\calC}, A)$ also does not depend on the choice of $\pi_{\ad}$. The inequality in the proposition follows by $m(\pi_{\ad}^{\calC}, \Gamma, G^{\calC}, A) \leq m(\pi_{\ad}, \Gamma, G, A)$.
	   
	   The last statement in the proposition then automatically follows because 
		\[
			m_{\ad}(n, \Gamma, G, A) = \sup_{\substack{\calD: \text{ finite set} \\ \text{of $\Gamma$-groups}}} m^{\calD}_{\ad}(n,\Gamma, G, A).
		\]
	\end{proof}

\section{The heights of pro-$\calC$ groups}\label{sect:height}
	
	\begin{definition}
	     For a finite group $H$, we define $\frakh(H)$ to be the smallest integer $n$ such that there exists a length-$n$ sequence of normal subgroups of $H$
	     \[
	        1=H_0 \lhd H_1 \lhd \cdots \lhd H_n=H,
	     \]
	     where $H_{i+1}/H_i$ is isomorphic to a direct product of minimal normal subgroups of $H/H_i$.
	     We define the height of $H$ to be
	     \[
	        \widehat{\frakh}(H)=\max\{\frakh(U) \mid \text{$U$ is a subquotient of $H$}\}.
	     \]
	     For a profinite group $H$, the height is defined as 
	     \[
	        \widehat{\frakh}(H)=\sup_{\substack{\text{$U$: finite}\\ \text{quotient of $H$}}} \widehat{\frakh}(U).
	     \]
	\end{definition}
	
	\begin{lemma}\label{lem:h-directproduct}
	    Let $G$ and $H$ be two finite groups. Then $\widehat{\frakh}(G\times H) \leq \max\{\widehat{\frakh}(G), \widehat{\frakh}(H)\}$.
	\end{lemma}
	
	\begin{proof}
	    It suffices to show that $\frakh(U) \leq \max\{\widehat{\frakh}(G), \widehat{\frakh}(H)\}$ for any subquotient $U$ of $G\times H$. Each subquotient $U$ of $G\times H$ is a quotient of a subgroup $V$ of $G\times H$. Then because a sequence of normal subgroups of $V$ induces a sequence of normal subgroups of $U$, and a minimal normal subgroup is mapped to a product of minimal normal subgroups or the trivial subgroup under any quotient map.
	    We see that
	    $\frakh(U) \leq \frakh(V)$,  so we only need to show that $\frakh(V) \leq \max\{\widehat{\frakh}(G), \widehat{\frakh}(H)\}$ for any subgroup $V\subset G\times H$.
	    
	    We let $\Proj_{G}$ and $\Proj_{H}$ be the projections mapping $G\times H$ to $G$ and $H$ respectively, and denote $V_G=\Proj_G(V)$, $V_H=\Proj_H(V)$ and $\overline{V}=V/(\ker (\Proj_G) \ker (\Proj_H))$. Then $\Proj_G\times \Proj_H$ maps $V$ injectively into $V_{G}\times V_H$. Let $n$ denote $\max\{\widehat{\frakh}(G), \widehat{\frakh}(H)\}$, and then there exists a sequence 
	    \[
	        1\lhd V_{G,1}\times V_{H,1} \lhd V_{G,2}\times V_{H,2} \lhd \cdots \lhd V_{G,n}\times V_{H,n}=V_G \times V_H.
	    \]
	    of normal subgroups of $V_G\times V_H$ of length $n$, where $\{V_{*,i}\}$ for $*=G$ or $H$ is a sequence of normal subgroups of $V_{*}$ such that $V_{*,i+1}/V_{*,i}$ is a direct product of minimal normal subgroups of $V_*/V_{*,i}$. Assume that $A$ is a minimal normal subgroup of $V_{G}$ contained in $V_{G,1}$. Since $V$ is a subgroup of $V_G \times V_H$, we have that $A\cap V$ is normal in $V$. Then under the surjection $V\to V_G$, $A\cap V$ maps to a normal subgroup of $V_G$ that is contained in $A$. We see that $A\cap V$ is either $A$ or 1, because $A$ is minimal normal in $V_G$. In particular, if $A\cap V=A$, then it is a minimal normal subgroup of $V$, because otherwise $\Proj_G$ would map a minimal normal subgroup of $V$ contained in $A\cap V$ to a normal subgroup of $V_G$ that is properly contained in $A$ which contradicts to the assumption that $A$ is minimal normal. Thus, we showed that $V\cap \left(V_{G,1} \times V_{H,1}\right)$ is a direct product of minimal normal subgroups of $V$. Then by induction on $i$, we see that $\{V_i:=V\cap (V_{G,i} \times V_{H,i})\}_{i=1}^n$ forms a sequence of normal subgroups of $V$ such that $V_{i+1}/V_i$ is a direct product of minimal normal subgroups of $V/V_i$, and hence $\frakh(V)\leq \max\{\widehat{\frakh}(G), \widehat{\frakh}(H)\}$. 
	\end{proof}

    \begin{proposition}\label{prop:C-finite-h}
        Let $\Gamma$ be a finite group and $\calC$ a finite set of isomorphism classes of finite $\Gamma$-groups. For any $\Gamma$-group $G$, we have that $\widehat{\frakh}(G^{\calC})$ is at most 
        \[
            \widehat{\frakh}_{\calC}:=\max\{\widehat{\frakh}(H) \mid H\in \calC\}.
        \]
    \end{proposition}
    
	\begin{proof}
		By definition of $\widehat{\frakh}(G^{\calC})$, it suffices to prove $\widehat{\frakh}(G)\leq \widehat{\frakh}_{\calC}$ for any $G\in \overline{\calC}$.
		So we just need to show that the three actions, 1) taking $\Gamma$-quotients, 2) taking $\Gamma$-subgroups, and 3) taking finite direct products, do not produce groups with larger value of $\widehat{\frakh}$. For the first two actions, it is obvious that if $H$ is a $\Gamma$-quotient or a $\Gamma$-subgroup of $G$, then it is a quotient or a subgroup of $G$ by forgetting the $\Gamma$ actions, and hence $\widehat{\frakh}(H) \leq \widehat{\frakh}(G)$ by definition of heights. The last one follows by Lemma~\ref{lem:h-directproduct}.
	\end{proof}

	We finish this section by applying Proposition~\ref{prop:C-finite-h} to prove the following number theory theorem.

	\begin{theorem}\label{thm:finite-GC}
		Let $k/Q$ be a Galois global field extension with $\Gal(k/Q)\simeq \Gamma$ and $S$ a finite $k/Q$-closed set of places of $k$. Let $\calC$ be a finite set of isomorphism classes of finite $\Gamma$-groups. Then $G_S(k)^{\calC}$ is a finite group.
	\end{theorem}

	\begin{proof}
	By Proposition~\ref{prop:C-finite-h}, we have that 
	\[
		h:=\widehat{\frakh}(G_{S}(k)^{\calC}) \leq \widehat{\frakh}_{\calC}
	\]
	is finite. So there exists a sequence of normal subgroups of $G_{S}(k)^{\calC}$,
	\[
		1 = H_0 \lhd H_1 \lhd \cdots \lhd H_h = G_{S}(k)^{\calC},
	\]
	such that $H_{i+1}/H_i$ is isomorphic to a direct product of minimal normal subgroups of $H_h/H_i$. Note that each of the minimal normal subgroups is a (not necessarily finite) direct product of isomorphic finite simple groups. So, for each $i$, $H_{i+1}/H_i$ as a group is a direct product of finite simple groups. On the other hand, $G_{S}(k)^{\calC}$ is a quotient of $G_{S}(k)$, so is the Galois group of an extension of $k$ that is unramified outside $S$. Therefore, $H_{i+1}/H_i$ is the Galois group of an extension $K_{i}/K_{i+1}$ of some intermediate global fields between $k_{S}$ and $k$. We denote by $S_i$ the set of primes of $K_i$ lying above $S$.
	
	For a prime $\frakP$ of $K_i$, the local absolute Galois group $\calG_{\frakP}(K_i)$ is finitely generated, so there are finitely many Galois extensions of $K_{\frakP}$ of a fixed Galois group.
	Then for a simple group $E$, there exists an integer $N_{E,\frakP}(K_i)$ for each $\frakP\in S_i$, such that any Galois extension of $K_i$ whose Galois group is a subgroup of $E$ has discriminant at most $N_{E, \frakP}(K_i)$. Let $N_{E,S}(K_i)$ denote the product $\prod_{\frakP\in S_i}N_{E, \frakP}(K_i)$.
	By the Hermite-Minkowski theorem (see \cite[Thm.~8.23.5(3)]{Goss} for the function field version of this theorem), for each finite simple group $E$, there are only finitely many extensions of $K_i$ that are of Galois group $E$ and of discriminant at most $N_{E,S}(K_i)$. Therefore, there are finitely many extensions of $K_i$ that are of Galois group $E$ and unramified outside $S_i$.
	
	Since $\calC$ is finite, there are only finitely many simple groups that appear as composition factors of groups in $\overline{\calC}$ (see \cite[Cor.~6.12]{LW2017}). Now we consider the tower of extensions $K_i$'s. Note that $K_h=k$ and $\Gal(K_{h-1}/K_h)\simeq H_{h}/H_{h-1}$. By the above argument, we conclude that $H_h/H_{h-1}$ is a direct product of finite simple groups, that there are finitely many choices of these finite simple groups, and that for each of them there are finitely many copies of this simple group appearing in $H_h/H_{h-1}$. So we obtain that $H_h/H_{h-1}$ is finite, and hence $K_{h-1}$ is a finite extension of $k$. By induction, we see that $H_{i+1}/H_i$ is finite for each $i=h-1, \cdots, 0$, and it follows that $G_{S}(k)^{\calC}$ is finite.	
	\end{proof}

%%%%%%%%%%%%%%%%%%%%%%%%%
%%%%%%%%%%%%%%%%%%%%%%%%%
%%%%%%%%%%%%%%%%%%%%%%%%%

\section{A Generalized Version of Global Euler-Poincar\'{e} Characteristic Formula}\label{sect:Euler-Poincare}

	Throughout this section, we let $k/Q$ be a finite Galois extension of global fields, and $S$ a \emph{finite nonempty} $k/Q$-closed set of primes of $k$ such that $S_{\infty}(k)\subseteq S$. For each $A\in \Mod(\Gal(k_S/Q))$, we define
	\[
		\chi_{k/Q,S}(A) = \frac{ \#H^2(G_S(k), A)^{\Gal(k/Q)} \#H^0(G_S(k), A)^{\Gal(k/Q)}}{\#H^1(G_S(k), A)^{\Gal(k/Q)}},
	\]
	where $\Gal(k/Q)$ acts on $H^i(G_S(k), A)$ by conjugation. We will prove the following theorem.

\begin{theorem}\label{thm:EPChar}
	Use the assumption at the beginning of this section. If $A\in \Mod_{S}(\Gal(k_S/Q))$ has order prime to $[k:Q]$, then we have
	\[
		\chi_{k/Q,S}(A) = \#\left(\bigoplus_{v \in S_{\infty}(Q)} \hH^0(Q_{v}, A')\right) \Bigg{\slash} \#\left(\bigoplus_{v \in S_{\infty}(Q)} H^0(Q_v, A')\right)
	\]
\end{theorem}

\begin{remark}
	\begin{enumerate}
		\item If $k$ is a function field, then the theorem says that $\chi_{k/Q,S}(A)=1$ since $S_{\infty}(k)=\O$.
		\item When $k=Q$, the theorem is exactly the Global Euler-Poincar\'e Characteristic Formula (\cite[Thm. (8.7.4)]{NSW}). 
	\end{enumerate}
\end{remark}

\subsection{Preparation for the proof}

\begin{lemma}\label{lem:Shapiro-mod}
    Let $G$ be a profinite group and $U$ an open normal subgroup of $G$. Let $H$ be an open subgroup of $G$ and $V$ denote $U\cap H$. Then $H/V$ is naturally a subgroup of $G/U$, and for an $H$-module $A$ we have 
    \[
        H^i(U, \Ind_G^H A) \cong \Ind_{G/U}^{H/V}H^i(V, A) 
    \]
    as $G/U$-modules for each $i\geq 0$.
\end{lemma}

\begin{proof}
    Under the quotient map $G \twoheadrightarrow G/U$, $H/V$ is the image of $H$, so it is a subgroup of $G/U$. Then we have
    \[
        \Ind_G^H A = \Ind_G^{UH} \Ind_{UH}^H A = \bigoplus_{\sigma\in G/UH} \sigma (\Ind_{UH}^H A),
    \]
    where we denote by $\sigma(\Ind^H_{UH} A)$ the $\sigma UH \sigma^{-1}$-module, whose underlying group is $\Ind^H_{UH} A$ and the action of $\tau\in \sigma UH \sigma^{-1}$ is given by $a \mapsto \sigma^{-1} \tau \sigma a$. So
    \begin{eqnarray}
        H^i(U, \Ind_G^H A) &=& \bigoplus_{\sigma\in G/UH} H^i(U, \sigma(\Ind_{UH}^H A)) \nonumber\\
        &=& \bigoplus_{\sigma\in G/UH} \sigma_* H^i(U, \Ind_{UH}^H A) \nonumber\\
        &=& \Ind_{G/U}^{H/V} H^i(U, \Ind_{UH}^H A), \label{eq:st1}
    \end{eqnarray}
    where the second equality follows by $U\unlhd G$ and the definition of the conjugation action $\sigma_*$ on cohomology groups,
    and the last equality is because the quotient map $G\to G/U$ maps a set of representatives of $G/UH$ to a set of representatives of $(G/U)/(H/V)$.
    Since $A$ is an $H$-module, $UH$ acts on $\Ind_U^V A$, and moreover, it follows by $V=H\cap U$ that $\Ind_U^V A = \Ind_{UH}^H A$ as $UH$-modules. 
    So we have the following identity of $H/V$-modules
    \begin{equation}\label{eq:st2}
    	H^i(U, \Ind_{UH}^H A) = H^i(U, \Ind_U^V A) \cong H^i(V, A),
    \end{equation}
    where the last isomorphism follows by the Shapiro's lemma. 
    Then we proved the lemma by \eqref{eq:st1} and \eqref{eq:st2}.
\end{proof}
    
    For the rest of this section, we assume $S$ is a nonempty $k/Q$-closed set of primes of $k$ containing $S_{\infty}$ and denote $G=\Gal(k_S/Q)$ and $U=G_S(k)$. For each open subgroup $H$ of $\Gal(k_S/Q)$ we let $V=U\cap H$ and $K$ be the fixed field of $V$, and define a map
    \begin{eqnarray*}
        \varphi_{H,S} : \Mod(H) &\to& K'_0(\Z[H/V]) \\
        A &\mapsto& [H^0(V,A)] - [H^1(V,A)] + [H^2(V,A)] \\
        && - \left[\bigoplus_{\frakP\in S_{\infty}(K)} \widehat{H}^0(K_{\frakP}, A')\right]^{\vee}+ \left[\bigoplus_{\frakP\in S_{\infty}(K)} H^0(K_{\frakP}, A')\right]^{\vee},
    \end{eqnarray*}
    where $H/V$ acts on $\oplus_{\frakP \in S_{\infty}(K)} H^0(K_{\frakP}, A')$ (similarly on Tate cohomology) by its permutation action on $S_{\infty}(K)$ and by the $\Gal_{\frakP}(K/Q) \cap H$ on each summand, and the Pontryagin dual is taking on the classes of $K_0'(\Z[H/V])$.

\begin{lemma}\label{lem:H0}
	Using the notation above, we have the following isomorphisms of $G/U$-modules for any $A\in\Mod(H)$
	\begin{equation}\label{eq:H0-eq}
		\bigoplus_{\frakp\in S_{\infty}(k)} H^0(k_{\frakp}, \Ind_G^H A) \cong \Ind_{G/U}^{H/V} \bigoplus_{\frakP \in S_{\infty}(K)} H^0(K_{\frakP}, A),
	\end{equation}
	\begin{equation}\label{eq:hH0-eq}
		\bigoplus_{\frakp\in S_{\infty}(k)} \hH^0(k_{\frakp}, \Ind_G^H A) \cong \Ind_{G/U}^{H/V} \bigoplus_{\frakP \in S_{\infty}(K)} \hH^0(K_{\frakP}, A).
	\end{equation}
\end{lemma}

\begin{proof}
	It suffices to fix a $v\in S_{\infty}(Q)$ and prove \eqref{eq:H0-eq} and \eqref{eq:hH0-eq} for places above $v$.
	For each $\frakp \in S_{v}(k)$, $\Ind_G^H A$ as a $\calG_{v}(Q)$-module has the following canonical decomposition (see \cite[\S~1.5, Ex.~5]{NSW})
	\begin{equation}\label{eq:1.5}
		\Res_{\calG_{v}}^G \Ind_G^H A = \bigoplus_{\sigma \in \calG_{v} \backslash G / H} \Ind_{\calG_{v}}^{\calG_{v} \cap \sigma H \sigma^{-1}} \sigma \Res^H_{\sigma^{-1} \calG_{v} \sigma \cap H} A.
	\end{equation}
		
	If $v$ splits completely in $k/Q$, then $\Gal_v(k/Q)=1$ and $\calG_\frakp(k)=\calG_v(Q)$.
	So we have the following identities of $\Gal_v(k/Q)$-modules
	\begin{eqnarray}
		H^0(k_{\frakp}, \Ind_G^H A) &=& \bigoplus_{\sigma \in \calG_{v} \backslash G/ H} H^0(\calG_{v} \cap \sigma H \sigma^{-1}, \sigma \Res^H_{\sigma^{-1} \calG_{v} \sigma \cap H} A) \nonumber\\ 
		&=& \bigoplus_{\sigma \in \calG_{v} \backslash G / H} \sigma_{*} H^0(\sigma\calG_{v}\sigma^{-1} \cap H, \Res^H_{\sigma^{-1} \calG_{v} \sigma \cap H} A), \label{eq:H0-decomposition}
	\end{eqnarray}
	where the first equality uses \eqref{eq:1.5} and the Shapiro's lemma, and the second follows by definition of the conjugation action on cohomology groups. We let $L$ denote the fixed field of $H$.
	Then, the set $\{ \sigma \calG_{\frakp} \sigma^{-1} \cap H \mid \sigma \in \calG_{\frakp} \backslash G \slash H\}$ is exactly the set $\{\calG_{w}(L) \mid w \in S_v(L)\}$.
	Therefore, we have the identity of abelian groups (hence of $\Gal_v(k/Q)$-modules since $\Gal_v(k/Q)=1$)
	\[
		H^0(k_{\frakp}, \Ind_G^H A) = \bigoplus_{w \in S_{v}(L)} H^0(L_w, A),
	\]
	and hence
	\begin{equation}\label{eq:LHS-real}
		\bigoplus_{\frakp \in S_{v}(k)} H^0(k_{\frakp}, \Ind_G^H A) = \Ind_{G/U}^1 \left( \bigoplus_{w\in S_{v}(L)} H^0(L_w, A)  \right)
	\end{equation}
	because the $\Gal(k/Q)$-action on this direct sum is determined by its permutation action on places above $v$.
	On the other hand, because of  $K=k\cap L$, the assumption that $v$ splits completely in $k/Q$ implies that $w$ splits completely in $K$ for any $w\in S_v(L)$ and then we obtain
	\begin{equation}\label{eq:RHS-real}
		\bigoplus_{\frakP\in S_{v}(K)} H^0(K_{\frakP}, A) = \Ind_{H/V}^1 \left( \bigoplus_{w \in S_{v}(L)} H^0(L_w, A)  \right).
	\end{equation}
	Thus, \eqref{eq:LHS-real} and \eqref{eq:RHS-real} prove \eqref{eq:H0-eq} in this case. The isomorphism in \eqref{eq:hH0-eq} can be proven using the exactly same argument.
	
	Otherwise, $v$ is ramified in $k/Q$, so $\Gal_v(k/Q)\simeq \Z/2\Z$, $\calG_{\frakp}(k)=1$ and $\calG_{\frakP}(K)=1$ for each $\frakp \in S_v(k)$ and $\frakP \in S_v(K)$. Then \eqref{eq:hH0-eq} automatically follows because of $\hH^0(k_{\frakp}, \Ind_G^H A)= \hH^0(K_{\frakP}, A) =0$.
	The set of right cosets $G/H$ naturally acts on $S_v(L)$, and moreover, for any $w\in S_v(L)$ and $\sigma_1, \sigma_2 \in G/H$, $\sigma_1^{-1} \sigma_2$ is contained in $\calG_w(L) \subset \Gal(K/L)$ if and only if $\sigma_1(w)=\sigma_2(w)$.
	So by \eqref{eq:1.5}, we have the following identities of $\Gal_{\frakp}(k/Q)$-modules
	\[
		H^0(k_{\frakp}, \Ind_G^H A)= \Res^G_{\calG_v(Q)}\Ind_G^H A = \bigoplus_{w \in S_{\R}(L)} A_w \oplus \bigoplus_{w \in S_{\C}(L)} (A_w \oplus \tau A_w),
	\]
	where $A_w:=\Res_{\calG_w(L)}^H A$ and $\tau$ denotes the nontrivial element in $\Gal_v(k/Q)$. So we have the following identity of $\Gal(k/Q)$-modules 
	\begin{equation}\label{eq:LHS-imaginary}
		\bigoplus_{\frakp \in S_v(k)} H^0(k_{\frakp}, \Ind_G^H A) = \bigoplus_{w\in S_{\R}(L) } \Ind_{\Gal(k/Q)}^{\Gal_v(k/Q)}  A_w \oplus \bigoplus_{w\in S_{\C}(L)} \Ind^1_{\Gal(k/Q)} A_w.
	\end{equation}
	Finally, because $w \in S_v(L)$ is imaginary if and only if $\Gal_w(K/L)=1$, we have
	\begin{eqnarray}
		&&\Ind_{G/U}^{H/V} \bigoplus_{\frakP \in S_{v}(K)} H^0(K_\frakP, A) \nonumber\\
		&=& \Ind_{G/U}^{H/V} \left(\bigoplus_{w \in S_v(L)} \bigoplus_{\frakP \in S_w(K)} A_w\right)\\
		&=& \Ind_{G/U}^{H/V} \left(\bigoplus_{w \in S_{\R}(L) } \Ind^{\Gal_w(K/L)}_{\Gal(K/L)} A_w \oplus \bigoplus_{w\in S_{\C}(L)} \Ind^1_{\Gal(K/L)} A_w \right)\nonumber \\
		&=& \bigoplus_{w\in S_{\R}(L)} \Ind_{\Gal(k/Q)}^{\Gal_v(k/Q)}  A_w \oplus \bigoplus_{w\in S_{\C}(L)} \Ind^1_{\Gal(k/Q)} A_w. \label{eq:RHS-imaginary}
	\end{eqnarray}
	Thus, \eqref{eq:H0-eq} follows by \eqref{eq:LHS-imaginary} and \eqref{eq:RHS-imaginary}.
\end{proof}
    
	The corollary below immediately follows by Lemmas~\ref{lem:Shapiro-mod}, \ref{lem:H0} and the fact $\Ind_G^H A'=(\Ind_G^H A)'$.
%	\Yuan{Note that $\Ind_G^H A'=\Ind_G^H \Hom(A, \mu)=\Hom_H(\Z[G], \Hom(A, \mu))=\Hom(\Z[G]\otimes_{\Z[H]} A, \mu)=\Hom(\Ind_G^H A, \mu)=(\Ind_G^H A)'$ because $H$ is open in $G$.}
  
\begin{cor} \label{cor:Phi-property-2}
	For any open subgroup $H$ of $G$ and $A\in \Mod(H)$, we have 
	\[
		\varphi_{G,S}(\Ind_G^H A)\simeq \Ind_{G/U}^{H/V} \varphi_{H,S}(A).
	\]
\end{cor}

\begin{lemma}\label{lem:Phi-property-1}
	 The map $\varphi_{G,S}$ is additive on short exact sequences of modules in $\Mod_S(G)$.
%	
%	For any open subgroup $H$ of $G$ and $A\in \Mod_S(H)$
%    \begin{enumerate}
%        \item \label{lem:Phi-property-1} The map $\varphi_{G,S}$ is additive on short exact sequences of modules in $\Mod_{S}(G)$.
%        \item \label{item:Phi-property-2} For any $A\in \Mod_{S}(H)$ we have $\varphi_{G,S}(\Ind_G^H A) \cong \ind_{G/U}^{H/V}\varphi_{H,S}(A).$ \yuan{$\ind_{G/U}^{H/V}$ is always a well defined map. See Benson pge 173}
%    \end{enumerate}
%\end{lemma}
%\yuan{$S$ has to be nonempty because of (8.6.10)}
\end{lemma}
\begin{proof}
    Denote $G_S(k)$ by $G_S$. Let $0\to A_1 \to A_2 \to A_3 \to 0$ be an exact sequence of finite modules in $\Mod_{S}(G)$. By considering the associated long exact sequence of group cohomology, we have the following identity of elements in $K'_0(\Z[\Gal(k/Q)])$
    \begin{equation}\label{eq:equal-1}
        \sum_{i=0}^{2}\sum_{j=1}^{3} (-1)^{i+j+1}[H^i(G_S, A_j)] = \sum_{i=3}^{4}\sum_{j=1}^{3} (-1)^{i+j}[H^i(G_S, A_j)]+[\delta H^4(G_S, A_3)],
    \end{equation}
    where $\delta$ denotes the connecting map $H^i\to H^{i+1}$ (or $\widehat{H}^i \to \widehat{H}^{i+1}$ for Tate cohomology groups) in the long exact sequence.
    By \cite[Thm. (8.6.10)(ii)]{NSW}, for $i\geq 3$ and any $j$, the restriction map $H^i(G_S,A_j)\to \oplus_{\frakp\in S_{\R}(k)} H^i(k_{\frakp}, A_j)$ is an isomorphism. Note that for $\frakp\in S_{\R}(k)$, we have $\calG_{\frakp}(k)=\Z/2\Z$, so by \cite[Prop.~(1.7.1) and (1.7.2)]{NSW} we have
    \begin{eqnarray*}
    	\sum_{i=3}^4 \sum_{j=1}^3 (-1)^{i+j} [H^i(G_S, A_j)]&=&\sum_{i=3}^4 \sum_{j=1}^3 (-1)^{i+j}\left[ \bigoplus_{\frakp \in S_{\R}(k)} H^i(k_{\frakp}, A_j)\right] \\
	&=&\sum_{\frakp \in S_{\R}(k)}  \sum_{i={-1}}^0 \sum_{j=1}^3 (-1)^{i+j} \left[ \widehat{H}^i(k_{\frakp}, A_j)\right]\\
	&=& 0.
    \end{eqnarray*}

%    and hence we have $[\oplus_{\frakp \in S_{\R}} H^3(k_{\frakp},A_j)] = [\oplus_{\frakp \in S_{\R}}\widehat{H}^{-1}(k_{\frakp}, A_j)]$ and $[\oplus_{\frakp \in S_{\R}}H^4(k_{\frakp}, A_j) ]=[\oplus_{\frakp \in S_{\R}} \widehat{H}^0 (k_{\frakp}, A_j)]$ by \cite[Prop.~(1.7.1)]{NSW}. 
%	Because $\widehat{H}^{-1}(k_{\frakp}, A_j)$ and  $\widehat{H}^0(k_{\frakp}, A_j)$ are of the same size and with the trivial action of $\calG_{\frakp}$, we see that if $\Gal_{\frakp}(k/Q)=1$, then
%	$[\widehat{H}^{-1}(k_{\frakp}, A_j)]=[\widehat{H}^0(k_{\frakp}, A_j)]$ as elements in $K'_0(\Z[\Gal_{\frakp}(k/Q)])=K'_0(\Z)$, and otherwise $[\widehat{H}^{-1}(k_{\frakp}, A_j)]=[\widehat{H}^0(k_{\frakp}, A_j)]=0$.
%	So we have $[\oplus_{\frakp \in S_{\R}} \widehat{H}^{-1}(k_{\frakp}, A_j)]=[\oplus_{\frakp \in S_{\R}} \widehat{H}^0(k_{\frakp}, A_j)]$ as elements in $K'_0(\Z[G/U])$.

 %   Also, one can check that $[\oplus_{\frakp \in S_{\R}}\widehat{H}^{-1}(k_{\frakp}, A_j)]=[\oplus_{\frakp \in S_{\R}}\widehat{H}^0(k_{\frakp}, A_j)]$ follows by definition of the Tate cohomology groups. \yuan{Use (1.7.9)}
    
    So \eqref{eq:equal-1} gives
    \begin{eqnarray}
        \sum_{i=0}^{2}\sum_{j=1}^{3} (-1)^{i+j+1}\left[H^i(G_S, A_j)\right] &=& \left[\delta H^4(G_S, A_3)\right] \nonumber \\
        &=& \left[\bigoplus_{\frakp\in S_{\R}(k)} \delta H^4(k_{\frakp}, A_3)\right] \nonumber \\
        &=& \left[\bigoplus_{\frakp\in S_{\R}(k)} \delta \widehat{H}^0(k_{\frakp}, A_3)\right] \nonumber \\
        &=& \left[\bigoplus_{\frakp\in S_{\R}(k)} \ker \left(\widehat{H}^1(k_{\frakp}, A_1) \to \widehat{H}^1(k_{\frakp}, A_2)\right)\right] \nonumber\\
        &=& \left[\bigoplus_{\frakp\in S_{\R}(k)} \coker \left(\widehat{H}^1(k_{\frakp}, A'_2) \to \widehat{H}^1(k_{\frakp}, A'_1)\right)\right]^{\vee} \nonumber \\
        &=& \left[ \bigoplus_{\frakp \in S_{\R}(k)} \delta \widehat{H}^1(k_{\frakp}, A'_1) \right]^{\vee} \label{eq:star1}
    \end{eqnarray}
    where the fourth equality and the last one uses the long exact sequence of Tate cohomology groups, the fifth one uses the local duality theorem \cite[Thm. (7.2.17)]{NSW}. On the other hand, again by \cite[Prop.~(1.7.1) and (1.7.2)]{NSW}, the long exact sequence induced by 
    \begin{equation}\label{eq:ses'}
        0\to A'_3 \to A'_2 \to A'_1 \to 0
    \end{equation} 
    implies
    \begin{eqnarray}
        && \sum_{j=1}^3 (-1)^{j+1} \left[\bigoplus_{\frakp \in S_{\R}(k)} \hH^0(k_{\frakp},A'_j)\right]  \nonumber\\
        &=& \sum_{j=1}^3 (-1)^{j+1}\left[\bigoplus_{\frakp \in S_{\R}(k)} \hH^1(k_{\frakp},A'_j)\right]  \nonumber \\
        &=& \sum_{j=1}^3 (-1)^{j+1} \left[\bigoplus_{\frakp \in S_{\R}(k)} H^0(k_{\frakp},A'_j)\right]  + \left[\bigoplus_{\frakp \in S_{\R}(k)} \delta H^1(k_{\frakp},A'_1)\right] \label{eq:star2}
    \end{eqnarray}
    where the last equality follows by the long exact sequence  of group cohomology induced by \eqref{eq:ses'}. Therefore, combining \eqref{eq:star1} and \eqref{eq:star2}, we obtain 
	\[
    		\varphi_{G,S}(A_1) - \varphi_{G,S}(A_2) +\varphi_{G,S}(A_3)=0.
  	\]
\end{proof}

\
\begin{lemma}\label{lem:EPChar-mu}
	If $\ell \in \NN(S)$ is a prime, then we have the following identities of elements in $K'_0(\F_{\ell}[\Gal(K/Q)])$ for any Galois extension $K$ of $Q$ with $k(\mu_\ell) \subset K \subset k_S$
	\begin{eqnarray*}
        [H^0(\Gal(k_S/K), \mu_{\ell})] &=& [\mu_{\ell}] \\
        {}[H^1(\Gal(k_S/K), \mu_{\ell})] &=&  [\calO^{\times}_{K,S}/\ell] + [\Cl_S(K)[\ell]] \\
        {}[H^2(\Gal(k_S/K),\mu_{\ell})] &=& [\Cl_S(K)/\ell] - [\F_{\ell}] + \left[\bigoplus_{\frakP \in S\backslash S_{\infty}(K)} \F_{\ell}\right] + \left[\bigoplus_{\frakP \in S_{\infty}(K)} \hH^0(\calG_\frakP, \F_{\ell})\right],
    \end{eqnarray*}
    where $\Cl_S(K)$ is the $S$-class group of $K$, $\Cl_S(K)[\ell]$ is the $\ell$-torsion subgroup of $\Cl_S(K)$, and $\calO_{K,S}^{\times}/\ell$ and $\Cl_S(K)/\ell$ denote the maximal exponent-$\ell$ quotients of $\calO_{K,S}^{\times}$ and $\Cl_S(K)$ respectively.
\end{lemma}

\begin{proof}
	The lemma follows directly from the claims (i)-(iii) in the proof of \cite[Thm.~8.7.4]{NSW}.  
	Though the proof of those claims only shows these identities when each terms are treated as Grothendieck group elements of $\Gal(K/k)$-modules, one can check that the ideas there work generally for the base field $Q$ instead of $k$.
\end{proof}

\subsection{Proof of Theorem~\ref{thm:EPChar}}

	For any $G$-module $A$ and $v\in S_{\infty}(Q)$,
	\[
		\bigoplus_{\frakp \in S_{v}(k)} H^0(k_{\frakp}, A') \cong \Ind^{\Gal_\frakp(k/Q)}_{\Gal(k/Q)} H^0(k_{\frakp}, A')
	\]
	as $\Gal(k/Q)$-modules, where $\frakp$ on the right-hand side is an arbitrarily chosen place in $S_{v}(k)$. So by the Shapiro's lemma, we have
	\begin{equation}\label{eq:H-fix}
		\left(\bigoplus_{\frakp \in S_{v}(k)} H^0(k_{\frakp}, A') \right)^{\Gal(k/Q)}\cong H^0(k_{\frakp}, A')^{\Gal_{\frakp}(k/Q)} = H^0(Q_v, A').
	\end{equation}
	If $\Gal_{\frakp}(k/Q) = \Z/2\Z$, then $\hH^0(k_{\frakp}, A') = \hH^0(Q_v, A')=0$ because $|A'|$ has to be odd as $\gcd(|A|, [k:Q])=1$. If $\Gal_{\frakp}(k/Q)=1$, then $\widehat{H}^0(k_{\frakp}, A')=\widehat{H}^0(Q_v, A')$. So we have
	\begin{equation}\label{eq:hH-fix}
		\left(\bigoplus_{\frakp \in S_{v}(k)} \hH^0(k_{\frakp}, A') \right)^{\Gal(k/Q)}\cong \hH^0(k_{\frakp}, A')^{\Gal_{\frakp}(k/Q)} = \hH^0(Q_v, A'),
	\end{equation}
	Note that for any $M \in \Mod(\Gal(k/Q))$, we have $(M^{\vee})^{\Gal(k/Q)}=\Hom_{\Gal(k/Q)}(M, \Q/\Z)\simeq M_{\Gal(k/Q)}$. When $M$ has order prime to $[k:\Q]$, $M^{\Gal(k/\Q)}$ and $M_{\Gal(k/\Q)}$ are isomorphic.	So the $\Gal(k/Q)$-invariants of 
	\[
		\left(\bigoplus_{\frakp \in S_{v}(k)} H^0(k_{\frakp}, A') \right)^\vee \quad \text{and} \quad \left(\bigoplus_{\frakp \in S_{v}(k)} \hH^0(k_{\frakp}, A') \right)^\vee
	\]
	are $H^0(Q_v, A')$ and $\hH^0(Q_v, A')$ respectively.

	We let $R$ denote the ring $\prod_{p\nmid [k:Q]} \Z_p$.
	Let $\Theta: K_0'(R[\Gal(k/Q)]) \to \Z$ be the map defined by sending the class $[A]$ to the size of $A^{\Gal(k/Q)}$, which is a group homomorphism because taking $\Gal(k/Q)$-invariants is an exact functor in the category of $R[\Gal(k/Q)]$-modules. So we want to show that $\Theta\circ\varphi_{G, S}$ is the zero map when restricted to modules in $\Mod_{S}(\Gal(k_S/Q))$ with order prime to $[k:Q]$.
	By Lemma \ref{lem:Phi-property-1} we just need to show
    \begin{equation}\label{eq:stat}
    	\Theta \circ \varphi_{G,S}(K'_0(\F_{\ell}[\Gal(E/Q)]))=0,
    \end{equation}
    for any prime integer $\ell\in \NN(S)$ with $\ell \nmid [k:Q]$ and any finite extension $E$ of $k$ that is Galois over $Q$.
%    Note that as $\ell \nmid \#\Gamma$, the Grothendieck group $K'_0(\F_{\ell}[\Gamma])$ is the free abelian group genenrated by the isomorphism classes of simple $\F_{\ell}[\Gamma]$-modules. 
    Because the codomain of the map $\Theta$ is free, \eqref{eq:stat} is equivalent to the vanishing of $\Theta\circ \varphi_{G,S}$ on the torsion-free part of $K'_0(\F_{\ell}[\Gal(E/Q)])$. Note that, by \cite[Lem.~(7.3.4)]{NSW}, the $\Q$-linear space $K'_0(\F_{\ell}[\Gal(E/Q)]) \otimes_{\Z} \Q$ is generated by classes in the form of $\Ind^{\overline{C}}_{\Gal(E/Q)} A$, where $\overline{C}$ runs over all cyclic subgroups of $\Gal(E/Q)$ of order prime to $\ell$ and $A$ runs over classes of $K_0'(\F_{\ell}[\overline{C}])$. 
    For such $\overline{C}$ and $A$, we denote $C$ the full preimage of $\overline{C}$ in $G=\Gal(k_S/Q)$, and then by Corollary~\ref{cor:Phi-property-2} and $\Ind^{\overline{C}}_{\Gal(E/Q)} A = \Ind^C_G A$, we have that 
    $\Theta \circ \varphi_{G,S}(\Ind^C_{G} A)=0$ if and only if $\Theta \circ \varphi_{C,S}(A)=0$. By setting $G$ to be $C$, $U$ to be $C\cap U$, $Q$ to be $(k_S)^C$ and $k$ to be $(k_S)^{C\cap U}$, we finally reduce the problem to the statement that we will prove in the rest of this section:
    \begin{gather}
    	\Theta\circ\varphi_{G,S}(A)=0 \text{ for all $A\in \Mod_{\ell}(G)$ such that $k(A)/Q$ is} \label{eq:statement} \\
	\text{a cyclic extension of $Q$ of order relatively prime to $\ell$.} \nonumber
    \end{gather}
        
    We let $K=k(A, \mu_{\ell})$. So under the assumption in \eqref{eq:statement}, we have that $\Gal(K/Q)$ is an abelian group of order relatively prime to $\ell$, in which case the Hochschild-Serre spectral sequence for the group extension
    \[
    	1\to \Gal(k_S/K) \to \Gal(k_S/k) \to \Gal(K/k) \to 1
    \]   
    and the module $A$ degenerates, and then for each $i\geq 0$ we have that
    \begin{equation}\label{eq:HS-isom}
        H^i(\Gal(k_S/k), A) \cong H^i(\Gal(k_S/K), A)^{\Gal(K/k)}.
    \end{equation}

    We first consider the module $A=\mu_{\ell}$, then $K=k(\mu_{\ell})$ and we let $\overline{G}=\Gal(K/Q)$. As $\ell \nmid \Gal(K/Q)$, in both the number field case (by \cite[Cor.~(8.7.3)]{NSW}) and the function field case (by a standard argument using the divisor group), we have that
    %\footnote{The proofs of \cite[Prop.~(8.7.2) and Cor.~(8.7.3)]{NSW} both work for function field case, although they are only stated for number field case. \yuan{Need better explanation! Maybe check if Tate's thesis is a good ref.}}
    \[
        [\calO^{\times}_{K,S}/\ell] = \left[\bigoplus_{\frakP \in S(K)} \F_\ell\right]+[\mu_{\ell}] - [\F_{\ell}]
    \]
    in $K'_0(\F_\ell[\overline{G}])$.
    Then since $[\Cl_S(K)[\ell]]=[\Cl_S(K)/\ell]$ as they are the kernel and the cokernel of the map $\Cl_S(K) \overset{\times \ell}{\longrightarrow} \Cl_S(K)$, by Lemma~\ref{lem:EPChar-mu} we have
    \begin{equation}\label{eq:chieq}
        \sum_{i=0}^2 (-1)^i [H^i(\Gal(k_S/K), \mu_\ell)]\\
        = \left[\bigoplus_{\frakP \in S_{\infty}(K)} \hH^0(K_\frakP, \F_{\ell})\right]-\left[\bigoplus_{\frakP \in S_{\infty}(K)} H^0(K_\frakP, \F_{\ell})\right],
    \end{equation}
    and hence $\varphi_{G, S}(\mu_{\ell})=0$ follows easily by \eqref{eq:HS-isom} and by the arguments in the first paragraph of this subsection. Thus, we have $\Theta\circ \varphi_{G, S}(\mu_{\ell})=0$.
%    \[
%        \chi_{k, S, \ell}(\mu_{\ell}) = [\bigoplus_{\frakp \in S_{\infty}(k)} \hH^0(\calG_{\frakp}, \F_{\ell})] - [\bigoplus_{\frakp \in S_{\infty}(k)} H^0(\calG_{\frakp}, \F_{\ell})].
%    \]
    
    For a general finite module $A\in \Mod_{\ell}(G)$, we again denote $K=k(A, \mu_{\ell})$ and $\overline{G}=\Gal(K/Q)$. 
    We define 
    \begin{eqnarray*}
        \chi: \Mod_{\ell}(\overline{G}) &\to& K'_0(\F_{\ell}[\overline{G}]) \\
        {}M &\mapsto& \sum_{i=0}^2 (-1)^i[H^i(\Gal(k_S/K), M)].
    \end{eqnarray*}
    Because $A$ and $\mu_\ell$ are both trivial $\Gal(k_S/K)$-modules, the pairing
    \begin{eqnarray*}
        \mu_\ell \times \Hom(A', \F_\ell) &\to& \Hom(A', \mu_{\ell})=A \\
        (\zeta, f) &\mapsto& (x\mapsto \zeta^{f(x)})
    \end{eqnarray*}
    defines $\overline{G}$-isomorphisms via the cup product
    \[
        H^i(\Gal(k_S/K), \mu_{\ell}) \otimes_{\Z} \Hom(A', \F_\ell) \overset{\sim}{\longrightarrow} H^i(\Gal(k_S/K), A).
    \]
    So we have $\chi(A)=[A'^\vee]\chi(\mu_\ell)$,
    and hence by \eqref{eq:chieq} we have 
    \[
        \chi(A)=[A'^{\vee}] \left(\left[\bigoplus_{\frakP\in S_{\infty}(K)} \hH^0(K_\frakP, \F_{\ell})\right] -\left[\bigoplus_{\frakP\in S_{\infty}(K)} H^0(K_\frakP, \F_{\ell})\right] \right).
    \]
    If $Q$ is a function field, then \eqref{eq:statement} follows immediately after taking the $\overline{G}$-invariants on both sides above. 
    
    For the rest of the proof we consider the number field case. 
    Let $S^{-}_{\infty}(Q)$ be the set of archimedean places of $Q$ lying below the imaginary places of $K$ if $\ell=2$, and be the set $S_{\infty}(Q)$ if $\ell$ is odd.
    One can check by definition of $\widehat{H}^0$ that for any module $M \in \Mod_{\ell}(\overline{G})$ (for example, $M=A'$ and $M=\F_{\ell}$), we have
    \[
    	 \left[\bigoplus_{\frakP\in S_{\infty}(K)} \hH^0(K_\frakP, M)\right] - \left[\bigoplus_{\frakP\in S_{\infty}(K)} H^0(K_\frakP, M)\right] 
	= \sum_{v \in S^{-}_{\infty}(Q)} -\left[\Ind_{\overline{G}}^{\overline{G}_{v}} M \right],
    \]
    where the group $\overline{G}_{v}$ is the decomposition subgroup $G_v(K/Q)$. Also, note that $(\Ind_{\overline{G}}^{\overline{G}_{v}} \F_\ell) \otimes_{\Z} M \cong \Ind_{\overline{G}}^{\overline{G}_{v}} M$ for any $M\in\Mod_{\ell}(\overline{G})$ and that
    \[
    	(\Ind_{\overline{G}}^{\overline{G}_{v}} M)^{\overline{G}}= H^0(\overline{G}, \Ind_{\overline{G}}^{\overline{G}_{v}} M)= H^0(\overline{G}_{v}, M)=M^{\overline{G}_{v}}.
    \]
    So we have 
 	\begin{eqnarray*}
    		\Theta\circ \varphi_{G,\ell}(A) &=& \# \left( \sum_{i=0}^2 (-1)^i [H^i(G_S(k), A)] -\left[ \bigoplus_{\frakp \in S_{\infty}(k)} \widehat{H}^0(k_{\frakp}, A')\right]^{\vee} + \left[ \bigoplus_{\frakp \in S_{\infty}(k)} H^0(k_{\frakp}, A')\right]^{\vee} \right)^{\Gal(k/Q)} \\
			&=&\#\left( \chi(A)  -  \left[\bigoplus_{\frakP\in S_{\infty}(K)} \hH^0(K_\frakP, A')\right]^{\vee} +\left[\bigoplus_{\frakP\in S_{\infty}(K)} H^0(K_\frakP, A')\right]^{\vee} \right)^{\overline{G}} \\
		&=&\sum_{v\in S_{\infty}^-(Q)} \#\left(-[A'^{\vee}] \left[ \Ind_{\overline{G}}^{\overline{G}_v} \F_{\ell} \right] + \left[ \Ind_{\overline{G}}^{\overline{G}_v} A' \right]^{\vee}\right)^{\overline{G}}\\
		&=& \sum_{v \in S_{\infty}^{-}(Q)} \# \left( -\left[\Ind_{\overline{G}}^{\overline{G}_v} A'^{\vee} \right] + \left[ \Ind_{\overline{G}}^{\overline{G}_v} A' \right]^{\vee}  \right)^{\overline{G}}   \\
		&=& 0.
	\end{eqnarray*}
	Then we finish the proof of Theorem~\ref{thm:EPChar}.
%    which is identically $0$ because
%    \[
%    	\dim_{\F_{\ell}}(A'^{\vee})^{\overline{G}_{v}} = \dim_{\F_\ell} \Hom(A', \F_{\ell})^{\overline{G}_{v}}= \dim_{\F_{\ell}}(A')_{\overline{G}_{v}} = \dim_{\F_{\ell}}(A')^{\overline{G}_{v}},
%    \]  
%    where the first two equalities follow by definition and the last one uses the assumption that $\ell\nmid |\overline{G}_{v}|$. 

\section{Definition and properties of $\B_S(k, A)$}\label{sect:RussianB}

Throughout this section, we assume that $k/Q$ is a finite Galois extension of global fields, and that $S$ is a $k/Q$-closed set of primes of $k$ (not necessarily nonempty or containing $S_{\infty}$).

Let $\frakp$ be a prime of the global field $k$. We denote $\calG_{\frakp}=\calG_{\frakp}(k)$ and $\calT_{\frakp}=\calT_{\frakp}(k)$. Recall that for a $\calG_{\frakp}$-module $A$ of order not divisible by $\Char(k)$, the unramified cohomology group is defined to be
\[
	H_{nr}^i(k_{\frakp}, A) = \im \left( H^i(\calG_{\frakp}/\calT_{\frakp}, A) \to H^i(k_{\frakp}, A) \right),
\]
where the map is the inflation map. Then we consider the following homomorphism of cohomology groups
\begin{equation}\label{eq:RB-def}
	\prod_{\frakp\in S} H^1(k_\frakp, A) \times \prod_{\frakp \not\in S} H^1_{nr}(k_\frakp, A) \hookrightarrow \prod_{\frakp} H^1(k_\frakp, A) \overset{\sim}{\to} \prod_{\frakp} H^1(k_\frakp, A')^{\vee} \to H^1(k, A')^{\vee}.
\end{equation}
The first map is the natural embedding of cohomology groups. The second arrow is an isomorphism because of the local Tate duality theorem \cite[Thms~7.2.6 and 7.2.17]{NSW}. The last map is defined by the Pontryagin dual of the product of restriction map $H^1(k, A') \to H^1(k_{\frakp}, A')$ for each prime $\frakp$ of $k$. In particular, the restriction of the composition of the last two maps in \eqref{eq:RB-def} to the restricted product is the map
	\[
		\prod'_{\frakp} H^1(k_{\frakp}, A) \to H^1(k, A')^{\vee}
	\]
	used in the long exact sequence of Poitou-Tate \cite[(8.6.10)(i)]{NSW}.

\begin{definition}\label{def:RB}
	 For a global field $k$, a set $S$ of primes of $k$, and $A\in \Mod(G_k)$ of order not divisible by $\Char(k)$ , we define
	\[
		\B_S(k, A)= \coker\left( \prod_{\frakp \in S} H^1(k_\frakp, A) \times \prod_{\frakp \not\in S} H^1_{nr}(k_{\frakp}, A) \to H^1(k, A')^{\vee}\right),
	\]
	where the map is the composition of maps in \eqref{eq:RB-def}.
\end{definition}

\begin{remark}
	\begin{enumerate}
		\item When $A$ is a finite $G_{Q}$-module and $S$ is $k/Q$-closed, the maps in \eqref{eq:RB-def} are compatible with the conjugation action of $\Gal(k/Q)$ on cohomology groups, so $\B_S(k,A)$ is naturally a $\Gal(k/Q)$-module. 
		\item Using the language of the Selmer groups, $\B_S(k,A)$ is the Pontryagin dual of the Selmer group of the Galois module $A'$ consisting of elements of $H^1(k, A')$ that have images inside the subgroup
		\[
			\prod_{\frakp \in S} 1 \times \prod_{\frakp \not\in S} \ker \left( H^1(k_{\frakp}, A') \to  H_{nr}^1(k_{\frakp}, A)^{\vee} \right) \subset \prod_{\frakp} H^1(k_{\frakp}, A').
		\]
		under the product of local restriction maps.	\end{enumerate}
\end{remark}

\begin{proposition}\label{prop:RB-trivial-mod}
	If $A=\F_{\ell}$ is the trivial $G_k$-module with $\ell\neq \Char(k)$, then $\B_S(k, \F_{\ell})$ is the Pontryagin dual of the Kummer group
	\[
		V_S(k, \ell)=\ker\left(k^{\times}/ k^{\times \ell} \to \prod_{\frakp \in S} k_{\frakp}^{\times}/ k_{\frakp}^{\times \ell} \times \prod_{\frakp \not\in S}k_{\frakp}^{\times} / U_{\frakp}k_{\frakp}^{\times \ell}\right).
	\]
\end{proposition}

\begin{proof}
	By the class field theory, we have 
	\[
		H^1(k, \mu_\ell)\cong k^{\times}/k^{\times \ell},\quad H^1(k_{\frakp}, \mu_{\ell})\cong k_{\frakp}^{\times}/k_{\frakp}^{\times \ell},\quad \text{and} \quad H_{nr}^1(k_{\frakp}, \F_{\ell})^{\vee} \cong k_{\frakp}^{\times}/U_{\frakp} k_{\frakp}^{\times \ell}.
	\]
	Then the proposition follows directly from Definition~\ref{def:RB}.
\end{proof}

\begin{lemma}\label{lem:les-B}
	Let $k/Q$ be a finite Galois extension of global fields, $T \supseteq S$ be $k/Q$-closed sets of primes of $k$, and $A\in \Mod(\Gal(k_S/Q))$ be of order not divisible by $\Char(k)$. Then we have the following exact sequence that is compatible with the conjugation by $\Gal(k/Q)$
	\[
		H^1(G_S(k), A) \hookrightarrow H^1(G_T(k), A) \to \bigoplus_{\frakp \in T\backslash S} H^1(\calT_\frakp(k), A)^{\calG_\frakp(k)} \to \B_S(k, A) \twoheadrightarrow \B_T(k, A).
	\]
\end{lemma}

\begin{proof}
	We consider the following commutative diagram
	\[\begin{tikzcd}
		& \Sha^1(k, A) \arrow[hook]{d} \arrow[dashed, hook]{dl} & \\
		H^1(G_S, A) \arrow[hook]{r} & H^1(k ,A) \arrow{r} \arrow{d}	& H^1(G_{k_S}, A)^{G_S} \arrow[hook]{d} \\
		\prod'\limits_{\frakp \in S} H^1(k_\frakp, A) \times \prod\limits_{\frakp \not\in S} H^1_{nr}(k_\frakp, A) \arrow[hook]{r} \arrow{d} & \prod'\limits_{\frakp} H^1(k_\frakp, A) \arrow[two heads]{r}\arrow{d} & \bigoplus\limits_{\frakp\not\in S} H^1(\calT_\frakp, A)^{\calG_\frakp} \\
		H^1(k, A')^{\vee} \arrow[two heads]{d} & H^1(k, A')^{\vee} \arrow[equal]{l} \arrow[two heads]{d} & \\
		\B_S(k, A) & \Sha^2(k, A') &
	\end{tikzcd}\]
	The exactnesses of the second row and the third row follow from the Hochschild-Serre spectral sequence, and last arrow in the third row is surjective because of the fact that $H^2_{nr}(\calG_\frakp,A)=0$ as $\calG_{\frakp}/\calT_{\frakp} \simeq \widehat{\Z}$ when $\frakp$ is nonarchimedean and $1$ when $\frakp$ is archimedean. The exact sequence of the first column follows from the definition of $\B_S(k,A)$, and the second column follows from the long exact sequence of Poitou-Tate \cite[(8.6.10)]{NSW}. The right vertical map is injective since $G_{k_S}$ is generated by the inertia groups of primes outside $S$. 
	
	We consider the map $H^1(k, A) \to \oplus_{\frakp \not \in S} H^1(\calT_{\frakp}, A)^{\calG_{\frakp}}$ in the diagonal of the square diagram on the right. Since $H^1(G_S, A)$ is the kernel of this map and $\Sha^1(k,A)$ is contained in this kernel, the top dashed arrow exists and is injective.
	Then by diagram chasing, we have an exact sequence
	\begin{equation}\label{eq:10.7.4(ii)}
		\Sha^1(k, A) \hookrightarrow H^1(G_S, A) \to \prod'_{\frakp \in S} H^1(k_\frakp, A) \times \prod_{\frakp \not\in S} H^1_{nr}(k_\frakp, A) \to H^1(k, A')^{\vee} \twoheadrightarrow \B_S(k, A).
	\end{equation}
	We apply the snake lemma to the following diagram
	\[\begin{tikzcd}
		\prod'\limits_{\frakp\in S} H^1(k_\frakp, A) \times \prod\limits_{\frakp\not\in S}H^1_{nr}(k_\frakp, A) \arrow[hook]{d} \arrow{r} & H^1(k, A')^{\vee}\arrow[equal]{d}\\
		\prod'\limits_{\frakp\in T} H^1(k_\frakp, A) \times \prod\limits_{\frakp\not\in T}H^1_{nr}(k_\frakp, A) \arrow[two heads]{d} \arrow{r} & H^1(k, A')^{\vee} \\
		\bigoplus\limits_{\frakp\in T \backslash S} H^1(\calT_\frakp, A)^{\calG_\frakp} & 
	\end{tikzcd}\]
	where the horizontal map above is from \eqref{eq:10.7.4(ii)}, and we obtain the following exact sequence
	\[
		\frac{H^1(G_S, A)}{\Sha^1(k,A)} \hookrightarrow \frac{H^1(G_T, A)}{\Sha^1(k, A)} \to \bigoplus_{\frakp\in T\backslash S} H^1(\calT_\frakp, A)^{\calG_\frakp} \to \B_S(k, A) \twoheadrightarrow \B_T(k, A).
	\]
	Note that the inflation map $H^1(G_S, A)\hookrightarrow H^1(G_T, A)$ maps the submodule $\Sha^1(k, A)$ to itself, because $\Sha^1(k, A)$ is the kernel of $H^1(G_*, A) \to \prod_{\frakp} H^1(k_{\frakp}, A)$ for $*=S, T$. Therefore we proved the exact sequence in the lemma, and it is naturally compatible with the conjugation action by $\Gal(k/Q)$.
\end{proof}

\begin{proposition} \label{prop:ShainB}
	Let $k/Q$ be a finite Galois extension of global fields and $S$ a $k/Q$-closed set of primes of $k$. Then for any $A\in \Mod(\Gal(k_S/Q))$ of order not divisible by $\Char(k)$, we have the following inequality of elements in $K'_0(\Gal(k/Q))$
	\[
		[\Sha_S^2(k, A)] \leq [\B_S(k, A)].
	\]
\end{proposition}

\begin{proof}
	We consider the commutative diagram
	\begin{equation}\label{eq:diag}
	\begin{tikzcd}
		H^1(G_S, A) \hookrightarrow H^1(k, A) \to H^1({k_S}, A)^{G_S} \arrow["\alpha"]{r} & H^2(G_S, A) \arrow["\beta"]{r} \arrow["\rho_S"]{d} & H^2(k, A) \arrow["\rho"]{d} \\
		 & \prod\limits_{\frakp\in S} H^2(k_\frakp, A) \arrow[hook]{r} & \prod\limits_{\frakp} H^2(k_\frakp, A)
	\end{tikzcd}
	\end{equation}
	where the first row is the Hochschild-Serre long exact sequence of $1\to G_{k_S} \to G_k \to G_S \to 1$. Because $\im \alpha =\ker \beta \subseteq \ker \rho\circ \beta =\ker \rho_S=\Sha_S^2(k, A)$, we have an exact sequence 
	\[
		H^1(G_S, A) \hookrightarrow H^1(k, A) \to H^1(G_{k_S}, A)^{G_S} \to \Sha_S^2(k, A) \twoheadrightarrow \beta(\Sha_S^2(k, A)).
	\]
	Comparing this exact sequence to Lemma~\ref{lem:les-B} using $T=\left\{ \text{all primes} \right\}$, we have
	\[\begin{tikzcd}
		H^1(k, A) \arrow{r}\arrow[equal]{d} & H^1({k_S}, A)^{G_S} \arrow{r}\arrow[hook]{d} & \Sha_S^2(k, A) \arrow[two heads]{r} & \beta(\Sha_S^2(k, A)) \\
		H^1(k, A) \arrow{r} & \bigoplus\limits_{\frakp\not\in S} H^1(\calT_\frakp, A)^{\calG_\frakp} \arrow{r} & \B_S(k, A) \arrow[two heads]{r} & \B_{\{\text{all primes}\}}(k, A).
	\end{tikzcd}\]
	So by the vertical injection above, we have 
	$\ker \beta \hookrightarrow N:=\ker(\B_S(k,A) \to \B_{\{\text{all primes}\}}(k,A))$. By the diagram in \eqref{eq:diag}, we have $\beta(\ker \rho_S) \subseteq \ker \rho$, which means $\beta(\Sha_S^2(k, A)) \subseteq \Sha^2(k, A)$. Also, note that by Definition~\ref{def:RB} and the Poitou-Tate duality we have $\B_{\{\text{all primes}\}}(k,A) = \Sha^1(k, A')^{\vee} \cong \Sha^2(k,A)$.  Then we consider the two short exact sequence
	\[
		0 \longrightarrow \ker \beta \longrightarrow \Sha_S^2(k, A) \longrightarrow \beta(\Sha_S^2(k,A)) \longrightarrow 0,
	\]
	\[
		0 \longrightarrow N \longrightarrow \B_S(k, A) \longrightarrow \B_{\{\text{all primes}\}}(k, A) \longrightarrow 0,
	\]
	Because $\ker \beta \hookrightarrow N$, $\beta(\Sha_S^2(k, A)) \hookrightarrow \B_{\{\text{all primes}\}}(k, A)$
	and every map respect the conjugation action by $\Gal(k/Q)$, we have the desired inequality $[\Sha_S^2(k, A)] \leq [\B_S(k,A)]$.
\end{proof}

\begin{remark}\label{rmk:ShainB}
	When $A=\F_{\ell}$ is the trivial module, then $\B_{\{\text{all primes}\}}(k,\F_{\ell})$ vanishes \cite[Prop.~9.1.12(ii)]{NSW}, so there is an embedding $\Sha_S^2(k, \F_\ell) \hookrightarrow \B_S(k, \F_\ell)$. However, for an arbitrary $A$, Proposition~\ref{prop:ShainB} does not give such an embedding.
\end{remark}

\begin{lemma}\label{lem:B&Sha}
	Let $k$ be a global field and $S$ a set of primes of $k$ containing $S_{\infty}(k)$. Then for any $A\in \Mod_S(G_S(k))$ of order not divisible by $\Char(k)$, we have $\Sha_S^1(k, A') \cong \B_S(k, A)^{\vee}$.
\end{lemma}

\begin{proof}
	We consider the following commutative diagram
	\[\begin{tikzcd}
		\prod\limits_{\frakp} H^1(k_\frakp, A') \arrow{r} \arrow["\sim"]{d} & \prod\limits_{\frakp \in S} H^1(k_\frakp, A') \times \prod\limits_{\frakp\not\in S} H^1(\calT_\frakp, A')^{\calG_\frakp} \arrow["\sim"]{d} \\
		\prod\limits_{\frakp} H^1(k_\frakp, A)^{\vee} \arrow{r} & \prod\limits_{\frakp\in S} H^1(k_\frakp, A)^{\vee} \times \prod\limits_{\frakp\not\in S} H_{nr}^1(k_{\frakp}, A)^{\vee},
	\end{tikzcd}\]
	where the two vertical arrows are isomorphisms by the Tate local duality theorem and its consequence that $H^1(\calT_\frakp, A')^{\calG_\frakp} \overset{\sim}{\longrightarrow} H^1_{nr} (k_\frakp, A)^{\vee}$ when $A$ is unramified at $\frakp$ and $\#\tor(A)$ is prime to the characteristic of the residue field of $k_{\frakp}$ (see the proof of \cite[Thm.~7.2.15]{NSW}). Then by definition, we have
	\begin{eqnarray*}
		\B_S(k, A)^{\vee} &=& \ker\left(H^1(k, A') \to \prod_{\frakp\in S} H^1(k_{\frakp}, A)^{\vee} \times \prod_{\frakp\not\in S} H^1_{nr}(k_\frakp, A)^{\vee}\right) \\
		&=& \ker \left(H^1(k, A') \to \prod_{\frakp\in S} H^1(k_\frakp, A') \times \prod_{\frakp\not\in S} H^1(\calT_\frakp, A')^{\calG_p} \right).
	\end{eqnarray*}
	So by applying the snake lemma to the following commutative diagram
	\[\begin{tikzcd}
		\Sha_S^1(k, A') \arrow[hook]{r} & H^1(G_S, A') \arrow{r} \arrow[hook]{d} & \prod\limits_{\frakp \in S} H^1(k_\frakp, A') \arrow[hook]{d} \\
		\B_S(k, A)^{\vee} \arrow[hook]{r} & H^1(k, A') \arrow{d} \arrow{r} &\prod\limits_{\frakp \in S} H^1(k_\frakp, A') \times \prod\limits_{\frakp \not \in S} H^1(\calT_\frakp, A')^{\calG_\frakp} \arrow[two heads]{d}\\
		& H^1({k_S}, A')^{G_S} \arrow[hook]{r} & \prod\limits_{\frakp\not\in S} H^1(\calT_\frakp, A')^{\calG_\frakp},
	\end{tikzcd}\]
	we obtain the desired isomorphism $\Sha_S^1(k, A') \overset{\sim}{\longrightarrow} \B_S(k, A)^{\vee}$. 
\end{proof}

\begin{corollary}\label{cor:Bisfinite}
	For any set $S$ of primes of a global field $k$ and any $A\in \Mod(G_S(k))$ of order not divisible by $\Char(k)$, we have that $\B_S(k,A)$ is finite.\end{corollary}

\begin{proof}
	Define $T=S\cup S_{\infty}(k) \cup S_{|A|}(k)$. By applying Lemma~\ref{lem:les-B}, we have
	\begin{equation}\label{eq:es}
		\bigoplus_{\frakp\in T\backslash S} H^1(\calT_{\frakp}, A)^{\calG_\frakp} \to \B_S(k, A) \twoheadrightarrow \B_T(k, A).
	\end{equation}
	Since $A\in \Mod_T(G_T)$, by Lemma~\ref{lem:B&Sha} and \cite[Thm.~8.6.4]{NSW}, we have that $\B_T(k, A)^{\vee}\cong \Sha_T^1(k, A')$ is finite. Also note that $H^1(k_\frakp, A)$ is finite \cite[Thm.~7.1.8(iv)]{NSW} and $H^1(\calT_{\frakp}, A)^{\calG_\frakp}$ is a quotient of $H^1(k_{\frakp}, A)$. 
	Thus, the direct product $\prod_{\frakp\in T\backslash S} H^1(\calT_\frakp, A)^{\calG_\frakp}$ is finite, and hence the corollary follows by \eqref{eq:es}.
\end{proof}

\section{Determination of $\delta_{k/Q, S}(A)$}\label{sect:compute-delta}

	\begin{definition}\label{def:delta}
		Let $k/Q$ be a finite Galois extension of global fields, $S$ a finite $k/Q$-closed set of primes of $k$, $\ell\neq \Char(k)$ a prime integer not dividing $[k:Q]$, and $A \in \Mod_{\ell}(\Gal(k_S/Q))$. We define 
		\[
			\delta_{k/Q, S}(A)=\dim_{\F_\ell} H^2(G_S(k), A)^{\Gal(k/Q)} - \dim_{\F_\ell}H^1(G_S(k), A)^{\Gal(k/Q)}.
		\]
	\end{definition}
	
	We will use the notation and assumption in Definition~\ref{def:delta} throughout this section.
	When $\ell \in \N(S)$ and $S_{\infty}(k)\subset S$, by rewriting the formula, we have our first case for which $\delta_{k/Q, S}(A)$ can be determined.
	
	\begin{proposition}\label{prop:EPChar-delta}
		Assume $\ell \in \NN(S)$ and $S\supset S_{\infty}(k)$ is nonempty. Then 
		\[
			\delta_{k/Q, S}(A)=\log_{\ell}(\chi_{k/Q, S}(A))-\dim_{\F_\ell}A^{\Gal(k_S/Q)}.
		\]
	\end{proposition}

	So in this section, we will consider the cases that are not covered by Proposition~\ref{prop:EPChar-delta}. In \S~\ref{subsect:functionfield}, we will deal with the case that $Q$ is a function field and $S=\O$, and obtain a formula for $\delta_{k/Q, \O}(A)$ (Proposition~\ref{prop:ff-delta}). Then in \S~\ref{subsect:numberfield}, we will give an upper bound of $\delta_{k/Q,S}(A)$ when $k$ is a number field with $S_{\ell}(k) \cup S_{\infty}(k) \not \subset S$ (Proposition~\ref{prop:nf-delta}). When $k=Q$, Theorem~\ref{thm:fin-pres} follows by Propositions~\ref{prop:EPChar-delta} and \ref{prop:nf-delta}.

	\begin{proof}[Proof of Theorem~\ref{thm:fin-pres}]
		We denote $G=G_S(k)$. Let $A$ be a finite simple $G$-module and $\ell$ denote the exponent of $A$. Since $\widehat{H}^0(k_{\frakp}, A')$ is naturally a quotient of $H^0(k_{\frakp},A')$ for each $\frakp \in S_{\infty}(k)$, we have $\log_{\ell}\chi_{k/k, T}(A)\leq 0$ for $T=S\cup S_{\ell}(k) \cup S_{\infty}(k)$. When $S\supset S_{\ell}(k) \cup S_{\infty}(k)$, Proposition~\ref{prop:EPChar-delta} shows that $\delta_{k/k,S}(A)\leq 0$.
		It follows by definition of $\epsilon_{k/k, S}(A)$ in Proposition~\ref{prop:nf-delta} that $\epsilon_{k/k,S}(A) \leq [k:\Q]\dim_{\F_\ell} A$. Also, note that, when $S \not \supset S_{\ell}(k) \cup S_\infty(k)$, we have $\dim_{\F_\ell}(A')^{G_T(k)}-\dim_{\F_\ell} A^{G_S(k)} \leq 0$ because $A$ cannot be $\mu_{\ell}$ if $\mu_\ell\not \subset k$.
		So Proposition~\ref{prop:nf-delta} shows that $\delta_{k/k,S}(A)\leq [k:\Q] \dim_{\F_\ell} A$, and hence the theorem follows by Proposition~\ref{prop:min-gen}.
	\end{proof}

\subsection{Function field case with $S=\O$}\label{subsect:functionfield}

	\begin{proposition}\label{prop:ff-delta}
		Assume $k$ and $Q$ are function fields. Let $g=g(k)$ be the genus of the curve corresponding to $k$. Then we have
		\begin{enumerate}
			\item\label{item:genus-0} If $g=0$, then $\delta_{k/Q, \O}(A)=-\dim_{\F_\ell}A_{\Gal(k_{\O}/Q)}$.
			
			\item\label{item:pos-genus} If $g>0$, then
			\[
				\delta_{k/Q, \O}(A)=\dim_{\F_\ell}(A')^{\Gal(k_{\O}/Q)} - \dim_{\F_\ell} A^{\Gal(k_{\O}/Q)}.
			\]
		
		\end{enumerate}
	\end{proposition}
	
	\begin{proof}
		When $g=0$, we have $G_{\O}(k)\cong \widehat{\Z}$ by \cite[Cor.~10.1.3(i)]{NSW}. So $H^2(G_{\O}, A)=0$ as $\widehat{\Z}$ and $H^1(G_{\O},A)\cong A_{G_{\O}}$ by \cite[Prop.~1.7.7(i)]{NSW}. Then we see that
		\[
			\delta_{k/Q, \O}(A)=-\dim_{\F_\ell}(A_{G_{\O}})^{\Gal(k/Q)}=-\dim_{\F_\ell}(A_{G_{\O}})_{\Gal(k/Q)}=-\dim_{\F_\ell}A_{\Gal(k_{\O}/Q)},
		\]
		where the second equality uses $\ell \nmid [k:Q]$, so we proved \eqref{item:genus-0}.
	
		For the rest, we assume $g>0$. Let $\kappa$ be the finite field of constants of $k$ and $C=\Gal(\overline{\kappa}/\kappa)\cong \widehat{\Z}$. Then there exists an exact sequence for each $j$
		\begin{equation}\label{eq:frob-1}
		\begin{tikzcd}
			H^j(G_{\O}(k\overline{\kappa}), A)^C \hookrightarrow H^j(G_{\O}(k\overline{\kappa}), A) \arrow["\Frob-1"]{r} &H^j(G_{\O}(k\overline{\kappa}), A) \twoheadrightarrow H^j(G_{\O}(k\overline{\kappa}),A)_C,
		\end{tikzcd}
		\end{equation}
		where $\Frob$ is the Frobenious action on the cohomology groups defined by conjugation.			
		Note that $\Gal(k\overline{\kappa}/Q)$  acts on cohomology groups in \eqref{eq:frob-1}, and 
		\[
			1 \to C=\Gal(k\overline{\kappa}/k) \to \Gal(k\overline{\kappa}/Q) \to \Gal(k/Q) \to 1
		\]
		is a central group extension because $\Gal(k/Q)$ acts trivially on the generator $\Frob$ of $C$. So the map $\Frob-1$ in \eqref{eq:frob-1} respects the $\Gal(k\overline{\kappa}/Q)$ actions.
		It follows that $H^j(G_{\O}(k\overline{\kappa}),A)^C$ and $H^j(G_{\O}(k\overline{\kappa}), A)_C$ are in the same class in $K_0'(\F_{\ell}[\Gal(k\overline{\kappa}/Q)])$, and hence they are in the same class in $K'_0(\F_{\ell}[\Gal(k/Q)])$. Because $\ell \nmid [k:Q]$ implies $\F_{\ell}[\Gal(k/Q)]$ is semisimple, we have 
		\begin{equation}\label{eq:Co-inv}
			H^j(G_{\O}(k\overline{\kappa}), A)^C \simeq H^j(G_{\O}(k\overline{\kappa}), A)_C
		\end{equation}	
		as $\Gal(k/Q)$-modules.
		Therefore,
		\[
			H^1(C, H^j(G_{\O}(k\overline{\kappa}),A)) \cong H^j(G_{\O}(k\overline{\kappa}),A)_{C} \simeq H^0(C, H^j(G_{\O}(k\overline{\kappa}),A))
		\]
		as $\Gal(k/Q)$-modules.
		Then we consider the Hochschild-Serre spectral sequence
		\[
			E_2^{ij}=H^i(C, H^j(G_{\O}(k\overline{\kappa}),A)) \Rightarrow H^{i+j}(G_{\O}(k),A).
		\]
		As $C$ has cohomological dimension 1, $E_2^{ij}=0$ for each $i>1$, and hence by \cite[Lem.~2.1.3(ii)]{NSW} we have the following exact sequence for every $j\geq 1$
		\begin{equation}\label{eq:ses-spectral}
			H^1(C, H^{j-1}(G_{\O}(k\overline{\kappa}),A)) \hookrightarrow H^j(G_{\O}(k), A) \twoheadrightarrow H^0(C, H^j(G_{\O}(k\overline{\kappa}),A)).
		\end{equation}
		Note that $G_{\O}(k)$ has strict cohomological dimension 3 by \cite[Cor.~10.1.3(ii)]{NSW}. Then as $\ell \nmid [k:Q]$, taking $\Gal(k/Q)$-invariants is exact on \eqref{eq:ses-spectral}, and by computing the alternating sum of \eqref{eq:ses-spectral} for $j=1,2,3$, we have
		\begin{eqnarray*}
			\sum_{j=1}^3 (-1)^j\dim_{\F_\ell}H^j(G_{\O}(k), A)^{\Gal(k/Q)} &=& - \dim_{\F_\ell}H^1(C, H^0(G_{\O}(k\overline{\kappa}),A))^{\Gal(k/Q)}\\
			&=&- \dim_{\F_\ell}H^0(C, H^0(G_{\O}(k\overline{\kappa}),A))^{\Gal(k/Q)}\\
			&=&-\dim_{\F_\ell}H^0(\Gal(k_{\O}/Q), A).
		\end{eqnarray*}
		Also, \cite[Cor.~10.1.3(ii)]{NSW} shows that $G_{\O}(k)$ is a Poincar\'e group of dimension 3 with dualizing module $\mu$, so we have a functorial isomorphism $H^3(G_{\O}(k), A)\cong H^0(G_{\O}(k), A')^{\vee}$. Combining the above computations, we see that
		\begin{eqnarray*}
			\delta_{k/Q, \O}(A)&=& \dim_{\F_\ell}(H^0(G_{\O}(k),A')^{\vee})^{\Gal(k/Q)} - \dim_{\F_\ell}H^0(\Gal(k_{\O}/Q), A)\\
			&=& \dim_{\F_\ell}H^0(G_{\O}(k),A')^{\Gal(k/Q)} - \dim_{\F_\ell}H^0(\Gal(k_{\O}/Q), A)\\
			&=& \dim_{\F_\ell}(A')^{\Gal(k_{\O}/Q)} - \dim_{\F_\ell} A^{\Gal(k_{\O}/Q)},
		\end{eqnarray*}
		where the second equality is because the $\Gal(k/Q)$-invariants of $M$ and $M^{\vee}$ have the same dimension for any $M\in \Mod_{\ell}(\Gal(k/Q))$.
	\end{proof}

\subsection{Number field case with $S_{\ell}\cup S_{\infty}\not\subset S$}\label{subsect:numberfield}

	\begin{proposition}\label{prop:nf-delta}
		Assume $k$ and $Q$ are number fields. Let $T=S\cup S_{\ell}(k) \cup S_{\infty}(k)$. Then we have 
		\begin{equation*}
			\delta_{k/Q, S}(A)\leq  \log_{\ell} (\chi_{k/Q, T}(A)) +\dim_{\F_{\ell}}(A')^{\Gal(k_T/Q)}-\dim_{\F_{\ell}} A^{\Gal(k_S/Q)} + \epsilon_{k/Q, S}(A),
		\end{equation*}
		where $\epsilon_{k/Q,S}(A)=-\sum_{v \in I} \log_\ell \|\#A\|_v$ \footnote{$\|x\|_v=q^{-\ord_v(x)}$ where $q$ is the cardinality of the residue field of $v$ and $\ord_v$ is the additive valuation with value group $\Z$.} with 
	\[
		I= \{v\in S_{\ell}(Q) \text{ such that }S_v(k) \not\subset S\}.
	\]
	In particular, when $S=\O$, the equality holds if and only if $\Sha_{\O}^2(k, A)$ and $\B_{\O}(k,A)$ are in the same class of $K'_0(\F_{\ell}[\Gal(k/Q)])$.
%		\[
%			\epsilon=\begin{cases}
%			\sum\limits_{v\in S_{\ell}(Q)} \ord_v|A| & \text{if } S_{\ell}(k) \not\subset S\\
%			0 & \text{otherwise}.
%		\end{cases}
%		\]
	\end{proposition}
	
	\begin{proof}
		First of all, by definition of $\Sha^2_S$ and Proposition~\ref{prop:ShainB}, we have the following inequalities of elements in $K'_0(\F_{\ell}[\Gal(k/Q)])$
		\begin{equation}\label{eq:nf-1}
			[H^2(G_S, A)] \leq [\Sha_S^2(k,A)] + \left[\bigoplus_{\frakp \in S} H^2(k_{\frakp}, A) \right] 
			\leq [\B_S(k,A)] + \left[ \bigoplus_{\frakp \in S} H^2(k_{\frakp}, A) \right].
		\end{equation}
		By applying Lemma~\ref{lem:les-B}, we have
		\begin{equation}\label{eq:nf-2}
			[\B_S(k, A)] - [H^1(G_S, A)] = [\B_T(k, A)] - [H^1(G_T, A)] +\left[ \bigoplus_{\frakp \in T\backslash S} H^1(\calT_{\frakp}, A)^{\calG_{\frakp}} \right]. 
		\end{equation}
		Since $T$ contains $S_{\ell}(k) \cup S_{\infty}(k)$, it follows that $[\B_T(k,A)]= [\Sha_T^2(k, A)]$ by Lemma~\ref{lem:B&Sha} and the Poitou-Tate duality theorem. Also, note that the long exact sequence of Poitou-Tate \cite[(8.6.10)]{NSW} induces an exact sequence 
		\[
			\Sha_T^2(k, A) \hookrightarrow H^2(G_T, A) \to \bigoplus_{\frakp \in T} H^2(k_{\frakp}, A) \twoheadrightarrow H^0(G_T, A')^{\vee}.
		\]
		Therefore we have 
		\begin{equation}\label{eq:nf-3}
			[\B_T(k,A)]=[H^2(G_T, A)] + [H^0(G_T, A')^{\vee}] - \left[\bigoplus_{\frakp \in T} H^2(k_{\frakp}, A) \right].
		\end{equation}
	Combining \eqref{eq:nf-1}, \eqref{eq:nf-2} and \eqref{eq:nf-3}, we have
	\begin{eqnarray*}
		[H^2(G_S,A)] - [H^1(G_S, A)] &\leq & [H^2(G_T, A)] - [H^1(G_T, A)] + [H^0(G_T, A')^{\vee}] \\ 
		&& +\left[ \bigoplus_{\frakp \in T \backslash S} H^1(\calT_{\frakp}, A)^{\calG_\frakp} \right] - \left[ \bigoplus_{\frakp \in T\backslash S} H^2(k_{\frakp}, A)\right].
	\end{eqnarray*}
	The dimension of $\Gal(k/Q)$-invariant of the left-hand side above is $\delta_{k/Q, S}(A)$. On the right-hand side, the dimension of $\Gal(k/Q)$-invariant of $[H^2(G_T,A)]-[H^1(G_T,A)]$ is 
	\[
		\log_{\ell}(\chi_{k/Q, T}(A)) -\dim_{\F_\ell} H^0(G_T,A)^{\Gal(k/Q)} = \log_{\ell}(\chi_{k/Q, T}(A)) -\dim_{\F_\ell} A^{\Gal(k_S/Q)}
	\]
	by the definition of $\chi_{k/Q,T}$ and the assumption that $A$ is a $\Gal(k_S/Q)$-module. Also, 
	\[
		\dim_{\F_\ell} \left(H^0(G_T, A')^{\vee}\right)^{\Gal(k/Q)} = \dim_{\F_{\ell}} H^0(G_T, A')^{\Gal(k/Q)} = \dim_{\F_\ell}(A')^{\Gal(k_T/Q)}.
	\]
	So to prove the inequality in the proposition, it suffices to show 
	\begin{equation}\label{eq:nf-4}
		\epsilon_{k/Q,S}(A) = \dim_{\F_\ell}\left(\bigoplus_{\frakp \in T\backslash S} H^1(\calT_{\frakp}, A)^{\calG_{\frakp}} \right)^{\Gal(k/Q)} - \dim_{\F_\ell}\left(\bigoplus_{\frakp \in T\backslash S} H^2(k_{\frakp}, A) \right)^{\Gal(k/Q)}.
	\end{equation}
	
	We first consider $v\in S_{\infty}(Q)$ such that $S_v(k) \not \subset S$. Since $\calT_\frakp(k)=\calG_\frakp(k)$, we know that $H^1(\calT_\frakp, A)^{\calG_\frakp}=H^1(k_{\frakp}, A)$ for each $\frakp \in S_v(k)$. For $i=1, 2$, we have
	\[
		\left(\bigoplus_{\frakp \in S_v(k)} H^i(k_{\frakp}, A)\right)^{\Gal(k/Q)} = \left( \Ind_{\Gal(k/Q)}^{\Gal_{\frakp}(k/Q)}  H^i(k_{\frakp}, A)\right)^{\Gal(k/Q)} = H^i(k_\frakp, A)^{\Gal_\frakp(k/Q)} = H^i(Q_v, A),
	\]
	where the second equality uses the Shapiro's lemma and the last one follows by the assumption that $\ell \nmid [k:Q]$ and the same argument for \eqref{eq:ses-spectral}. Therefore, we have 
	\begin{eqnarray*}
		&& \dim_{\F_\ell} \left(\bigoplus_{\frakp \in S_v(k)} H^1(\calT_{\frakp}, A)^{\calG_{\frakp}} \right)^{\Gal(k/Q)} - \dim_{\F_\ell}\left( \bigoplus_{\frakp \in S_v(k)} H^2(k_{\frakp}, A)\right)^{\Gal(k/Q)} \\
		&=& \dim_{\F_\ell} H^1(Q_v, A) - \dim_{\F_\ell} H^2(Q_v, A),
	\end{eqnarray*}
	which always equals $0$ since $Q_v$ is a cyclic group (\cite[Prop.~1.7.6]{NSW}).

	Finally, we consider $v\in S_{\ell}(Q)$ such that $S_v(k)\not\subset S$. 
	Because $\calG_{\frakp}/\calT_{\frakp}$ is procyclic, we have that $H^1(\calG_{\frakp}/\calT_{\frakp}, A) \cong A_{\calG_{\frakp}/\calT_{\frakp}}$; and by the same argument from \eqref{eq:frob-1} to \eqref{eq:Co-inv}, we have an isomorphism $H^1(\calG_{\frakp} / \calT_{\frakp}, A) \simeq A^{\calG_\frakp/\calT_{\frakp}} = A^{\calG_{\frakp}}$ that is compatible with the conjugation action by $\Gal_\frakp(k/Q)$. So we see that
	\begin{eqnarray}
		\dim_{\F_\ell} \left( \bigoplus_{\frakp \in S_\frakp(k)} H^1(\calG_{\frakp}/ \calT_{\frakp}, A)\right)^{\Gal(k/Q)} 
		&=& \dim_{\F_\ell} \left( \Ind_{\Gal(k/Q)}^{\Gal_{\frakp}(k/Q)} H^1(\calG_{\frakp}/\calT_{\frakp}, A) \right)^{\Gal(k/Q)} \nonumber \\
		&=& \dim_{\F_\ell} H^1(\calG_{\frakp}/\calT_{\frakp}, A)^{\Gal_{\frakp}(k/Q)} \nonumber\\
		&=& \dim_{\F_{\ell}} A^{\calG_\frakp(Q)}. \label{eq:nf-6}
	\end{eqnarray}
	Therefore, we compute
	\begin{eqnarray*}
		&& \dim_{\F_\ell} \left(\bigoplus_{\frakp \in S_v(k)} H^1(\calT_{\frakp}, A)^{\calG_{\frakp}} \right)^{\Gal(k/Q)} - \dim_{\F_\ell} \left( \bigoplus_{\frakp \in S_v(k)} H^2(k_{\frakp}, A) \right)^{\Gal(k/Q)}\\
		&=& \dim_{\F_\ell} \left( \bigoplus_{\frakp \in S_v(k)} H^1(k_{\frakp}, A) \right)^{\Gal(k/Q)} -\dim_{\F_\ell} \left( \bigoplus_{\frakp \in S_v(k)} H^2(k_{\frakp}, A) \right)^{\Gal(k/Q)} - \dim_{\F_{\ell}} A^{\calG_v(Q)}\\
		&=& \dim_{\F_{\ell}} H^1(k_\frakp, A)^{\Gal_\frakp(k/Q)} - \dim_{\F_\ell} H^2(k_\frakp, A)^{\Gal_\frakp(k/Q)} - \dim_{\F_\ell} A^{\calG_v(Q)}\\
		&=& \dim_{\F_\ell} H^1(Q_v, A) - \dim_{\F_\ell} H^2(Q_v, A) - \dim_{\F_\ell} A^{\calG_v(Q)} \\
		&=& -\log_{\ell}\|\#A\|_v.
	\end{eqnarray*}
	The first equality above uses \eqref{eq:nf-6} and the exact sequence $H^1(\calG_{\frakp}/\calT_{\frakp}, A) \hookrightarrow H^1(k_{\frakp},A)\twoheadrightarrow H^1(\calT_{\frakp}, A)^{\calG_{\frakp}}$, the second one uses the Shapiro's lemma, the third one uses the assumption that $\ell \nmid [k:Q]$, and the last one uses the Tate's local Euler-Poincar\`e Characteristic formula \cite[Thm.~7.3.1]{NSW}. Then we proved \eqref{eq:nf-4}.
	
	When $S=\O$, we have $\Sha_{\O}^2(k, A) = H^2(G_{\O}, A)$, so the first inequality in \eqref{eq:nf-1} is an equality, and hence we have the last statement in the proposition.
	\end{proof}

\section{Proof of the main theorem}\label{sect:proof-main}

	In this section, we will prove Theorem~\ref{thm:main}. We assume that $\Gamma$ is a nontrivial finite group, $Q=\Q$ or $\F_q(t)$ with $\gcd(q,|\Gamma|)=1$, and let $k/Q$ be a Galois extension with $\Gal(k/Q)\simeq \Gamma$.
	By Theorem~\ref{thm:finite-GC}, $G_{\O}(k)^\calC$ is a finite $\Gamma$-group when $\calC$ is finite, so for a sufficiently large $n$ there is a $\Gamma$-presentation $F_n(\Gamma) \to G_{\O}(k)^{\calC}$. 
	In \S\ref{ss:const-G}, we construct a finitely generated $\Gamma$-quotient $G$ of $G_{\O}(k)$ such that $G^{\calC}\simeq G_{\O}(k)^{\calC}$ as $\Gamma$-groups. With the help of the group $G$, we employ the cohomology of $G_{\O}$ to compute the multiplicities in a pro-$\calC$ admissible $\Gamma$-presentation of $G_{\O}(k)^{\calC}$. In \S\ref{ss:determine-mult}, we compute the multiplicities $m_{\ad}^{\calC}(n, \Gamma, G_{\O}(k)^{\calC}, A)$, and then the compute multiplicities $m_{\ad}^{\calC}(n, \Gamma, G_{\O,\infty}(k)^{\calC}, A)$ for a finite simple $G_{\O,\infty}(k)^{\calC}\rtimes \Gamma$-module $A$. Using these multiplicities, finally in \S\ref{ss:presentation}, we show that the kernel of a pro-$\calC$ admissible $\Gamma$-presentation $\calF_n(\Gamma)^{\calC} \to G_{\O, \infty}(k)^{\calC}$ can be normally generated by elements $\{r^{-1}\gamma(r)\}_{r\in X, \gamma \in \Gamma}$ with $X$ a subset of $\calF_n(\Gamma)$ of cardinality $n+1$.

	Note that in Theorem~\ref{thm:main}, $k/Q$ is assumed to be split completely at $\infty$, and the $\Gamma$-groups in $\calC$ are of order prime to $|\mu(Q)|$, $|\Gamma|$ and $\Char (Q)$. However, in the proof, we do not use these assumptions until \S\ref{ss:determine-mult}.
	So right now, we only assume that $k/Q$ is a Galois field extension with $\Gal(k/Q) \simeq \Gamma$ and that $\calC$ is a finite set of isomorphism classes of finite $\Gamma$-groups of order prime to $|\Gamma|$.

\subsection{Construction of a specific finitely generated quotient of $G_{\O}(k)$}\label{ss:const-G}

	Because $G_{\O}(k)^{\calC}$ is finite, when $n$ is sufficiently large, there exists a $\Gamma$-equivariant surjection $\pi: F_n(\Gamma) \to G_{\O}(k)^{\calC}$, where $F_n(\Gamma)$ is the free profinite $\Gamma$-group defined in Section~\ref{sect:pre-Gamma}. Then $\pi$ factors through $\pi^{\calC}: F_n(\Gamma)^{\calC} \to G_{\O}(k)^{\calC}$ as defined in Definition~\ref{def:pro-C-map}.

	\begin{lemma}\label{lem:split-in-C}
		Use the notation above. If $A$ is a finite simple $G_{\O}(k)^{\calC}\rtimes \Gamma$-module with $m(\pi^{\calC}, \Gamma, G_{\O}(k)^{\calC}, A)>0$, then $A\rtimes G_{\O}(k)^{\calC} \in \overline{\calC}$.
	\end{lemma}
	
	\begin{proof}
		We denote $G_{\O}(k)^{\calC}$ by $G_0$ for convenience purposes. If $m(\pi^{\calC}, \Gamma, G_0, A)>0$, then there is a $\Gamma$-group extension 
		\[
			1 \to A \to H \overset{\varpi}{\longrightarrow} G_0 \to 1,
		\]
		such that $H$ is a quotient of $F_n^{\calC}$, and so $H \in \overline{\calC}$. We let $E$ be the fiber product $H\times_{G_0} H$ defined by $\varpi$, i.e. 
		$E= \{(x,y) \in H\times H \mid \varpi(x)=\varpi(y)\}$.
		Note that $E$ is a subgroup of $H\times H$, so is in $\overline{\calC}$.
		There are a natural diagonal embedding $H\hookrightarrow E$ mapping $x$ to $(x,x)$, and a normal subgroup $\{(a,1) \mid a\in A\}$ of $E$ that is isomorphic to $A$. From this, one can check that $E\simeq A \rtimes H$, where the $H$ action on $A$ factors through $\varpi(H)=G_0$. So by taking the quotient map $\varpi$ on the subgroup $H$ of $E$, we obtain that $A\rtimes G_0$ is a quotient of $E$, and therefore we proved the lemma.	
	\end{proof}
	
	Now we fix a finite simple $G_{\O}(k)^{\calC}\rtimes \Gamma$-module $A$ with $m(\pi^{\calC}, \Gamma, G_{\O}(k)^{\calC}, A)>0$, and construct the desired quotient of $G_{\O}(k)$ for $A$. We let $\varphi_0$ denote the quotient map $G_{\O}(k) \to G_{\O}(k)^{\calC}$, and again let $G_0$ denote $G_{\O}(k)^{\calC}$. We define $G_1$ to be the quotient of $G_{\O}(k)$ satisfying the following $\Gamma$-group extension
	\begin{equation}\label{eq:ext-0}
		1 \to A^{m(\varphi_0, \Gamma, G_0, A)} \to G_1 \overset{\varpi_0}{\longrightarrow} G_0 \to 1.
	\end{equation}
	By definition of the multiplicities, $G_1$ is well-defined. Since $G_1$ is a quotient of $G_{\O}(k)$, we have that $G_1^{\calC}$ is exactly $G_0$. Then we claim that the extension \eqref{eq:ext-0} is ``completely nonsplit'' (that is, if a subgroup of $G_1$ maps surjectively onto $G_0$, then it has to be $G_1$ itself). Indeed, if it's not completely nonsplit, then $G_1$ has a $\Gamma$-quotient isomorphic to $A\rtimes G_0$, and hence by Lemma~\ref{lem:split-in-C} we have $A\rtimes G_0\in \overline{\calC}$, which contradicts to $G_1^{\calC}=G_0$.
	
	Similarly, we define $G_2, G_3, \cdots$ to be the $\Gamma$-quotients of $G_{\O}(k)$ inductively via 
	\[
		1\to A^{m(\varphi_i, \Gamma, G_i, A)} \to G_{i+1} \overset{\varpi_i}{\longrightarrow} G_i \to 1,
	\]
	where the map $\varphi_i$ is the quotient map $G_{\O}(k) \to G_i$. Using the argument in the previous paragraph, we see that each of these group extensions is completely nonsplit, and $G_i^{\calC}=G_0$ for each $i$. Then we take the inverse limit
	\[
		G:= \varprojlim_{i} G_i.
	\]
	Then the profinite group $G$ is the maximal extension of $G_0$ in $G_{\O}(k)$ that can be obtained via group extensions by $A$.
	
	\begin{lemma}\label{lem:G-prop}
		\begin{enumerate}
			\item \label{item:G-prop-0}	 A subset of $G$ is a generator set if and only if its image in $G_0$ generates $G_0$.
			\item \label{item:G-prop-1} The map $\pi: F_n(\Gamma) \to G_0$ defined at the beginning of this subsection factors through $G$.
			\item \label{item:G-prop-2} Let $\varphi$ be the natural quotient map $G_{\O}(k) \to G$ defined by inverse limit of $\varphi_i$. Then $\Hom_{G}\left((\ker \varphi)^{ab}, A\right)=0$.
		\end{enumerate}
	\end{lemma}
	
	\begin{proof}
		The group extension $\varpi_i: G_{i+1} \to G_i$ is completely nonsplit, so any lift of a generator set of $G_i$ is a generator set of $G_{i+1}$. So we have \eqref{item:G-prop-0} by taking inverse limit, and then \eqref{item:G-prop-1} follows. 
	
		Note that $G$ acts on the abelianization $(\ker\varphi)^{ab}$ of $\ker \varphi$ by conjugation. Suppose that $\Hom_G((\ker\varphi)^{ab}, A) \neq 0$. Then it means that $\varphi$ factors through a group extension $H$ of $G$ by a kernel $A$. However, $G$ does not have such a group extension in $G_{\O}(k)$ by definition. So we proved \eqref{item:G-prop-2}.
	\end{proof}

\subsection{Determination of the multiplicity of $A$}\label{ss:determine-mult}

We continue to use notation and assumptions given previously in this section. In particular, we remind the reader that $A$ is a fixed finite simple $G_{\O}(k)^{\calC}\rtimes \Gamma$-module where $\Gamma\simeq \Gal(k/Q)$, and $G$ is defined to be depending on $A$. 
	The goal of this subsection is to compute the multiplicity  of $A$ in an admissible $\Gamma$-presentation of $G_{\O,\infty}(k)^{\calC}$.
	The $\Gamma$-group $G$ plays a very important role in this computation.
	
	\begin{lemma}\label{lem:delta-upper}
		Let $\ell$ be the exponent of $A$ and assume that $\ell\neq \Char(Q)$ is prime to $|\Gamma|$. Then we have 
		\[
			\dim_{\F_{\ell}} H^2(G, A)^{\Gamma} - \dim_{\F_{\ell}} H^1(G, A)^{\Gamma} \leq \delta_{k/Q, \O}(A).
		\]
	\end{lemma}
	
	\begin{proof}
		We consider the $\Gamma$-equivariant short exact sequence
		\[
			1 \to M \to G_{\O}(k) \overset{\varphi}{\to} G \to 1.
		\]
		By the Hochschild-Serre exact sequence, we have
		\begin{equation}\label{eq:exs}
			0 \to H^1(G,A) \to H^1(G_{\O}(k), A) \to H^1(M, A)^{G} \to H^2(G,A) \to H^2(G_{\O}(k),A),
		\end{equation}
		which is compatible with the conjugation action by $\Gamma$. Since $M$ acts trivially on $A$, we see that $H^1(M,A)^G=\Hom_G(M^{ab}, A)=0$ by Lemma~\ref{lem:G-prop}\eqref{item:G-prop-2}. So by taking the $\Gamma$-invariants on \eqref{eq:exs} and computing the dimensions, we have that 
		\begin{eqnarray*}
			\dim_{\F_\ell} H^2(G, A)^{\Gamma} - \dim_{\F_\ell} H^1(G, A)^{\Gamma} &\leq& \dim_{\F_\ell} H^2(G_{\O}(k), A)^{\Gamma} - \dim_{\F_\ell} H^1(G_{\O}(k), A)^{\Gamma}\\
			&=& \delta_{k/Q, \O}(A).
		\end{eqnarray*}
	\end{proof}

	Starting from now, we assume that $\calC$ is a finite set of isomorphism classes of finite $\Gamma$-groups all of whose orders are prime to $|\Gamma|$, $\Char Q$ and $|\mu(Q)|$.
	Let $\widehat{\pi}$ denote the $\Gamma$-equivariant surjective map $F_n(\Gamma) \to G$ used in Lemma~\ref{lem:G-prop}\eqref{item:G-prop-1}. Then the pro-$\calC$ completion of $\widehat{\pi}$ is $\pi^{\calC}: {F}_n^{\calC} \twoheadrightarrow G_{\O}(k)^{\calC}$.
%	we have the following diagrams
%	\begin{equation}\label{eq:nf-diag}\begin{tikzcd}
%		F'_n \arrow["\widehat{\pi}", two heads]{r} \arrow[two heads]{d} \arrow["\pi", two heads]{dr} & G \arrow[two heads]{d} \\
%		{F'_n}^{\calC} \arrow["\pi^{\calC}", two heads]{r} & G_{\O}(k)^{\calC},
%	\end{tikzcd}\end{equation}
%	where the vertical maps are taking pro-$\calC$ completions. 
	If $Q=\Q$, then $G_{\O}(k)^{\calC}$ is exactly $G_{\O,\infty}(k)^{\calC}$. If $Q$ is a function field, then $k_{\O}/k$ is not split completely at primes over $\infty$. Instead, $G_{\O,\infty}(k)$ is the $\Gamma$-quotient of $G_{\O}(k)$ obtained via modulo the decomposition subgroup $\Gal_\frakp(k_{\O}/k)$ of one prime $\frakp$ of $k$ above $\infty$ (because $\Gamma$ acts transitively on all the primes of $k$ above $\infty$). Since this decomposition subgroup $\Gal_{\frakp}(k_{\O}/k)$ is isomorphic to $\widehat{\Z}$ and $G$ is a quotient of $G_{\O}(k)$, we can define $g_n$ to be an element of $G$ that is the image of one generator of $\Gal_{\frakp}(k_{\O}/k)$. In other words, denoting $G^{\#}$ the quotient $G$ modulo the $\Gamma$-closed normal subgroup generated by $g_n$, we have the following diagram
	\begin{equation}\label{eq:ff-diag}\begin{tikzcd}
		F_n \arrow["\widehat{\pi}", two heads]{r} \arrow[two heads]{d} \arrow["\pi", two heads]{dr} \arrow[bend left, two heads, "\varpi"]{rr} &G \arrow[two heads, "\eta","/{[g_n]}"']{r} \arrow[two heads]{d} & G^{\#} \arrow[two heads]{d} \\
		{F_n}^{\calC} \arrow["\pi^{\calC}", two heads]{r} \arrow[bend right, two heads, "\varpi^{\calC}", swap]{rr}  &G_{\O}(k)^{\calC} \arrow[two heads, "\eta^{\calC}"]{r} & G_{\O,\infty}(k)^{\calC},
	\end{tikzcd}\end{equation}
	where the vertical maps are taking pro-$\calC$ completions. To make the notation consistent between the number field and the function field cases, when $Q=\Q$, we let $g_n=1$, and hence $\eta$ and $\eta^{\calC}$ in \eqref{eq:ff-diag} are both identity maps.
	First of all, we want to determine $m(\widehat{\pi}, \Gamma, G, A)$.

	\begin{proposition}\label{prop:mult-hatpi}
		Let $\ell$ be the exponent of $A$. Assume $\ell\neq \Char(Q)$ is relatively prime to $|\mu(Q)| |\Gamma|$.
		If $Q=\Q$, then 
		\[
			m(\widehat{\pi}, \Gamma, G, A) \leq \dfrac{(n+1)\dim_{\F_{\ell}} A - \dim_{\F_\ell} A^{\Gamma}}{h_{G\rtimes \Gamma}(A)}.
		\] 
		If $Q=\F_q(t)$ and $A\neq \mu_{\ell}$, then 
		\[
			m(\widehat{\pi}, \Gamma, G, A) \leq \dfrac{n\dim_{\F_{\ell}} A - \dim_{\F_\ell} A^{\Gamma}}{h_{G\rtimes \Gamma}(A)}.
		\]
	\end{proposition}
	
	\begin{remark}
		Recall that in Theorem~\ref{thm:main} we assume that $k/Q$ is split completely at $\infty$. In the function field case, $\mu_{\ell}$ is a $\Gal(k_{\O}/Q)$-module but not a $\Gal(k_{\O,\infty}/Q)$-module, so we exclude the situation that $A=\mu_\ell$.
	\end{remark}
	
	\begin{proof}
		By the assumptions, we can apply Proposition~\ref{prop:d'-cohom} to compute the multiplicities. Because $\ell\nmid |\Gamma|$, we have for $i=1,2$ that $H^i(G\rtimes \Gamma, A) = H^i(G, A)^{\Gamma}$. Then by Lemma~\ref{lem:delta-upper}, we have
	\begin{equation}\label{eq:comp-m}
		m(\widehat{\pi}, \Gamma, G, A) \leq \frac{n\dim_{\F_\ell} A -\xi(A) + \delta_{k/Q,\O}(A)}{h_{G\rtimes \Gamma} (A)}.
	\end{equation}
	So we just need to compute $\delta_{k/Q, \O}(A)$.
	
	In the function field case, since $k/Q$ is split completely above $\infty$, the genus of $k$ is positive. By Proposition~\ref{prop:ff-delta}\eqref{item:pos-genus}, 
	\[
		\delta_{k/Q, \O}(A) = \dim_{\F_\ell} \Hom_{\Gal(k_{\O}/Q)}(A, \mu_{\ell}) - \dim_{\F_\ell} A^{\Gal(k_{\O}/Q)}.
	\]
	Recall that $A$ is a simple $\F_{\ell}[\Gal(k_{\O}/Q)]$-module that is not $\mu_{\ell}$, so we see that $\delta_{k/Q, \O}(A)$ is $-1$ if $A=\F_{\ell}$, and is 0 otherwise. So we proved the result in function field case.
	
	In the number field case that $Q=\Q$, we need to compute each terms in the formula in Proposition~\ref{prop:nf-delta}. Let $T=S_{\ell}(k) \cup S_{\infty}(k)$. In this case, $\ell$ is odd as $\mu_2 \subset Q$. First, we apply Theorem~\ref{thm:EPChar}
	\[
		\log_{\ell}\chi_{k/Q, T}(A)= -\dim_{\F_\ell} H^0(\Q_{\infty}, A') = -\dim_{\F_\ell}(A')^{\Gal(\C/\R)},
	\]
	where the first equality uses $\widehat{H}^0(\Q_\infty, A')=0$ because of $\#\calG_{\infty}(\Q)=2$ and \cite[Prop.~1.6.2(a)]{NSW}. Then because $A$ is a simple $\F_{\ell}[\Gal(k_{\O}/Q)]$-module, it is totally real and hence $(A')^{\Gal(\C/\R)}=\Hom_{\Gal(\C/\R)}(A, \mu_{\ell})=0$. So we have $\log_{\ell} \chi_{k/Q, T}(A)=0$. Then note that $\epsilon_{k/Q,\O}(A)$ in the formula in Proposition~\ref{prop:nf-delta} is $\dim_{\F_{\ell}} A$ in this case, and we obtain
	\[
		\delta_{k/Q, \O}(A) \leq \dim_{\F_\ell} \Hom_{\Gal(k_T/Q)}(A, \mu_{\ell}) - \dim_{\F_\ell} A^{\Gal(k_{\O}/Q)} + \dim_{\F_\ell} A,
	\]
	where the right-hand side is 0 if $A=\F_{\ell}$ and is $\dim_{\F_\ell} A$ otherwise. So we proved the number field case.
	\end{proof}

	\begin{lemma}\label{lem:mult-varpi}
		Use the assumptions in Proposition~\ref{prop:mult-hatpi}. Consider the function field case and the diagram \eqref{eq:ff-diag}. When $n$ is sufficiently large, we have
		\[
			m(\varpi, \Gamma, G^{\#}, A) \leq 			\dfrac{(n+1) \dim_{\F_\ell} A - \dim_{\F_\ell} A^{\Gamma}}{h_{G^{\#}\rtimes \Gamma}(A)} 
		\]
	\end{lemma}
	
	\begin{proof}
		Again, we use $x_1, \cdots, x_n$ to denote the generators of $F_n$. We can make $n$ large to assume $\widehat{\pi}(x_n)=g_n$ (recall that the multiplicity depends on $n$ but not on the choice of $\varpi$). Then we have a commutative diagram
		\[\begin{tikzcd}
			F_n \arrow[two heads, "\lambda", "/{[x_n]}"']{rr} \arrow[two heads, "\widehat{\pi}"]{d} \arrow[two heads, "\varpi"']{drr} && F_{n-1} \arrow[two heads, "\phi"]{d} \\
			G \arrow[two heads, "\eta", "/{[g_n]}"']{rr} && G^{\#},
		\end{tikzcd}\]
		where the top map are defined by modulo the $\Gamma$-closed normal subgroup generated by $x_n$. Note that the composition of the top and the right arrows satisfies the conditions in Lemma~\ref{lem:m-sum}\eqref{item:m-sum-2}, so we have
		\[
			m(\varpi, \Gamma, G^{\#}, A) = m(\lambda, \Gamma, F_{n-1}, A) + m(\phi, \Gamma, G^{\#}, A).
		\]
		By the statement and the computation of $H^i(F_n\rtimes \Gamma, A)$ in the proof of Lemma~\ref{lem:d-cohom}, we see that
		\[
			m(\lambda, \Gamma, F_{n-1}, A)=\frac{\dim_{\F_\ell} A}{h_{G^{\#}\rtimes \Gamma}(A)}.
		\]
		So it suffices to prove
		\begin{equation}\label{eq:wts-mult-varpi}
			m(\phi, \Gamma, G^{\#}, A) \leq m(\widehat{\pi}, \Gamma, G, A),
		\end{equation}
		which will immediately follow after we prove the following embedding
		\begin{eqnarray*}
		&&\left\{ U \,\mid\,   \text{max. proper $F_{n-1}\rtimes \Gamma$-normal subgroup of $\ker \phi$ s.t. $\ker \phi/U\simeq_{G^{\#}\rtimes \Gamma} A$}\right\} \\
		&\overset{\kappa}{\longhookrightarrow}& \left\{ V \,\mid\,   \text{max. proper $F_{n}\rtimes \Gamma$-normal subgroup of $\ker \widehat{\pi}$ s.t. $\ker \widehat{\pi}/V\simeq_{G\rtimes \Gamma} A$}\right\}
		\end{eqnarray*}
		mapping $U$ to $\lambda^{-1}(U) \cap \ker \widehat{\pi}$.
		
		Since $\ker \varpi = \ker \widehat{\pi} \ker \lambda$, for each $U$ in the first set, we have
		\[
			\faktor{\ker \widehat{\pi}}{\lambda^{-1}(U) \cap \ker \widehat{\pi}} =\faktor{\lambda^{-1}(U) \ker \widehat{\pi} }{\lambda^{-1}(U)}=\faktor{\ker \varpi}{\lambda^{-1}(U)} \simeq_{G^{\#}\rtimes \Gamma} A,
		\]
		so the map $\kappa$ is well-defined. Also, if $V=\kappa(U)$, then 
		\begin{equation}\label{eq:nie-1}
			\faktor{\ker \varpi}{V\ker \lambda} = \faktor{\ker \widehat{\pi} (V\ker\lambda)}{V\ker \lambda} = \faktor{\ker \widehat{\pi}}{\ker \widehat{\pi} \cap (V \ker \lambda)}.
		\end{equation}
		Since $V\subset \ker \widehat{\pi}$ and $\ker \widehat{\pi}/V$ is a simple module, the last quotient is either 1 or isomorphic to $A$. On the other hand, both of $V$ and $\ker \lambda$ are contained in $\lambda^{-1}(U)$, so is $V\ker \lambda$. Then \eqref{eq:nie-1} implies that $V\ker \lambda =  \lambda^{-1}(U)$. So we see that if $\kappa(U_1)=\kappa(U_2)=V$, then $\lambda^{-1}(U_1)=\lambda^{-1}(U_2)$ and hence $U_1=U_2$. So we conclude that $\kappa$ is injective.
	\end{proof}
	
	\begin{proposition}\label{prop:mult-G-C}
		Let $A$ be a finite simple $G_{\O,\infty}(k)^{\calC}\rtimes \Gamma$-module of exponent $\ell\neq \Char(k)$ relatively prime to $|\mu(Q)||\Gamma|$. When $n$ is sufficiently large, there exists an admissible $\Gamma$-presentation $\calF_n(\Gamma) \twoheadrightarrow G_{\O,\infty}(k)^{\calC}$, and 
		\[
			m_{\ad}^{\calC}(n, \Gamma, G_{\O,\infty}(k)^{\calC}, A)\leq m_{\ad}(n, \Gamma, G^{\#}, A)\leq\dfrac{(n+1) (\dim_{\F_\ell} A - \dim_{\F_\ell} A^{\Gamma})}{h_{G_{\O,\infty}(k)^{\calC} \rtimes \Gamma}(A)}.
		\]
	\end{proposition}
	
	\begin{remark}\label{rmk:property-E}
		The proposition shows that $m^{\calC}_{\ad}(n, \Gamma, G_{\O,\infty}(k)^{\calC}, \F_\ell)=0$. In other words, $G_{\O,\infty}(k)^{\calC}$ does not admit any nonsplit central group extension
		\[
			1 \to \F_\ell \to \widetilde{G} \rtimes \Gamma \to G_{\O,\infty}(k)^{\calC}\rtimes \Gamma \to 1,
		\]
		such that $\widetilde{G}$ is of level $\calC$. This is equivalent to the solvability (i.e. the existence of the dashed arrow) of the following embedding problem
		\[\begin{tikzcd}
			& & & G_{\O,\infty}(k)^{\calC}\rtimes \Gamma\arrow[two heads,"\alpha"]{d} \arrow[dashed]{dl} & \\
			1 \arrow{r} & \F_\ell \arrow{r} &\widetilde{H} \rtimes \Gamma \arrow{r} & H \rtimes \Gamma \arrow{r} &1
		\end{tikzcd}\]
		for any nonsplit central group extension in the lower row with $\widetilde{H}$ of level $\calC$, and for any surjection $\alpha$. In \cite{LWZB}, this solvability is called the \emph{Property E} of $G_{\O,\infty}(k)$ and is proven using the classical techniques of embedding problems. So Proposition~\ref{prop:mult-G-C} provides a new proof of the Property E by counting multiplicites.
	\end{remark}
	
	\begin{proof}
		By \cite[Prop.~2.2]{LWZB}, $G_{\O,\infty}(k)$ is an admissible $\Gamma$-group, so is $G^{\#}$, because $G^{\#}$ is a $\Gamma$-quotient of $G_{\O,\infty}(k)$. Since $G_{\O,\infty}(k)^{\calC}$ is finite, when $n$ is large, there exist elements $a_1, \cdots, a_n$ of $G_{\O,\infty}(k)^{\calC}$ such that $\{Y(a_i)\}_{i=1}^n$ forms a generator sets of $G_{\O,\infty}(k)^{\calC}$. Then we choose a preimage $b_i \in G^{\#}$ of $a_i$, and hence $\{Y(b_i)\}_{i=1}^n$ generates $G^{\#}$ by Lemma~\ref{lem:G-prop}\eqref{item:G-prop-0}. Recall that the multiplicity does not depend on the choice of presentation, so we assume $\varpi$ in \eqref{eq:ff-diag} maps $y_i \in F_n$ to $b_i \in G^{\#}$ for each $i=1, \cdots, n$. Then the restriction $\varpi|_{\calF_n}$ is an admissible $\Gamma$-presentation of $G^\#$. We have by Corollary~\ref{cor:ad-mult-compute} that
		\[
			m_{\ad}(n, \Gamma, G^{\#}, A) = m(n, \Gamma, G^{\#}, A) -\dfrac{n \dim_{\F_\ell}A^{\Gamma}}{h_{G^{\#}\rtimes \Gamma}(A)}.
		\]
		Then the desired result follows by Propositions~\ref{prop:C-mult-compute}, \ref{prop:mult-hatpi} and \ref{lem:mult-varpi}.
	\end{proof}

\subsection{Existence of the presentation \eqref{eq:main}}  \label{ss:presentation}

	Finally, we will prove that when $n$ is sufficiently large, there exists a set $X$ of $\calF_n^{\calC}$ containing $n+1$ elements for which the following isomorphism, which is \eqref{eq:main} in Theorem~\ref{thm:main}, holds
	\[
		G_{\O,\infty}(k)^{\calC} \simeq \faktor{\calF_n(\Gamma)^{\calC}}{[r^{-1}\gamma(r)]_{r\in X , \gamma \in \Gamma}}.
	\]
	
	In Proposition~\ref{prop:mult-G-C}, we showed that when $n$ is sufficiently large, there is an admissible $\Gamma$-presentation, denoted by
	\[
		1 \to N \to  \calF_n^{\calC} \overset{\varpi^{\calC}_{\ad}}{\longrightarrow} G_{\O,\infty}(k)^{\calC} \to 1.
	\]
        	Let $M$ be the intersection of all maximal proper $\calF^{\calC}_n \rtimes \Gamma$-normal subgroups of $N$, and define $R=N/M$ and $F=\calF_n^{\calC}/M$. Note that because $\calC$ is finite, we have that $\calF_n^{\calC}$ is finite \cite[Cor.~15.72]{Neumann67}. Then $R$ is a finite direct product $\prod_{i=1}^t A_i^{m_i}$ of finite irreducible $F\rtimes \Gamma$-groups $A_i$'s. Assume $A_i$ and $A_j$ are not isomorphic as $F\rtimes \Gamma$-groups if $i\neq j$. When a factor $A_i$ is abelian, its multiplicity $m_i$ is $m_{\ad}^{\calC}(n, \Gamma, G_{\O,\infty}(k)^{\calC},A_i)$ computed in Proposition~\ref{prop:mult-G-C}.
	
	Let $X$ be a subset of $\calF_n^{\calC}$. Then the closed $\calF_n^{\calC}\rtimes \Gamma$-normal subgroup generated by $\{r^{-1}\gamma(r)\}_{r\in X, \gamma\in \Gamma}$ is $N$ if and only if the closed $F\rtimes\Gamma$-normal subgroup generated by $\{\overline{r}^{-1} \gamma(\overline{r})\}_{\overline{r}\in \overline{X}, \gamma\in \Gamma}$  is $R$, where $\overline{X}$ and $\overline{r}$ are the images of $X$ and $r$ in $R$ respectively. Recall the properties of $\calF_n$ listed at the beginning of \S\ref{sect:pres-ad}. Because of the property \eqref{item:prop-calF-1}, in the definition of $Y$ in \eqref{item:prop-calF-3}, we can take the generator set $\{\gamma_1, \cdots, \gamma_d\}$ to be the whole group $\Gamma$, then 
	\[
		\{r^{-1}\gamma(r)\}_{r\in X, \gamma \in \Gamma}= Y(\{r\}_{r\in X}) \quad \text{and} \quad \{\overline{r}^{-1}\gamma(\overline{r})\}_{\overline{r}\in \overline{X}, \gamma \in \Gamma}= Y(\{\overline{r}\}_{\overline{r}\in \overline{X}}).
	\]
	By \cite[Prop.~4.3]{LWZB}, for a fixed integer $u$, the probability that the images under the map $Y$ of $n+u$ random elements of $R$ generate $R$ as an $F\rtimes \Gamma$-normal subgroup is
	\begin{eqnarray*}
		&& \Prob([Y(\{r_1, \cdots, r_{n+u}\})]_{F\rtimes \Gamma}=R)\\
		&=& \prod_{\substack{1\leq i \leq t \\ A_i \text{ abelian}}} \prod_{j=0}^{m_i-1} (1-h_{F\rtimes \Gamma}(A_i)^j |Y(A_i)|^{-n-u}) \prod_{\substack{1\leq i \leq t \\ A_i \text{ non-abelian}}} (1-|Y(A_i)|^{-n-u})^{m_i}.
	\end{eqnarray*}
	This product in the formula is a finite product.
	By \cite[Lem.~3.5]{LWZB}, we have $|Y(A_i)|=|A_i|/|A_i^{\Gamma}|$ for each $i$. Note that Lemma~\ref{lem:adm-Gamma-fix} shows that $|Y(A_i)| > 1$ when $A_i$ is non-abelian, so the product over non-abelian factors in the above formula is always positive. The term for an abelian factor $A_i$ is positive if and only if 
	\[
		m_i \leq \frac{(n+u)\log_{\ell}|Y(A_i)|}{h_{G_{\O,\infty}(k)^{\calC} \rtimes \Gamma}(A_i)} = \frac{(n+u)(\dim_{\F_\ell} A_i -\dim_{\F_\ell} A_i^{\Gamma})}{h_{G_{\O,\infty}(k)^{\calC} \rtimes \Gamma}(A_i)}.
	\]
	Therefore, by Proposition~\ref{prop:mult-G-C}, $R$ can be $F\rtimes \Gamma$-normally generated by the $Y$-values of $n+1$ elements, and hence we finish the proof.

\section{Exceptional cases}\label{sect:exceptional}

	We will discuss the cases that are not covered by the Liu--Wood--Zureick-Brown conjecture, using the techniques developed in this paper. In this section, the base field $Q$ can be any global field. If $Q$ is a number field, we denote $r_1$ and $r_2$ the numbers of the real embeddings and the complex embeddings of $Q$ respectively.
	
	Again, we let $\Gamma$ be a nontrivial finite group and $k/Q$ a Galois extension of global fields with $\Gal(k/Q) \simeq \Gamma$, such that $\Char(Q)$ and $|\Gamma|$ are relatively prime. We assume that $\ell$ is a prime integer that is not $\Char(Q)$ and is prime to $|\Gamma|$. 
	Recall that $G_{\O}(k)(\ell)$ denotes the pro-$\ell$ completion of $G_{\O}(k)$. So $G_{\O}(k)(\ell)$ is the Galois group of the maximal unramified pro-$\ell$ extension of $k$, which we will denote by $k_{\O}(\ell)/k$. Note that $G_{\O}(k)(\ell)$ is finitely generated, because $\dim_{\F_\ell}H^1(k_{\O}, \F_{\ell})$ is the minimal number of generators of $G_{\O}(k)(\ell)$ and is finite. So when $n$ is sufficiently large, there is a $\Gamma$-presentation $\pi: F'_n(\Gamma)\to G_{\O}(k)(\ell)$.
	Moreover, we assume, throughout this section, that the $\ell$-primary part of the class group of $Q$ is trivial. Then $G_{\O, \infty}(k)(\ell)$ is admissible by the proof of \cite[Prop.~2.2]{LWZB}, and hence we can assume that the presentation $\pi$ induces an admissible presentation, i.e. $\pi^{\ad}:=\pi|_{\calF_n}$ is surjective.
	
	In this section, we use the assumptions above and study the multiplicities from the presentation $\pi^{\ad}$ in the following two cases:
	\begin{enumerate}
		\item When $Q$ is a number field with $\mu_{\ell}\not\subset Q$, and $k/Q$ is not required to be split completely at $S_{\infty}(Q)$ (see Section~\ref{ss:other-sgn}).
		\item When $Q$ contains the $\ell$-roots of unity $\mu_\ell$ (see Section~\ref{ss:rootsof1}).
	\end{enumerate}
	 We will compare the multiplicities in these two cases with the multiplicities from Theorem~\ref{thm:main}, to see why the random group model used in the Liu--Wood--Zureick-Brown conjecture cannot be applied to these two exceptional cases.
	 
	 We point out that we study only $G_{\O}(k)(\ell)$ instead of $G_{\O}(k)^{\calC}$ for a general $\calC$, simply because we want to keep the computation easy in this section and there is no previous work discussing these two exceptional cases beyond the distribution of $\ell$-class tower groups. One can generalize the argument in this section to any finite set $\calC$. 

\subsection{Other signatures}\label{ss:other-sgn}

	Assume $Q$ is a number field with $\mu_{\ell}\not\subset Q$ (so $\ell$ is odd), and $k$ is a $\Gamma$-extension of $Q$. For each $v\in S_{\infty}(Q)$, we denote $\Gamma_v$ to be the decomposition subgroup at $v$ of the extension $k/Q$. 
	
	\begin{lemma}\label{lem:other-sgn-mult}	
	For a finite simple $\F_{\ell}[\Gal(k_{\O}(\ell)/Q)]$-module $A$, we have
		\begin{eqnarray*}
			&&m_{\ad}(n, \Gamma, G_{\O}(k)(\ell), A) \\
			&\leq& 
			\begin{cases}
				r_1 +r_2 -1 & \text{if }A=\F_\ell \\
				n + r_2 & \text{if } A=\mu_{\ell} \\
			\dfrac{(n+r_1+r_2)\dim_{\F_\ell}A - \sum_{v\in S_{\R}(Q)} \dim_{\F_\ell} A/A^{\Gamma_v}-(n+1)\dim_{\F_\ell} A^{\Gamma}}{h_{\Gal(k_{\O}(\ell)/Q)}(A)} & \text{otherwise.}	
			\end{cases}
		\end{eqnarray*}
	\end{lemma}
	
	\begin{proof}
		Let $T$ be $S_{\infty}(k) \cup S_{\ell}(k)$. Since $\ell$ is odd, $\widehat{H}^0(Q_{v}, A')=0$ for any $v\in S_{\infty}(Q)$, and hence we have 
		\begin{eqnarray*}
			\log_{\ell}(\chi_{k/Q, T}(A)) 			&=& - \sum_{v\in S_{\infty}(Q)}\dim_{\F_\ell} H^0(Q_v, A') \\
			&=& -\sum_{v\in S_{\C}(Q)}\dim_{\F_\ell} A -\sum_{v\in S_{\R}(Q)} \dim_{\F_\ell} A/A^{\Gamma_v}.
		\end{eqnarray*}
		The last equality is because:
		\begin{enumerate}
			\item If $v\in S_{\C}(Q)$, then $\calG_v(Q)=1$ acts trivially on both $\mu_{\ell}$ and $A$.
			\item If $v\in S_{\R}(Q)$, then $\calG_v(Q)\simeq \Z/2\Z$ acts on $\mu_{\ell}$ as taking inverse. Since the action of $\calG_v(Q)$ on $A$ factors through $\Gamma_v$, and $\Gamma_v$ acts on $A/A^{\Gamma_v}$ as taking inverse, we have $\dim_{\F_\ell}(A')^{\calG_v(Q)}=\dim_{\F_\ell}\Hom_{\calG_v(Q)}(A, \mu_{\ell})=\dim_{\F_\ell}\Hom_{\calG_v(Q)}(A/A^{\Gamma_v}, \mu_{\ell})=\dim_{\F_\ell} A/A^{\Gamma_v}$.
		\end{enumerate}	
		By Proposition~\ref{prop:nf-delta}, we have 
		\[
			\delta_{k/Q, S}(A) \leq \begin{cases}
				\epsilon_{k/Q,\O}(A)-r_2-1 & \text{if } A=\F_{\ell} \\
				\epsilon_{k/Q,\O}(A)-r_2-r_1+1 & \text{if }A=\mu_{\ell} \\
				\epsilon_{k/Q,\O}(A)-r_2\dim_{\F_\ell}A-\sum_{v\in S_{\R}(Q)} \dim_{\F_\ell} A/A^{\Gamma_v} & \text{otherwise,}
			\end{cases}
		\]
		where $A$ can be $\mu_{\ell}$ only if $\mu_{\ell}\subset k$.
		Note that by definition, $\epsilon_{k/Q, \O}(A)$ is equal to $[Q:\Q]\dim_{\F_\ell} A$.
		So the desired result following by Proposition~\ref{prop:d'-cohom},  Corollary~\ref{cor:ad-mult-compute}, and Proposition~\ref{prop:C-mult-compute}.
	\end{proof}
	
	\begin{corollary}\label{cor:imag-quad}
		Let $k/\Q$ be an imaginary quadratic field such that $k\neq \Q(\sqrt{-3})$, and $\gamma$ denote the nontrivial element of $\Gamma=\Gal(k/\Q)\simeq \Z/2\Z$. For an odd prime $\ell$, we have the following isomorphism of $\Gamma$-groups
		\begin{equation}\label{eq:imag-quad}
			G_{\O}(k)(\ell)\simeq \faktor{\calF_n(\Gamma)(\ell)}{[r^{-1}\gamma(r)]_{r\in X}}
		\end{equation}
		for sufficiently large positive integer $n$ and some set $X$ consisting of $n$ elements of $\calF_n(\Gamma)(\ell)$.
	\end{corollary}
	
	\begin{remark}\label{rmk:imag-quad}
		If we choose the $n$ elements of set $X$ randomly with respect to the Haar measure, then the quotient in \eqref{eq:imag-quad} gives a random group that defines a probability measure on all $n$-generated pro-$\ell$ admissible $\Gamma$-groups. By taking $n\to \infty$, there is a limit probability measure, which can be computed using formulas in \cite{LWZB}. The discussion in \cite[\S~7.2 and Thm.~7.5]{LWZB} shows that this limit probability measure agrees with the probability measure used in the Boston--Bush--Hajir Heuristics \cite{BBH-imaginary}.
	\end{remark}
	
	\begin{proof}
		When $Q=\Q$ and $k$ is imaginary quadratic, we have $r_1=1$, $r_2=0$, and $\Gamma_{\infty}=\Gamma$. Let $A$ be a finite simple $\F_{\ell}[\Gal(k_{\O}(\ell)/\Q)]$-module. Also, $\mu_{\ell}\not\subset k$ for any odd $\ell$ because $k\neq \Q(\sqrt{-3})$, so $A\neq \mu_{\ell}$.
		By Lemma~\ref{lem:other-sgn-mult}	, when $n$ is sufficiently large, we have
		\[
			m_{\ad}(n,\Gamma, G_{\O}(k)(\ell), A) \leq \begin{cases}
				0 & \text{if } A=\F_{\ell} \\
				\dfrac{n(\dim_{\F_\ell}A - \dim_{\F_\ell}A^{\Gamma})}{h_{\Gal(k_{\O}(\ell)/\Q)} A} &\text{otherwise.}
			\end{cases}
		\]
		Note that $\Gamma\simeq \Z/2\Z$ implies that the normal subgroup of $\calF_n(\Gamma)(\ell)\rtimes \Gamma$ generated by $Y(X)$ is exactly $[r^{-1}\gamma(r)]_{r\in X}$. Thus, the corollary follows by \cite[Prop.~4.3]{LWZB}.
	\end{proof}

\subsection{When $Q$ contains the $\ell$-th roots of unity} \label{ss:rootsof1}

	In this subsection, we assume $\mu_{\ell}\subset Q$. In this case, $\mu_{\ell}$ becomes the trivial $\Gal(k_{\O}/Q)$-module $\F_\ell$, which makes the multiplicities in a presentation of $G_{\O}(k)(\ell)$ significantly different from the previous cases.
	
	\begin{lemma}\label{lem:rootsof1}
		Assume $\mu_{\ell}\subset Q$. For a finite simple $\F_{\ell}[\Gal(k_{\O}(\ell)/Q)]$-module $A$, we have
		\begin{enumerate}
			\item If $Q$ is a function field and the genus of $k$ is not $0$, then $\delta_{k/Q, \O}(A)=0$.
			\item If $Q$ is a number field, then $\delta_{k/Q, \O}(A) \leq (r_1+r_2)\dim_{\F_\ell}A$.
		\end{enumerate}
	\end{lemma}
	
	\begin{proof}
		Because of the assumption $\mu_{\ell}\subset Q$, we have 
		\begin{equation}\label{eq:dim-A'}
			\dim_{\F_\ell}(A')^{\Gal(k_{\O}/Q)} = \dim_{\F_\ell} (A^{\vee})^{\Gal(k_{\O}/Q)}=\dim_{\F_\ell} A_{\Gal(k_{\O}/Q)}=\dim_{\F_\ell} A^{\Gal(k_{\O}/Q)}.
		\end{equation}
		Then the first statement follows directly by Proposition~\ref{prop:ff-delta}\eqref{item:pos-genus}.
		For the rest we assume that $Q$ is a number field and denote $T=S_{\ell}(k) \cup S_{\infty}(k)$. If $\ell$ is odd, then the assumption $\mu_{\ell}\subset Q$ implies that $Q$ is totally imaginary. Then we can easily see by Theorem~\ref{thm:EPChar} that $\log_{\ell}\chi_{k/Q, T}(A)=-r_2\dim_{\F_\ell} A$, and hence the statement for odd $\ell$ follows by Proposition~\ref{prop:nf-delta} and \eqref{eq:dim-A'}. If $\ell=2$, then we first want to compute for each $v\in S_\infty(Q)$
		\begin{equation}\label{eq:charv}
			\dim_{\F_\ell} \widehat{H}^0(Q_{v}, A')-\dim_{\F_\ell} H^0(Q_v, A').
		\end{equation}
		 For each $v\in S_{\C}(Q)$, we have $\calG_{v}(Q)=1$, and hence \eqref{eq:charv} equals to $-\dim_{\F_\ell} A$. For each $v\in S_{\R}(Q)$, the assumption $\ell \nmid |\Gamma|$ implies that $|\Gamma|$ is odd. So for each $\frakp \in S_v(k)$, $\frakp$ is real, and so is any prime of $k_{\O}(\ell)$ lying above $\frakp$.
		 Thus, $\calG_{\frakp}(k)$ acts trivially on $A$, so it also acts trivially on $A'$, which implies that $\widehat{H}^0(k_{\frakp}, A')=H^0(k_{\frakp}, A')$. Then \eqref{eq:charv} equals $0$, and we obtain the statement for $\ell=2$ by Proposition~\ref{prop:nf-delta} and \eqref{eq:dim-A'}.
	\end{proof}
	
	Then by the same arguments in \S\ref{sect:proof-main}, we obtain the following bounds for the multiplicity of $A$.
	
	\begin{corollary}\label{cor:mult-roots}
		Assume $\mu_\ell \subset Q$. When $k$ is a function filed, we assume that $Q=\F_q(t)$ for some prime power $q$ such that $\ell \mid q-1$ and $k/Q$ is split completely at $\infty$. Let $A$ be a finite simple $\F_{\ell}[\Gal(k_{\O}(\ell)/Q)]$-module. Then for a sufficiently large $n$, we have
		\begin{eqnarray*}
			&&m_{\ad}(n, \Gamma, G_{\O, \infty}(k)(\ell), A) \\
			&\leq& 
			\begin{cases}
				\dfrac{(n+1)\dim_{\F_\ell}A - \xi(A)- n\dim_{\F_\ell}A^{\Gamma}}{h_{G_{\O,\infty}(k)(\ell)\rtimes \Gamma}(A)}&\text{if $Q$ is a function field} \\
				\dfrac{(n+r_1+r_2) \dim_{\F_\ell}A -\xi(A)-n \dim_{\F_\ell}A^{\Gamma}}{h_{G_{\O,\infty}(k)(\ell)\rtimes \Gamma}(A)}&\text{if $Q$ is a number field}.
			\end{cases}
		\end{eqnarray*}
	\end{corollary}
	
	\begin{remark}\label{rmk:roots}
	\begin{enumerate}
	\item\label{item:roots-1} Readers can compare the corollary with Proposition~\ref{prop:mult-G-C}. When $A=\F_\ell$ and $Q$ is $\Q(\zeta_{\ell})$ or $\F_q(t)$ with $\ell\mid q-1$, one can check that the upper bound of the multiplicity is positive, which suggests the failure of the Property E of $G_{\O, \infty}(k)$. Therefore, the random group model used in the Liu--Wood--Zureick-Brown conjection is not expected to work in this exceptional case.
	\item\label{item:roots-2} If the upper bounds in Corollary~\ref{cor:mult-roots} are sharp, then it also suggests that we should not expect the coincidence of the distributions of $G_{\O, \infty}(k)(\ell)$ between the function field case and the number field case.
%	, since $r_2$ and $\epsilon_{k/Q, \O}(A)$ appear in the upper bound for number fields. 
	For example, when $Q=\Q, \ell=2$ or $Q=\Q(\zeta_3), \ell=3$, the upper bound in the corollary equals the one for  function fields. However, when $Q=\Q(\zeta_\ell)$ with $\ell > 3$, the upper bound is 
	\[
	 	\dfrac{(n+(\ell-1)/2) \dim_{\F_\ell}A -\xi(A)-n \dim_{\F_\ell}A^{\Gamma}}{h_{G_{\O,\infty}(k)(\ell)\rtimes \Gamma}(A)},
	\]
	which is strictly larger than the upper bound for function fields.
	\end{enumerate}
	\end{remark}



\begin{bibdiv}
\begin{biblist}

\bib{Achter2006}{article}{
      author={Achter, Jeffrey~D.},
       title={The distribution of class groups of function fields},
        date={2006},
        ISSN={0022-4049},
     journal={J. Pure Appl. Algebra},
      volume={204},
      number={2},
       pages={316\ndash 333},
         url={https://doi-org.proxy.lib.umich.edu/10.1016/j.jpaa.2005.04.003},
      review={\MR{2184814}},
}

\bib{Achter2008}{incollection}{
      author={Achter, Jeffrey~D.},
       title={Results of {C}ohen-{L}enstra type for quadratic function fields},
        date={2008},
   booktitle={Computational arithmetic geometry},
      series={Contemp. Math.},
      volume={463},
   publisher={Amer. Math. Soc., Providence, RI},
       pages={1\ndash 7},
         url={https://doi-org.proxy.lib.umich.edu/10.1090/conm/463/09041},
      review={\MR{2459984}},
}

\bib{Adam-Malle}{article}{
      author={Adam, Michael},
      author={Malle, Gunter},
       title={A class group heuristic based on the distribution of
  1-eigenspaces in matrix groups},
        date={2015},
        ISSN={0022-314X},
     journal={J. Number Theory},
      volume={149},
       pages={225\ndash 235},
         url={https://doi-org.proxy.lib.umich.edu/10.1016/j.jnt.2014.10.018},
      review={\MR{3296009}},
}

\bib{BBH-imaginary}{article}{
      author={Boston, Nigel},
      author={Bush, Michael~R.},
      author={Hajir, Farshid},
       title={Heuristics for {$p$}-class towers of imaginary quadratic fields},
        date={2017},
        ISSN={0025-5831},
     journal={Math. Ann.},
      volume={368},
      number={1-2},
       pages={633\ndash 669},
         url={https://doi-org.proxy.lib.umich.edu/10.1007/s00208-016-1449-3},
      review={\MR{3651585}},
}

\bib{BBH-real}{article}{
      author={Boston, Nigel},
      author={Bush, Michael~R.},
      author={Hajir, Farshid},
       title={Heuristics for $p$ -class towers of real quadratic fields},
        date={2019},
     journal={Journal of the Institute of Mathematics of Jussieu},
       pages={1–24},
}

\bib{Boston-Wood}{article}{
      author={Boston, Nigel},
      author={Wood, Melanie~Matchett},
       title={Non-abelian {C}ohen-{L}enstra heuristics over function fields},
        date={2017},
        ISSN={0010-437X},
     journal={Compos. Math.},
      volume={153},
      number={7},
       pages={1372\ndash 1390},
         url={https://doi-org.proxy.lib.umich.edu/10.1112/S0010437X17007102},
      review={\MR{3705261}},
}

\bib{Cohen-Lenstra}{incollection}{
      author={Cohen, H.},
      author={Lenstra, H.~W., Jr.},
       title={Heuristics on class groups of number fields},
        date={1984},
   booktitle={Number theory, {N}oordwijkerhout 1983 ({N}oordwijkerhout, 1983)},
      series={Lecture Notes in Math.},
      volume={1068},
   publisher={Springer, Berlin},
       pages={33\ndash 62},
         url={https://doi-org.proxy.lib.umich.edu/10.1007/BFb0099440},
      review={\MR{756082}},
}

\bib{Cohen-Martinet}{article}{
      author={Cohen, H.},
      author={Martinet, J.},
       title={Class groups of number fields: numerical heuristics},
        date={1987},
        ISSN={0025-5718},
     journal={Math. Comp.},
      volume={48},
      number={177},
       pages={123\ndash 137},
         url={https://doi-org.proxy.lib.umich.edu/10.2307/2007878},
      review={\MR{866103}},
}

\bib{EVW}{article}{
      author={Ellenberg, Jordan~S.},
      author={Venkatesh, Akshay},
      author={Westerland, Craig},
       title={Homological stability for {H}urwitz spaces and the
  {C}ohen-{L}enstra conjecture over function fields},
        date={2016},
        ISSN={0003-486X},
     journal={Ann. of Math. (2)},
      volume={183},
      number={3},
       pages={729\ndash 786},
         url={https://doi-org.proxy.lib.umich.edu/10.4007/annals.2016.183.3.1},
      review={\MR{3488737}},
}

\bib{Friedman-Washington}{incollection}{
      author={Friedman, Eduardo},
      author={Washington, Lawrence~C.},
       title={On the distribution of divisor class groups of curves over a
  finite field},
        date={1989},
   booktitle={Th\'{e}orie des nombres ({Q}uebec, {PQ}, 1987)},
   publisher={de Gruyter, Berlin},
       pages={227\ndash 239},
      review={\MR{1024565}},
}

\bib{Garton}{article}{
      author={Garton, Derek},
       title={Random matrices, the {C}ohen-{L}enstra heuristics, and roots of
  unity},
        date={2015},
        ISSN={1937-0652},
     journal={Algebra Number Theory},
      volume={9},
      number={1},
       pages={149\ndash 171},
         url={https://doi-org.proxy.lib.umich.edu/10.2140/ant.2015.9.149},
      review={\MR{3317763}},
}

\bib{Goss}{book}{
      author={Goss, David},
       title={Basic structures of function field arithmetic},
      series={Ergebnisse der Mathematik und ihrer Grenzgebiete (3) [Results in
  Mathematics and Related Areas (3)]},
   publisher={Springer-Verlag, Berlin},
        date={1996},
      volume={35},
        ISBN={3-540-61087-1},
         url={https://doi.org/10.1007/978-3-642-61480-4},
      review={\MR{1423131}},
}

\bib{Golod-Shafarevich}{article}{
      author={Golod, E.~S.},
      author={\v{S}afarevi\v{c}, I.~R.},
       title={On the class field tower},
        date={1964},
        ISSN={0373-2436},
     journal={Izv. Akad. Nauk SSSR Ser. Mat.},
      volume={28},
       pages={261\ndash 272},
      review={\MR{0161852}},
}

\bib{Koch}{book}{
      author={Koch, H.},
       title={Galoissche {T}heorie der {$p$}-{E}rweiterungen},
   publisher={Springer-Verlag, Berlin-New York; VEB Deutscher Verlag der
  Wissenschaften, Berlin},
        date={1970},
        note={Mit einem Geleitwort von I. R. \v{S}afarevi\v{c}},
      review={\MR{0291139}},
}

\bib{Katz-Sarnak1999}{book}{
      author={Katz, Nicholas~M.},
      author={Sarnak, Peter},
       title={Random matrices, {F}robenius eigenvalues, and monodromy},
      series={American Mathematical Society Colloquium Publications},
   publisher={American Mathematical Society, Providence, RI},
        date={1999},
      volume={45},
        ISBN={0-8218-1017-0},
      review={\MR{1659828}},
}

\bib{Liu-rou}{article}{
      author={Liu, Yuan},
       title={Non-abelian {Cohen--Lenstra} heuristics in the presence of roots of unity},
        date={2022},
        note={preprint, arXiv:2202.09471},
}

\bib{Liu-pgp}{article}{
      author={Liu, Yuan},
       title={{On the $p$-rank of class groups of $p$-extensions}},
        date={2022},
        note={preprint, arXiv:2202.09888},
}


\bib{Lubotzky}{article}{
      author={Lubotzky, Alexander},
       title={Pro-finite presentations},
        date={2001},
        ISSN={0021-8693},
     journal={J. Algebra},
      volume={242},
      number={2},
       pages={672\ndash 690},
         url={https://doi.org/10.1006/jabr.2001.8805},
      review={\MR{1848964}},
}

\bib{LW2017}{article}{
      author={Liu, Yuan},
      author={Wood, Melanie~Matchett},
       title={The free group on n generators modulo n + u random relations as n
  goes to infinity},
        date={2018},
     journal={Journal für die reine und angewandte Mathematik},
      number={0},
  url={https://www.degruyter.com/view/journals/crll/ahead-of-print/article-10.1515-crelle-2018-0025/article-10.1515-crelle-2018-0025.xml},
}

\bib{LWZB}{article}{
      author={Liu, Yuan},
      author={Wood, Melanie~Matchett},
      author={Zureick-Brown, David},
       title={A predicted distribution for {Galois} groups of maximal
  unramified extensions},
        date={2019},
        note={preprint, arXiv:1907.05002},
}

\bib{Malle2008}{article}{
      author={Malle, Gunter},
       title={Cohen-{L}enstra heuristic and roots of unity},
        date={2008},
        ISSN={0022-314X},
     journal={J. Number Theory},
      volume={128},
      number={10},
       pages={2823\ndash 2835},
         url={https://doi-org.proxy.lib.umich.edu/10.1016/j.jnt.2008.01.002},
      review={\MR{2441080}},
}

\bib{Malle2010}{article}{
      author={Malle, Gunter},
       title={On the distribution of class groups of number fields},
        date={2010},
        ISSN={1058-6458},
     journal={Experiment. Math.},
      volume={19},
      number={4},
       pages={465\ndash 474},
  url={https://doi-org.proxy.lib.umich.edu/10.1080/10586458.2010.10390636},
      review={\MR{2778658}},
}

\bib{Neumann67}{book}{
      author={Neumann, Hanna},
       title={Varieties of groups},
   publisher={Springer-Verlag New York, Inc., New York},
        date={1967},
      review={\MR{0215899}},
}

\bib{NSW}{book}{
      author={Neukirch, J\"{u}rgen},
      author={Schmidt, Alexander},
      author={Wingberg, Kay},
       title={Cohomology of number fields},
     edition={Second},
      series={Grundlehren der Mathematischen Wissenschaften [Fundamental
  Principles of Mathematical Sciences]},
   publisher={Springer-Verlag, Berlin},
        date={2008},
      volume={323},
        ISBN={978-3-540-37888-4},
  url={https://doi-org.ezproxy.library.wisc.edu/10.1007/978-3-540-37889-1},
      review={\MR{2392026}},
}

\bib{Shusterman}{article}{
      author={Shusterman, Mark},
       title={Balanced presentations for fundamental groups of curves over
  finite fields},
        date={2018},
        note={preprint, arXiv:1811.04192},
}

\bib{Ellenberg-Venkatesh-ICM}{inproceedings}{
      author={Venkatesh, Akshay},
      author={Ellenberg, Jordan~S.},
       title={Statistics of number fields and function fields},
        date={2010},
   booktitle={Proceedings of the {I}nternational {C}ongress of
  {M}athematicians. {V}olume {II}},
   publisher={Hindustan Book Agency, New Delhi},
       pages={383\ndash 402},
      review={\MR{2827801}},
}

\bib{Wood-nonab}{article}{
 AUTHOR = {Wood, Melanie Matchett},
     TITLE = {Nonabelian {C}ohen-{L}enstra moments},
      NOTE = {With an appendix by the author and Philip Matchett Wood},
   JOURNAL = {Duke Math. J.},
  FJOURNAL = {Duke Mathematical Journal},
    VOLUME = {168},
      YEAR = {2019},
    NUMBER = {3},
     PAGES = {377--427},
      ISSN = {0012-7094},
   MRCLASS = {11R29 (11R45)},
  MRNUMBER = {3909900},
MRREVIEWER = {Christopher Frei},
       DOI = {10.1215/00127094-2018-0037},
       URL = {https://doi-org.proxy2.library.illinois.edu/10.1215/00127094-2018-0037},
}

\bib{Wang-Wood19}{article}{
AUTHOR = {Wang, Weitong and Wood, Melanie Matchett},
     TITLE = {Moments and interpretations of the
              {C}ohen-{L}enstra-{M}artinet heuristics},
   JOURNAL = {Comment. Math. Helv.},
  FJOURNAL = {Commentarii Mathematici Helvetici. A Journal of the Swiss
              Mathematical Society},
    VOLUME = {96},
      YEAR = {2021},
    NUMBER = {2},
     PAGES = {339--387},
      ISSN = {0010-2571},
   MRCLASS = {11R29 (11N45 11R32 11R45 60B15)},
  MRNUMBER = {4277275},
MRREVIEWER = {Kevin H. Wilson},
       DOI = {10.4171/cmh/514},
       URL = {https://doi-org.proxy2.library.illinois.edu/10.4171/cmh/514},
}


\end{biblist}
\end{bibdiv}
\end{document}